\documentclass{article}
\usepackage[utf8]{inputenc}
\usepackage{hyperref}
\hypersetup{colorlinks,linkcolor=blue,citecolor=blue,urlcolor=purple}
\usepackage{float}
\usepackage{graphicx,wrapfig}

\usepackage[utf8]{inputenc}
\usepackage{amsthm,amsmath,amsfonts,amssymb,mathtools,color,graphicx,mathrsfs,dsfont,wasysym}
\usepackage{graphicx}
\usepackage[shortlabels]{enumitem}
\usepackage{todonotes}
\usepackage{physics}
\usepackage{setspace}
\usepackage{titling}
\usepackage{nicefrac}
\usepackage[normalem]{ulem}
\usepackage[sc]{mathpazo} 
\usepackage[scaled]{helvet} 
\usepackage{verbatim}
\usepackage{faktor}
\usepackage{caption,subcaption} \usepackage{tikz-cd}
\usepackage{tikz}
\usepackage{xcolor}
\usepackage{framed}
\usepackage{marginnote}
\usepackage[top=4cm, bottom=3.5cm, outer=4cm, inner=4cm, heightrounded, marginparwidth=2cm, marginparsep=.3cm]{geometry}
\newlength{\ulength}
\newcommand{\R}{\mathbb{R}}

\newcommand{\bE}{\mathbf{E}}

\newcommand{\union}{\mathop{\cup}}

\newcommand{\connint}{\mathop{\mathrlap{\mkern-0.mu\raisebox{5.5pt}{$\scriptstyle\nabla$}}\int}\limits}
\newcommand{\conninttext}{\mathop{\mathrlap{\mkern-1.5mu\raisebox{2.5pt}{$\scriptscriptstyle\nabla$}}\int}\limits}
\newcommand{\boldd}{\boldsymbol{d}}
\newcommand{\bolddnab}{\boldsymbol{d^\nabla}}

\newcommand{\connintidx}[1]{\mathop{\mathrlap{\mkern-10.mu\raisebox{5.5pt}{$\underset{\scriptscriptstyle#1}{\scriptstyle\nabla}$}}\int}\limits}

\definecolor{OliveGreen}{RGB}{0,110,40}

\newtheorem{theorem}{Theorem}[section]
\newtheorem{theoremanddefinition}[theorem]{Theorem and Definition}
\newtheorem{proposition}[theorem]{Proposition}
\newtheorem{corollary}[theorem]{Corollary}
\newtheorem{lemma}[theorem]{Lemma}
\newtheorem{definition}[theorem]{Definition}
\newtheorem{remark}[theorem]{Remark}

\usepackage[noend]{algpseudocode}
\usepackage{algorithm}
\usepackage{faktor}

\renewcommand{\abs}[1]{\left\vert#1\right\vert}
\newcommand{\KILLME}[1]{} 

\title{A Discrete Exterior Calculus of Bundle-valued Forms}
\author{\small Theo Braune\thanks{Email: theo.braune@polytechnique.edu}\\ \scriptsize LIX-Ecole Polytechnique, IP Paris, Inria \and \small Yiying Tong\thanks{Email: ytong@msu.edu}\\ \scriptsize Michigan State U. \and \small Fran\c{c}ois Gay-Balmaz\thanks{Email: francois.gb@ntu.edu.sg}\\ \scriptsize Nanyang Technological University \and \small Mathieu Desbrun\thanks{Email: mathieu.desbrun@inria.fr}\\ \scriptsize Inria, Ecole Polytechnique}

\begin{document}
	
	\maketitle
\begin{abstract}
	The discretization of Cartan's exterior calculus of differential forms has been fruitful in a variety of theoretical and practical endeavors: from computational electromagnetics to the development of Finite-Element Exterior Calculus, the development of structure-preserving numerical tools satisfying exact discrete equivalents to Stokes' theorem or the de Rham complex for the exterior derivative have found numerous applications in computational physics.
	However, there have been relatively few efforts to develop a more general discrete calculus, this time for differential forms with values in vector bundles over a combinatorial manifold equipped with a connection.
	In this work, we propose a discretization of the exterior covariant derivative of bundle-valued differential forms while pointing out the importance of choosing local frame fields for numerical evaluation. We demonstrate that our discrete operator mimics its continuous counterpart, satisfies the Bianchi identities on simplicial cells, and contrary to previous attempts at its discretization, ensures numerical convergence to its exact value with mesh refinement under mild assumptions.
\end{abstract}
\vspace{0.5em} 

\vspace{-0.1em} 

\section{Introduction}
Over the past few decades, structure-preserving finite-dimensional approximations of Cartan's exterior calculus such as Finite-Element Exterior Calculus (FEEC~\cite{FEEC}) or Discrete Exterior Calculus (DEC~\cite{disc_differential_forms_modelling}) have proven effective in fields as varied as computational physics and geometry processing~\cite{DEC-DGPcourse}. At their heart is a discretization of smooth, \emph{scalar-valued} differential forms as  cochains~\cite{Bossavit_CEM}: the natural pairing
of a cochain with a chain (which, itself, provides a discrete notion of domain) then becomes a discrete analog of the \emph{integration} of a continuous form over a domain.
Defining the \emph{discrete} exterior derivative as the adjoint of the boundary operator on chains thus leads to a structure-preserving discretization of the \emph{continuous} exterior derivative operator  in the sense that it preserves discrete counterparts to fundamental differential properties such as Stokes' theorem or the de Rham cohomology. 

However, to date, far less attention has been dedicated to the discretization of \emph{differential forms with values in a vector bundle} --- even if Cartan's original work made extensive use of such bundle-valued forms. One obvious difficulty is that bundle-valued forms do not seamlessly admit a simple discretization: indeed, the mere pointwise expression of a bundle-valued form requires a pointwise fiber, making it difficult to properly define a discrete counterpart to the exterior covariant derivative where the integration over a domain can only be achieved in a common fiber. An exception is the case of discrete connections on two-dimensional surfaces, which has garnered significant interest in vector field processing~\cite{Crane:2010:TCD,DeGoes:2020} to define parallel transport and connection curvature for instance; but this 2D case can be handled with connections that are seen as angle-valued forms (since 2D rotations are commutative) once a section of the frame bundle has been chosen, while the general case requires the use of discrete forms with values in the group of rotation matrices, bringing significant difficulties due in part by their non-commutativity.

Yet, the use of connection curvature is key in fields such as Yang-Mills theory and relativity in computational physics, since it provides an invariant, frame-independent measure. It is therefore natural to hope that the development of structure-preserving discretizations of the exterior calculus of bundle-valued forms can be beneficial for a wide range of numerical endeavors. This paper thus develops a formal, readily-computable discretization of connections, bundle-valued forms, and connection one-forms, along with a discrete exterior covariant derivative of bundle-valued forms that not only satisfies the well-known Bianchi identities in this discrete realm, but also converges to its smooth equivalent under mesh refinement --- thus fixing a recent attempt at defining such an operator on vector-valued forms\cite{Hirani_Bianchi}, based on synthetic differential geometry~\cite{Kock1997,kock_2006}, which fails to be numerically convergent as we will discuss.

	\section{Smooth Theory of Connections}
\label{sec:continuousWorld}

We begin by first reviewing the geometric concepts that we will discretize, highlighting the properties we wish to preserve in the discrete realm.

\subsection{Preliminaries}

In the study of an $n$-dimensional manifold $M$, we often encounter a vector bundle $\pi: E \rightarrow M$, and a local trivialization $\phi: E|_{U} \rightarrow U \times \mathbb{R} ^r$ with $U \subset M$ being an open subset. This provides us with a local frame field $\{ f_1,...,f_r\}$ on $U$ associated with the canonical basis $\{e_1,...,e_r\}$ of $ \mathbb{R} ^r$ given by $f_a(x)= \phi ^{-1} (x, e_a)$ for each point $x\in U$.
If we switch to another local trivialization $\widetilde \phi: \pi^{-1}(\widetilde U) = E|_{\widetilde U} \rightarrow \widetilde U \times \mathbb{R} ^r$ (where $\widetilde U$ is another open subset of $M$), the resulting local frame $\{ \widetilde f_a\}$ relates to the local frame $\{ f_a \}$ via a matrix of change of bases $A: U \cap \widetilde U \rightarrow GL(r)$ through
\begin{equation}\label{change_basis} 
	f_a = \widetilde f_b A_a^b .
\end{equation}
For a section $s \in \Gamma(E)$ of the vector bundle (i.e., an assignment for each point $p\in M$ of a vector $s(p)$ in the point's fiber $\pi^{-1}(p)$), which we can express locally as $s = f_a s^a  =  \widetilde f_a \widetilde s^{\,a}$, the components $s^a$ and $\widetilde s^{\,a}$ are, instead, related through
$$
\widetilde s^{\,a} = A^a_b  s^b.
$$

\paragraph{Connection.}
Given a vector bundle $ \pi : E\!\to\!M$, a connection $\nabla$ is a bilinear map $\nabla: \Gamma(TM) \!\times\! \Gamma(E) \rightarrow \Gamma(E)$ such that for $g \in C^\infty(M)$, $s \in \Gamma(E)$, and $X \in \Gamma(TM)$, we have
\begin{itemize}
	\item[(i)] $\nabla _X(gs)= g \nabla _X\; s+ d g (X)\; s$\qquad (i.e., it satisfies Leibniz Rule)
	\item[(ii)] $ \nabla _{(gX)}\; s = g \nabla _X\;s$.\qquad\qquad\qquad\qquad (i.e., it is tensorial in $X$)
\end{itemize}

\noindent In local frames, where $s=  f_a s^a$, we can write $\nabla s = f_a \otimes (d s^ a + \omega _b^a s^b ) $, i.e., $\nabla_X s = f_a (d s^ a (X) + \omega _b^a(X) s^b ) $, for all $X \in \Gamma (TM)$, where $\omega \!\in\! \Omega ^1(U, \mathfrak{gl}(r))$ is the \textit{local connection $1-$form} given by $\nabla f_a \!=\! f_b \omega _a^b$.
Using Eq.~\eqref{change_basis} we can express the change of the local connection $1-$form $\widetilde \omega$ for another local trivialization through the \emph{gauge transformation}: \begin{equation}
	\label{connection_trans}
	\omega = A^{-1}\widetilde \omega A+ A^{-1}  {\rm d} A \quad\Leftrightarrow\quad \widetilde \omega = A\omega A^{-1} -  {\rm d} A\; A^{-1}. 
\end{equation} 
This relationship between the local connection one-forms is central to the understanding of the behavior of the connection under changes of local trivializations.

A connection $\nabla$ on the vector bundle $\pi :E \rightarrow M$ naturally induces a connection $\nabla^{\rm End}$ on the endomorphism bundle ${\rm End}(E) \rightarrow M$, whose fiber at $p$ is the space ${\rm End}(E_p)= \operatorname{Hom} (E_p,E_p)$ of endomorphisms of $E_p$.  It is defined by enforcing the Leibniz rule on $L(x)$,
\begin{equation}\label{def_nabla_End} 
	\big( \nabla ^{\rm End}_X L \big)  ( s):= \nabla _X (L(s))- L( \nabla _Xs),
\end{equation} 
for all $L \in \Gamma ( \operatorname{End}(E))$.

\paragraph{Pullback connection.}
Given a smooth map $\Psi\colon M\to N$ between two manifolds and a vector bundle $ \pi : E \rightarrow N$ over $N$, we can define the \emph{pullback bundle} as $ \pi ': \Psi^\ast E\to M$, $(\Psi^\ast E)_p :=E_{\Psi(p)}$. Given a connection $\nabla$ on $E$, we want to equip the pullback bundle with a pullback connection $\Psi^\ast\nabla$ such that, for any given  section $s\in \Gamma E$, we have 
$$(\Psi^\ast\nabla)\Psi^\ast s = \Psi^\ast(\nabla s).$$
Such a connection exists and is called the \emph{pullback connection}. We refer the reader to \cite{milnor1974characteristic} for the proof of its existence and uniqueness.

\paragraph{Parallel transport.}
Parallel transport is a fundamental concept in differential geometry and plays a key role in understanding the geometry of manifolds. Given a curve $\gamma:[0,1] \rightarrow M$ and a vector $v_0 \in E_{\gamma(0)}$ in the fiber at $\gamma(0)$, parallel transport along $\gamma$ is the process of transporting $v_0$ to $\gamma(t)$ while maintaining its direction in the vector bundle. The resulting vector at $\gamma(t)$ is denoted by $v(t)$ and is a solution to the differential equation
\begin{equation}\label{PTM}
	D_tv(t)=0 \quad\text{and}\quad v(0)=v_0,
\end{equation}
where $D_tv$ is the covariant time derivative of curves $v(t) \in E$ associated to a connection $\nabla$ on $E$, given locally by
\begin{equation}\label{local_Dt}
	D_t v= f_a (\partial _t v^a + \omega ^a_b (\dot \gamma ) v^b).
\end{equation}
Here, $v^a$ are the components of $v$ in a local frame ${f_a}$ for the fiber, and $\omega ^a_b$ are the connection coefficients associated with $\nabla$ in $f_a$.

The unique solution to the initial value problem \eqref{PTM} along $\gamma$ yields a one-parameter family of parallel transport maps
\[ \mathcal{R}_{ \gamma ,t} \colon E_{\gamma(t)} \to E_{\gamma(0)}, \]
which is an intrinsic object. We will simply write $ \mathcal{R} _t$ the parallel transport maps, when the choice of the curve $ \gamma $ is obvious from the context.
If we consider a local frame field $f_a(x)$ in a neighborhood of the curve, the linear map can be represented by a matrix $(R_t)^a_b$, via $ \mathcal{R} _t f_a ( \gamma (t))= f_b( \gamma (0))(R_t)_a^b $.
Moreover,  according to Eq.~\eqref{connection_trans}, the matrix $R_t$ is the solution to $\frac{d}{dt}  R_t = R_t\omega_ { \gamma(t)}( \dot  \gamma (t))$.

The solution to this last equation can be written as the path-ordered matrix exponential $P\left(\operatorname{exp}\int_0^t  \omega_ { \gamma( \tau )}( \dot  \gamma ( \tau )) d \tau \right) $. As $\omega$ is applied to the right of $R$, we order the terms with larger $\tau$ to the right. For instance, the second-order term in the power series of the matrix exponential is
$$P\; \frac12 \left( \int_{0}^t \omega_{\gamma(\tau)}(\dot{\gamma}(\tau))\ d\tau\right)^2=
\int_{\tau_2=0}^t \int_{\tau_1=0}^{\tau_2} \omega_{\gamma(\tau_1)}(\dot{\gamma}(\tau_1))\ \omega_{\gamma(\tau_2)}(\dot{\gamma}(\tau_2))\ d\tau_1 d\tau_2.$$

By utilizing the parallel transport matrices associated with \( (f_a) \) along \( \gamma \) as a gauge field, we can derive a new frame field
\[ \widetilde{f}_a(\gamma(t)) = f_b(\gamma(t)) (R_t^{-1})^b_a.
\]
Consequently, in this new frame, \( \Tilde{\omega}(\dot{\gamma}) = 0 \) by construction, and a constant \( \widetilde{v}\in \mathbb{R}^r \) along the curve satisfies the initial value problem \eqref{PTM}, i.e., \( \widetilde{v} = R_t v(t) = v_0 \). Alternatively, in the original frame \( {f_a} \), the parallel transport is given by \( v(t) = R_t^{-1} v_0 \).

If $X \!=\! \dot{\gamma}(0)\in T_{\gamma(0)}M$ and $\alpha\in \Gamma (E)$ is a section of $E$, then the parallel transport map $ \mathcal{R} _t$ can be used to define the covariant derivative of $\alpha$ with respect to $X$ at $\gamma(0)$ as
\begin{equation}\label{eq:par-trans-conn}
	(\nabla_X \alpha)_{\gamma(0)} = \lim_{t\to 0} \frac{ \mathcal{R} _t\alpha_{\gamma(t)} - \alpha_{\gamma(0)}}{t} = \frac{d}{dt} \bigg\vert_{t=0} \mathcal{R} _t\alpha_{\gamma(t)}.
\end{equation}

\paragraph{Integration of bundle-valued forms.}
Once we are given parallel transport maps, we can define a \emph{connection-dependent integration of a bundle valued $1-$form} quite naturally. 
\begin{definition}
	\label{def:bundle-valued-integral}
	Let $ \pi : E\to M$ be a vector bundle with connection $\nabla$ and let $\alpha\in \Omega^1(M;E)$ be a bundle-valued $1-$form. For a curve $\gamma\colon [0,1]\to M$, we define:
	$$\connint_\gamma \alpha = \connint_{[0,1]} \gamma^\ast\alpha =\int_{0}^{1}\mathcal{R}_{ \gamma ,t}\alpha _{ \gamma (t)}\left( \dot  \gamma (t)\right)dt \in E_{ \gamma (0)}, $$
	$$\int_\gamma \alpha = \connint_{[0,1]} \gamma^\ast\alpha =\int_{0}^{1}\mathcal{R}_{ \gamma ,t}\alpha _{ \gamma (t)}\left( \dot  \gamma (t)\right)dt \in E_{ \gamma (0)}, $$
	where $\mathcal{R}_{ \gamma ,t}\colon E_{\gamma(t)}\to E_{\gamma(0)}$ denotes the parallel transport along $\gamma$ for a fixed $t$.  We note that the integral gives a vector in the fiber attached to the initial point of the curve $ \gamma $.
\end{definition}

We can further extend this notion of connection-dependent integral $\conninttext$
 to an arbitrary $k$-form over a retractable region $S\subset M$: the existence of a connection allows us to define a proper integration through parallel transport to bring the final value to a same ``evaluation'' vector space, as defined below. 

\begin{definition}
	\label{def:bundle-valued-integral-higher-order}
	Let $ \pi :E\!\to\! M$ be a vector bundle with connection $\nabla$ and let $\alpha\!\in\! \Omega^k(M;E)$ be a bundle valued $k$-form. For a region $S\!\subset\! M,$ with $S$ homeomorphic to a closed $k$-dimensional ball,
	let $\varphi\colon S\!\to\! B_k(0)\subset \R^k$ be a homeomorphism from $S$ to the unit $k$-dimensional ball centered at 0. 
	Furthermore, let $\Tilde{\gamma}_{\varphi(v),\varphi(p)}\colon [0,1]\to \varphi(S)\!=\!B_k(0)$ be the straight joining path from $\varphi(p)$ to a chosen point $\varphi(v)$ and let $\gamma_{v,p}:=\varphi^{-1}(\Tilde{\gamma}_{\varphi(v),\varphi(p)})\colon [0,1]\!\to\! S$ be the induced curve in $S$.
	For $p\!\in\! S$ let 
	\begin{equation}\label{local_R}
		\mathcal{R}^{\nabla,\varphi_v}_p\in \mathrm{Hom}(E_p,E_v)
	\end{equation}
	be the parallel transport along $\gamma_{v,p}$. We denote by $\mathcal{R}^{\varphi_v}$ the induced transport field, i.e., a function mapping a point $p$ to a linear map between $E_p$ and $E_v$. We define the parameterization-induced connection-dependent integral as 
	$$\connintidx{\varphi_v}_S\alpha = \int_{S} {\mathcal{R}^{\nabla,\varphi_v}\alpha} \, \in \! E_v.$$
	If the choice of the parameterization is clear from context, we will suppress the parameterization $\varphi$ in the notation for the integral and denote the parallel-transport field as $\mathcal{R}^{\nabla,v}$.

\end{definition}

The connection-dependent integral can more generally be defined through a smooth strong deformation retraction $\varphi_v\colon [0,1]\!\times\! S\to S$ onto $v\!\in\! S,$ with the path $\gamma$ defined through $\gamma_p(t)\!=\!\varphi_v(1-t,p)$, since $ \varphi _v$ satisfies $ \varphi _v(0,p)=p$, $ \varphi _v(1,p)\!=\!v$, and $ \varphi _v(t,v)\!=\!v$.

\subsection{Operators and their properties}
\label{sec:operators-smooth-world}
We end this section with a review of the basic operators (and their properties) for which we will provide discrete equivalents in the remainder of this paper.

\paragraph{Wedge product.}

For scalar-valued differential forms on a smooth manifold $M$ there is a well-defined product on the space of differential forms: given $\alpha\!\in\! \Omega^k(M)$ and $\beta\!\in\! \Omega^\ell(M)$, the $(k\!+\!\ell)-$form $\alpha\wedge\beta\!\in\!\Omega^{k+\ell}(M)$ is the antisymmetrization of the tensor product of the two forms, i.e., for all $X_1,\ldots,X_{k+\ell}\in \Gamma(TM)$,
\begin{align*}
	\alpha\wedge \beta(X_1,\ldots,X_{k+\ell}) &= \frac{(k+\ell)!}{k!\ell !}\mathrm{Alt}(\alpha\otimes\beta)(X_0,\ldots,X_{k+\ell})\\
	&= \frac{1}{k!\ell!}\sum_{\sigma\in S_{k+\ell}}\mathrm{sgn}(\sigma)\alpha(X_{\sigma(0)},\ldots,X_{\sigma(k)})  \beta(X_{\sigma(k+1)},\ldots,x_{\sigma(k+\ell+1)}),
\end{align*}
where $S_{k+\ell}$ denotes the permutation of $(k\!+\!\ell)$ indices.

This notion can be similarly extended to forms that take values in vector bundles: the wedge product $\wedge$ between bundle-valued forms and scalar-valued forms can be defined as:
$$\wedge\colon \Omega^k(M,E)\times \Omega^\ell (M)\to \Omega^{k+\ell}(M,E)\colon (\alpha,\beta)\mapsto \alpha\wedge\beta = \frac{(k+\ell)!}{k!\ell !} \mathrm{Alt}(\alpha\otimes\beta), $$
where the tensor product applied to $(k\!+\!\ell)$ vectors $X_1,\ldots,X_{k+\ell}\!\in\! \Gamma (TM)$ is 
$$\alpha\otimes\beta(X_1,\ldots,X_{k+\ell}) = \alpha(X_1,\ldots,X_k) \beta(X_{k+1},\ldots,X_{k+\ell}).$$

\paragraph{Exterior covariant derivative.}
The exterior covariant derivative is a natural extension of the standard exterior derivative of scalar-valued forms to forms that take values in a vector bundle. In particular, given a connection $\nabla$ on $\pi:E\!\rightarrow\! M$, the exterior covariant derivative $d^\nabla: \Omega^k(M,E) \rightarrow \Omega^{k+1}(M,E)$ obeys
\begin{equation}\label{eq:ext-cov-derivative}
	d^\nabla (\alpha \otimes s) = d \alpha \otimes s + (-1)^k \alpha \wedge \nabla s, \quad \forall \alpha \in \Omega^k(M), \forall s\in \Gamma(E),
\end{equation}
where $\otimes$ denotes the tensor product of differential forms and sections of the bundle $E$.
In local coordinates, Eq.~\eqref{eq:ext-cov-derivative} is expressed as
\begin{equation}\label{eq:ext-cov-derivative-local}
	(d^\nabla \gamma)^a = d\gamma^a + \omega^a_b \wedge \gamma^b,
\end{equation}
where $\gamma \!=\! f_a \gamma^a  \in \Omega^k(U,E)$ in a given local frame field $\{f_a\}$ of $E$, while the associated local connection $1-$form for this frame field has components $\omega^a_b$.

Note that given a smooth map $\Psi\colon M\to N$ between two manifolds, and a vector bundle $ \pi : E \rightarrow N$ with connection $\nabla$, the exterior covariant derivative commutes with the pullback: given a form $\alpha\in \Omega^k(N,E)$ one has
\begin{equation}
	\label{eq:pullback-commutes-with-d-nabla-smooth}
	d^{\Psi^{\ast}\nabla}(\Psi^\ast\alpha) = \Psi^\ast(d^\nabla\alpha).
\end{equation}
Additionally, given a bundle-valued form $\alpha\!\in\! \Omega^k(M,E)$ and a scalar-valued form $\lambda\!\in\! \Omega^\ell(M),$ the \emph{Leibniz rule} between the wedge product and the exterior covariant derivative holds, i.e.,
$$d^\nabla(\alpha\wedge\lambda) = {\rm d}^\nabla\alpha\wedge\lambda + (-1)^k \alpha\wedge d\lambda.$$
For a proof of these last two properties, see~\cite{marsh2018mathematics} for instance.

\paragraph{Curvature two-form and Bianchi identities.}

Connections are geometric structures that capture the idea of parallel transport along curves, and are an essential tool in the study of fiber bundles. One important geometric invariant of connections is the \emph{curvature $2-$form} $\Omega^\nabla\!\!\in\!\Omega^2(M,\mathrm{End}(E))$, which measures the failure of a connection to be flat, i.e., the impossibility of endowing the base manifold with a Euclidean structure. This skew-symmetric tensor is useful in geometry and physics in areas such as general relativity, gauge theory, and topological quantum field theory. It quantifies the noncommutativity of parallel transport around closed loops and encodes the infinitesimal holonomy of the connection, also known as Ambrose-Singer theorem \cite{amrose_singer}.
The \textit{curvature $2-$form} is defined as 
\[
\Omega ^ \nabla (X,Y)s:= d^ \nabla ( d^ \nabla s)(X,Y), \quad s \in \Gamma (E), \; X,Y \in \Gamma (T M).
\]

If we are given a local frame field, it can be shown that we have the local formula
\begin{equation}
	\label{eq:Omega-local}
	\Omega ^a_b= d \omega ^a_b + \omega ^a_c \wedge \omega ^c_b = (d\omega + \omega\wedge\omega)^a_{b}.
\end{equation}
This local formula is also sometimes written abusively as $\Omega^ \nabla\!=\!d^ \nabla \omega$, to be understood as applying the exterior covariant derivative for vector-valued forms to the columns of the connection $1-$form $\omega$ in the local frame. 
One can then show that the following general equation holds:
\begin{equation}
	\label{eq:algBianchi}
	d ^ \nabla d ^ \nabla \alpha = \Omega^\nabla( \cdot , \cdot ) \wedge \alpha ,\quad \forall  \alpha \in \Omega ^k (M, E),
\end{equation}
a property proving that unlike the exterior derivative ${\rm d}$ of scalar-valued differential forms, the operator ${\rm d}^\nabla$ is \emph{not} nilpotent in general. Note that this property is often referred to as the \emph{algebraic Bianchi identity}.

Another important property is what is often referred to as the \emph{differential Bianchi identity}, a consequence of the Leibniz rule and antisymmetry of the exterior covariant derivative: it states that 
\begin{equation}
	d^{\nabla^{\rm End}} \Omega^\nabla=0,
	\label{eq:diffBianchi}
\end{equation}
where $\nabla^{\rm End}$ is the induced connection on the endomorphism bundle $\operatorname{End}(E) \rightarrow M$.
More generally it can be shown that for the induced connection on the endomorphism bundle, Eq. ~\eqref{eq:algBianchi} can be expressed as 
\begin{equation}
	\label{eq:algBianchi-endom}
	d^{\nabla^{\rm End}} d^{\nabla^{\rm End}} \beta = [\Omega^\nabla \wedge \beta]  ,\quad \forall  \beta \in \Omega ^k (M, \mathrm{End}(E)),
\end{equation}
where $[\cdot\wedge\cdot]$ is the commutator using the wedge product.

Another usual identity linking curvature tensor and covariant derivatives is that, given a connection $\nabla$ on a vector bundle $\pi: E\!\rightarrow\! M$ and for any 0-form $s\!\in\! \Gamma(E)$, one has:
\begin{equation}
	\Omega^\nabla(X,Y)s = \nabla_X\nabla_Y s - \nabla_Y\nabla_X s - \nabla_{[X,Y]} s, \quad \forall X,Y\in \Gamma(TM).
	\label{eq:curv_tensor}
\end{equation}

The curvature $2-$form and its two associated Bianchi identities are central geometric notions in the study of connections as they are independent of the choice of local trivializations of the associated vector bundle $E$. Consequently, we will ensure that our discretization respects these important properties.

\begin{remark}
	\label{rem:curvature_gauge_invariant}
	The curvature is \emph{gauge invariant}. This means that given two local frame fields $\{f_a\}$ and $\{\tilde f_a\}$, related via a matrix of change of bases as $f_a= \tilde f_b A^b_a$ (see Eq.~\eqref{change_basis}) with associated local connection $1-$form $\omega$ and $\tilde{ \omega }$, then the components of the matrix-valued curvature in both frame are related as \[
	\widetilde{\Omega}^\nabla =d\tilde{\omega} + \tilde{\omega}\wedge\tilde{\omega} = A (d\omega + \omega\wedge\omega) A^{-1} = A \Omega ^ \nabla A ^{-1}.
	\]
	That is, a change of basis in the fiber only amounts to a change of basis for the representation matrix of the curvature endomorphism, see \cite[Eq. (1.29), page 108]{chern1999lectures}.
\end{remark}

\paragraph{Riemannian connections.} 
Given a smooth manifold $M$, we call $M$ a \emph{Riemannian manifold} if the tangent bundle possesses a Riemannian metric, i.e., a section $g\!\in\! \Gamma(\smash{\mathrm{Sym}^2(TM)})$ such that 
$g_x:T_xM \times T_xM \rightarrow \mathbb{R}$ is a scalar product for each fiber space.
A connection $\nabla$ on the tangent bundle is said to be compatible with the metric if it preserves the metric structure, which is a natural condition in Riemannian geometry. In other words, $\nabla$ is compatible with the metric if the covariant derivative of the metric with respect to a vector field $X$ can be expressed in terms of the metric and $\nabla$ as

\[ d\big(g(s_1,s_2)\big)(X) = g(\nabla_X s_1,s_2) + g(s_1,\nabla_X s_2) \quad \forall s_1,s_2 \in \Gamma (TM), X \in \Gamma (TM).\]

If the connection is compatible with the metric, for any two vectors $v,w \in T_xM$, the inner product $g(R_t v, R_t w)$ is preserved along parallel transport, where $R_t$ denotes the parallel transport map associated with the connection $\nabla$. In a local orthonormal frame, this implies that the local expressions of $R_t$ belong to the special orthogonal group $SO(n)$. This compatibility of the connection with the metric structure is a fundamental requirement in Riemannian geometry: it plays a major role in the definition of the Levi-Civita connection, which is the unique connection that is both compatible with the metric \emph{and} torsion-free --- see the next paragraph for a definition of torsion.

\paragraph{Solder form and torsion.}

For the special case where $E\!=\! TM,$ the tautological $1-$form $\theta$, or \emph{solder form}, is a bundle-valued $1-$form defined as $\theta(X) \!=\! X$ for any vector field $X$ on $M$. For a connection $\nabla$, the \emph{torsion $2-$form} is derived from the solder form through 
\begin{equation}
	\Theta^\nabla = d^\nabla\theta = d\theta + \omega\wedge\theta\ \in \Omega^2(M,TM),
	\label{eq:smooth_torsion_exterior_derivative}
\end{equation}
where the last equality holds only locally.
While the curvature $2-$form $\Omega^\nabla$ measures the deviation of the covariant derivative from the exterior derivative, the torsion $2-$form $\Theta$, on the other hand,  measures the deviation of the covariant derivative from the Lie derivative, meaning that it measures the failure of the connection to preserve the Lie bracket of vector fields under parallel transport. Specifically, given a connection $\nabla$, the torsion $2-$form $\Theta$ is also expressed through
$$\Theta^\nabla(X,Y) = d^\nabla\theta(X,Y) = \nabla_X Y - \nabla_Y X - [X,Y]$$
for vector fields $X$ and $Y$ on $M$.

The \emph{algebraic Bianchi identity} often refers to the differential identity from Eq.~\eqref{eq:algBianchi} specifically applied to the solder form; it relates the exterior derivative of the torsion $2-$form to the curvature $2-$form and the solder form through
\begin{equation}
	\label{eq:algebraic-bianchi-identity}
	d^\nabla\Theta^\nabla = d^\nabla d^\nabla\theta = \Omega^\nabla\wedge\theta.
\end{equation}

\begin{remark}[Bianchi Identities in (Pseudo) Riemannian Geometry]
	Suppose we consider the special case of a Riemannian manifold and $E \!=\! TM$. Let $\nabla$ be the Levi-Civita connection, i.e., the unique torsion free connection compatible with the metric. In this case we usually write $ \Omega ^ \nabla (X,Y)Z= R(X,Y)Z$, with $R$ the Riemann curvature $(1,3)$-tensor, i.e., $R\!\in\! TM^\ast\!\otimes\! TM^\ast\!\otimes\! TM^\ast\!\otimes\! TM$. We have
	\[
	{\rm d}^{\nabla^{\rm End}} \Omega^\nabla=0 \;\Leftrightarrow\; \nabla _XR(Y,Z)+ \nabla _YR(Z,X)+ \nabla _ZR(X,Y)=0
	\]
	and
	\[
	0 = d^\nabla d^\nabla\theta = \Omega^\nabla\wedge\theta  \;\Leftrightarrow \;R(X,Y)Z+R(Y,Z)X+R(Z,X)Y=0,
	\]
	
	which are the familiar Bianchi identities as formuluated in (pseudo)-Riemannian geometry. A common form found in literature uses tensor notation. Let $(\partial_1,\ldots,\partial_n)$ be a local orthonormal frame field of $TM$ consisting of vector fields for which the Lie bracket vanishes. Let $(dx^1,\ldots,dx^n)$ the associated dual frame. Using the Einstein sum convention the coordinates of the Riemannian curvature tensor can be expressed as the components of the curvature $2-$form as
	$$[\,\Omega^\nabla]^{\mu}_{\nu} = R^{\mu}_{\nu,\alpha,\beta} dx^\alpha\wedge dx^\beta,$$
	where 
	$$ R^{\mu}_{\nu,\alpha,\beta}= dx^\mu\left( R( \partial _ \alpha , \partial _ \beta ) \partial _\nu \right).$$
	The algebraic Bianchi identity then reads
	$$0 =R^\mu_{\nu \alpha\beta}+R^{\mu}_{\alpha\beta\nu}+ R^\mu_{\beta\nu\alpha}.$$
	As for the endomorphism-valued case, the definition for the covariant derivative of a tensor field this time becomes: 
	$$R^\mu_{\nu\alpha\beta;\lambda} = {\partial_{\lambda} R^{\mu}_{\nu\alpha\beta}} +  R^{\kappa}_{\nu\alpha\beta}\omega^{\mu}_{\lambda\kappa} - R^{\mu}_{\kappa\alpha\beta}\omega^{\kappa}_{\lambda\nu} -R^{\mu}_{\nu\kappa\beta}\omega^{\kappa}_{\lambda\alpha}-R^{\mu}_{\nu\alpha\kappa}\omega^{\kappa}_{\lambda\beta}.$$
	Using this definition of covariant tensor derivative, the differential Bianchi identity for the Levi-Civita connection reads
	$$0 = R^{\mu}_{\nu\alpha\beta;\lambda}+ R^\mu_{\nu\lambda\alpha;\beta}+R^\mu_{\nu\beta\lambda;\alpha}. $$
	For a detailed derivation of these formulas from Eqs.~\eqref{eq:diffBianchi} and \eqref{eq:algebraic-bianchi-identity}, see~\cite{frankel_2011} p.298-301.

\end{remark}

	\section{Discrete Connections and Bundle-valued Differential Forms}

We now discuss our discretization of bundle-valued forms and connections, for which related operators will be derived in the next section. Throughout this section and for the remainder of this paper, we consider an arbitrary discrete manifold $M$ represented by a \emph{simplicial complex} embedded in $\mathbb{R}^n$. We denote its vertices by $\mathcal{V}=\{v_i\}$, its edges by $\mathcal{E}=\{e_{ij}\}$ (where $e_{ij}$ is the oriented edge between $v_i$ and $v_j$), its faces by $\mathcal{F}=\{f_{ijk}\}$ (whose boundaries consist of oriented edges), its 3D cells by $\mathcal{C}=\{c_{ijkl}\}$, etc. More generally, we will also refer to the set of all $k-$simplices (simplices made out of $(k+1)$ vertices) of $M$ as $\smash{\mathcal{M}^k}$. For simplicity of exposition, we will assume that the manifold is \emph{oriented}, which is a negligible constraint given the local nature of this theory.

\subsection{Discrete Vector Bundles, Frame Bundles, and Connections}
In our discrete setup, a number of discrete notions of continuous definitions are rather natural to define as a collection of discrete objects on $M$, representing a discretization of their respective continuous notions. While these definitions alone are not sufficient (in particular, no notion of bundle topology will be defined yet), we will complete the discrete picture in the upcoming sections.

\begin{definition}[Discrete Vector Bundle]
	A discrete vector bundle of rank $r$ over a discrete orientable manifold $M$ is simply defined as a collection of vector spaces $\{\mathbf{E}_{v_i}\}_{v_i \in \mathcal{V}}\!\coloneqq\!\mathbf{E}(M) $ (i.e., one vector space per vertex), with $\mathrm{dim}(\mathbf{E}_{v_i})\!=\!r$. 
\end{definition}

We can then equip each of these vertex-based vector spaces with a frame to form a \emph{discrete analog of a local section of the frame bundle} over $M$.
\begin{definition}[Section of Discrete Frame Bundle]
	A section of the \emph{discrete frame bundle} of the rank-$r$ vector bundle $\mathbf{E}(M)$ consists in a collection of frames $\{\mathbf{F}_{v_i}\}_{v_i\in \mathcal{V}}\!\coloneqq\!\mathbf{F}(M)$ defining an arbitrary choice of frame for each vector space $\mathbf{E}_{v_i}$.
\end{definition}

We can now define a \textit{discrete connection} $\boldsymbol{\nabla}$ as a collection of maps between the extremities of each edge of $M$.

\begin{definition}[Discrete Connection]
	A \textit{discrete connection} $\boldsymbol{\nabla}$ is defined as an assignment, for each edge $e_{ij}$ of the discrete orientable manifold $M$,  of an invertible linear map $\mathcal{R}_{ij}\colon \mathbf{E}_{v_j}\to \mathbf{E}_{v_i}$, with $\mathcal{R}_{ij} \circ \mathcal{R}_{ji} \!=\! \mathbf{Id}_{v_i}$ representing the parallel transport induced by a continuous connection $\nabla$ along $e_{ij}$. 
\end{definition}

These maps $\{\mathcal{R}_{ij}\}_{e_{ij}\in \mathcal{E}}$ can be thought of as discrete parallel transport maps along edges, where the continuous definition from Eq.~\eqref{local_R} is integrated along an edge to define a map between two adjacent vertices. 
In practice, one can express each linear map $\mathcal{R}_{ij}$ of a connection through a matrix $R_{ij}\!\in\! GL(r)$, stored on vertex $v_i$, using a discrete frame field $\mathbf{F}(M)$. Note that a number of properties in the continuous case apply to these discrete equivalents. For instance, if each of the fibers $\mathbf{E}_{v_i}$ possesses an inner product $\langle\cdot,\cdot\rangle\smash{_{\mathbf{E}_{v_i}}}$, we say that the discrete connection is \textit{compatible with the metric} if the maps $\smash{\mathcal{R} _{ij}:  ( \mathbf{E}_{v_j} , \left\langle \cdot , \cdot \right\rangle _{v_j} ) \rightarrow ( \mathbf{E}_{v_i} , \left\langle \cdot , \cdot \right\rangle _{v_i} )}$ are orthonormal.

When each frame $\mathbf{F}_{v_i}$ is orthonormal, this implies $R_{ij}\!\in\! \mathrm{SO}(r)$: a frame at a vertex is naturally parallel-transported along an edge through a pure rotation.

	\begin{figure}[tb]
		\centering
        \includegraphics[width=\linewidth]{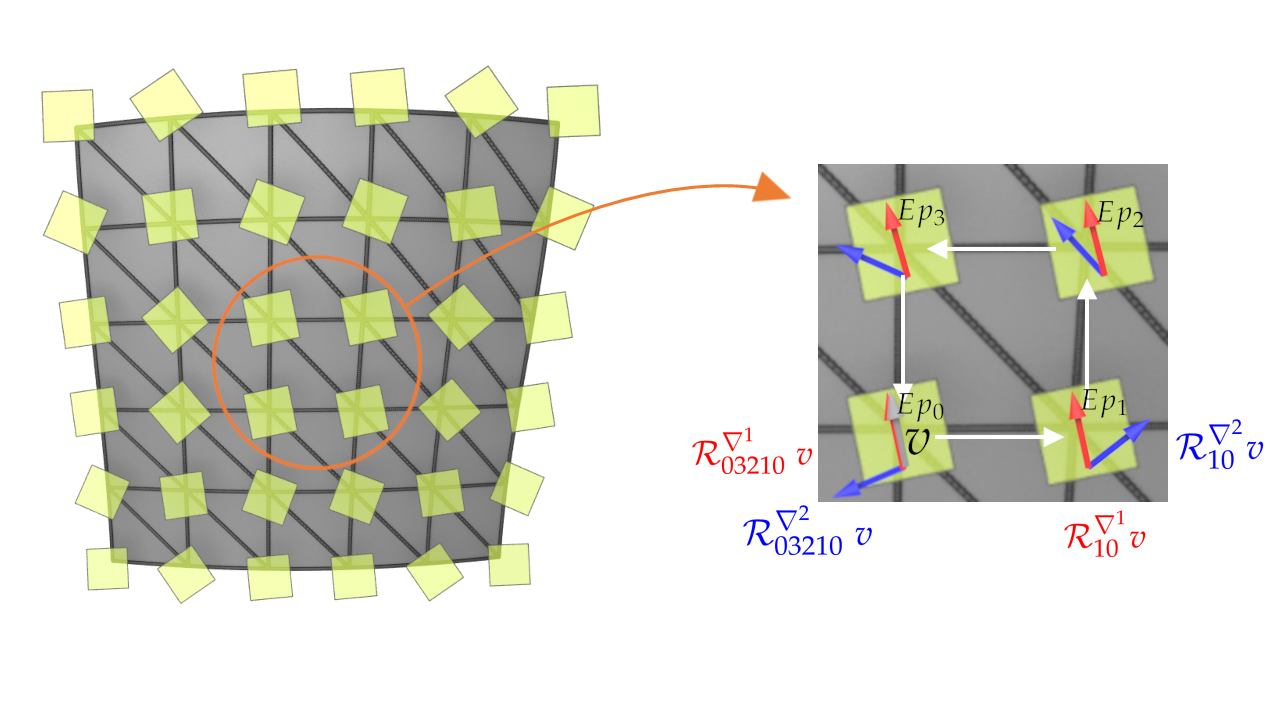}
		\vspace*{-6mm}
        \caption{\textbf{Discrete Connections.} A discrete vector bundle $\boldsymbol{E}$ assigns a vector space to each vertex on a mesh. To compare quantities in different fibers, we need to encode how the individual fibers are \textit{connected}. A connection consists of a collection of linear maps between neighboring fibers. In the figure, each box represents a frame for each fiber space over a vertex. The right part of the figure visualizes the local action of two different discrete connections. Starting with an initial vector \( v \in \boldsymbol{E}_{p_0} \) (gray), we illustrate the different parallel transports induced by two different connections, \(\nabla^1\) and \(\nabla^2\). Different connections on the same vector bundle result in distinct notions of parallel transport and curvature.}
		\label{fig:parallel_transporting}
	\end{figure}
    
	Given two discrete manifolds ${M,N}$ and a discrete vector bundle $\mathbf{E}$ over $N$, with a discrete connection $\boldsymbol{\nabla}$, we can finally define the notion of \emph{discrete pullback bundle} as follows:  if  $f\colon M\to N$ is a simplicial map associating simplices in $M$ to simplices in $N$,  the pullback bundle $f^\ast \mathbf{E}$ over $M$ is naturally defined as $(f^\ast \mathbf{E})_{v_i} \!=\! \mathbf{E}_{f(v_i)}$, $\forall v_i\!\in\!\mathcal{V}.$
	We then follow \cite{Hirani_Bianchi} for the definition of the \emph{pullback connection}.
	
	\begin{definition}[Pullback Connection~\cite{Hirani_Bianchi}]
		Let $f\colon M\to N$ be a simplicial map and $\mathbf{E}$ a discrete vector bundle over $N$ with discrete connection $\boldsymbol{\nabla}$. A pullback connection $f^\ast\boldsymbol{\nabla}$ is defined as a collection of linear maps $\{\mathcal{R}_{ij}\colon \mathbf{E}_{f(j)}\to \mathbf{E}_{f(i)}\}_{e_{ij}\in\mathcal{E}(M)}$ with:
		$$\mathcal{R}_{ij} := 
		\begin{cases*}
			\mathcal{R}_{f(v_i)f(v_j)} &  if $e_{f(v_i),f(v_j)}\in \mathcal{E}({N})$\\
			\mathrm{Id}_{\mathbf{E}_{f(v_i)}} = \mathrm{Id}_{\mathbf{E}_{f(v_j)}}& otherwise, $f(v_i)=f(v_j)$.
		\end{cases*}$$
	\end{definition}
	
	\subsection{Discrete Bundle-Valued Differential Forms}
	\label{sec:discrete_bundle_valued_forms}
	Given a discrete vector bundle $\bE$ over a discrete manifold $M$, we now define the notion of \emph{discrete bundle-valued $\ell-$form} as, for now, a collection of abstract maps. How these maps relate to the continuous case will be elucidated later (see Sec.~\ref{sec:DBEC}) in order to guarantee proper convergence to continuous forms for increasing mesh resolution. 
	
	\begin{definition}[Discrete (1,0)-tensor-valued $\ell-$form]
		\label{def:discreteForm}
		A discrete vector-valued $\ell-$form $\boldsymbol{\alpha}$ on $M$ is a collection of maps which, for each $\ell-$simplex $\sigma$ and one of its vertices $v$, returns a vector in $\mathbf{E}_{v}$, i.e., 
		\begin{equation}
			\label{eq:discrete_k_f}
			\boldsymbol{\alpha}\colon \sigma \in \mathcal{M}^\ell,v\in \sigma \subset \mathcal{V}(M) \mapsto  \boldsymbol{\alpha} ( \sigma , v) \in \mathbf{E}_{v},
		\end{equation}
		such that if $\bar{\sigma}$ is the simplex $\sigma$ with reversed orientation, one has
			$\boldsymbol{\alpha}(\bar{\sigma},v) = - \boldsymbol{\alpha}({\sigma},v) $
			for all \(v\!\in \!\sigma \).
	\end{definition}

	Note that this definition means that in contrast to the discretization of scalar-valued forms in FEEC or DEC, discrete bundle-valued differential forms are \emph{not} to be understood as a linear space of cochains, but rather as \emph{maps for simplex-vertex pairs}: bundle-valued $\ell-$forms should be regarded as abstract ``sided'' maps (as they are attached to a particular vertex of the simplex) for now. Moreover, we will assume \emph{for now} that a discrete bundle-valued $\ell-$form is defined through its values on all simplex-vertex pairs $(\sigma, v)$ for $\sigma$ is an $\ell-$simplex and $v$ is one of its vertices: later on (see Sec~\ref{sec:revisitingDisc}), once we precisely define how a discrete bundle-valued form is sampled from a continuous form, we will restrict this definition to only \textbf{one vector value per $\ell-$simplex}, like in existing finite-dimensional (scalar-valued) exterior calculus frameworks.

	We also define the notion of \emph{discrete $(1,1)-$tensor-valued $\ell-$forms} (or \emph{discrete homomorphism-valued} $\ell-$forms).
	Given a vector bundle $ \pi :E \!\rightarrow\! M$ with connection $ \nabla $, we consider the bundle $\mathrm{End}(E) \rightarrow M$ with induced connection $\nabla^{\mathrm{End}}$. One can equivalently describe it as the bundle of $(1,1)-$tensors of $E$. This latter description allows for a more flexible representation, preserving the underlying structure of endomorphism forms by considering separate evaluation and input fibers for each cell.

			\begin{definition}[Discrete $(1,1)-$tensor-valued $\ell-$form]
				
				\label{def:discrete-1-1-tensor-valued-forms} A discrete $(1,1)-$tensor-valued $\ell-$form $\boldsymbol{\beta}$ on $M$ is a collection of maps which, for each $\ell-$simplex $\sigma$ and two of its vertices ($w$, the input [or 'cut'] fiber, and $v$, the output (or \emph{evaluation}) fiber), returns a homomorphism between $\mathbf{E}_{v}$ and $\mathbf{E}_{w}$, i.e., 
				\begin{equation}
					\label{eq:discrete_k_f2}
					\boldsymbol{\beta}\colon \sigma \in \mathcal{M}^\ell, v\in \sigma, w\in \sigma \mapsto  \boldsymbol{\beta} ( \sigma , v, w) \in  \operatorname{Hom}  (\bE_w,\bE_v),
				\end{equation}
				such that if $\bar{\sigma}$ is the simplex $\sigma$ with reversed orientation, one has
					$\boldsymbol{\beta}(\bar{\sigma},v,w) \!=\! \!- \boldsymbol{\beta}({\sigma},v,w)$
					for all \(v,w \!\in\! \sigma \).
			\end{definition}
			
			\begin{remark} 
				The discrete curvature $ \boldsymbol{\Omega}^\nabla $ defined in Def.~\eqref{def:discrete-curvature} with evaluation and cut fibers at two different vertices is an example of a discrete endomorphism-valued $2-$form. Note that we could trivially extend these previous definitions to $(p,q)-$tensor-valued discrete forms using $p$ input vertices and $q$ output vertices; but we will not explore this extension in this paper.
			\end{remark}
			
			Following~\cite{Hirani_Bianchi} we can finally define the notion of \emph{pullback of a discrete bundle-valued form} as follows:
			\begin{definition}[Pullback of a Discrete Vector-Valued Form~\cite{Hirani_Bianchi}]
				Given a simplicial map $f\colon M\to N$ and a discrete vector bundle $\bE$ over $N$, let $\boldsymbol{\alpha}$ be an $\bE-$valued discrete $\ell-$form. We define the $f^\ast \bE-$valued pullback form $f^\ast\boldsymbol{\alpha}$ on $M$ as 
				$$f^\ast\boldsymbol{\alpha}(\underbrace{s\!\coloneqq\![v_0,\ldots,v_\ell]}_\text{$\in \mathcal{M}^\ell$},{v_0}) = \begin{cases}
					\boldsymbol{\alpha}([f(v_0),\ldots,f(v_\ell)],{f(v_0)}) &\text{if } [f(v_0),\ldots,f(v_\ell)]\in \mathcal{N}^\ell,\\
					0& \text{otherwise.}
				\end{cases}$$  		
			\end{definition}
			A similar definition for \emph{discrete endomorphism-/homomorphism-valued forms} holds trivially.
			
			\subsection{Discrete Connection One-Form}
			\label{sec:discrete1Form}
			Recall that in the smooth theory, once we are given a trivialization of the smooth vector bundle $E\!\to\! M$, we can treat the fibers as vector spaces of rank $r$  (say, copies of $\R^r$). The local connection $1-$form $\omega$ associated with a given connection $\nabla$ is then defined as $\nabla s^a\!=\! d s^a+ \omega^a_b s^b$, where $\omega$ measures how much the connection deviates from the exterior derivative $d$ (applied to coordinates); so
			if $\omega \!=\! 0$, i.e., $\nabla \!=\! d$ in a given frame field, then the matrix representation of parallel transport in this frame field satisfies $R_{ij} \!=\! \mathrm{Id}_{\R^r}$.
			
			We thus formulate the definition of a \emph{discrete connection $1-$form} as follows:
			
			\begin{definition}
				Let $M$ be a discrete manifold with discrete vector bundle $\bE(M)$ of rank $r$ and let $\mathbf{F}(M)$ be a section of the discrete frame bundle. For a given discrete connection $\boldsymbol{\nabla}$ defined by a set of edge matrices $R_{ij}$ when expressed in $\mathbf{F}$, we define its associated discrete local connection $1-$form $\boldsymbol{\omega}$ as the collection of $r\!\times\!r$ matrices $\boldsymbol{\omega}_{ij}$ (one per edge $e_{ij}\!\in\! M$) with\vspace*{-1mm}
				\begin{equation} \label{eq:discOneForm}
					\boldsymbol{\omega}_{ij} \coloneqq  R_{ij} - \mathrm{Id}_{v_i} \in \R^{r\times r}.\vspace*{-1mm}
				\end{equation}
			\end{definition}
			
			Note that a discrete local connection $1-$form thus also measures how much the discrete connection-induced parallel transport locally deviates from the identity, mimicking the continuous definition $\nabla s^a\!=\! ds^a + \omega^a_bs^b$. And just like in the smooth case, the realizations of the per-coordinate exterior derivative and the connection $1-$form for a given connection change depending on the chosen frame field $\mathbf{F}$. We will, in fact, exploit this property in our work.
			
			\begin{remark}
				\label{rem:omega-or-R-exact}
				We can also regard Eq.~\eqref{eq:discOneForm} as a linear approximation of the smooth case: if each parallel-transport edge matrix $R_{ij}$ comes from the integration of a continuous $1-$form $\omega$ over the edge $e_{ij}$, then it is given by $R_{ij}\!=\!P\exp\smash{\bigl(\int_{e_{ij}} \omega\bigr)}$.
				The first-order Taylor expansion of this path-ordered matrix exponential is $\smash{\mathrm{Id}_{v_i} + \int_{e_{ij}} }\omega$,
				ensuring consistency with the continuous case. 
				One could instead define $\boldsymbol{\omega}_{ij}$ as the integrated value of the continuous connection $1-$form $\omega$, and then linearize the exponential map by setting $R_{ij} \!\coloneqq\! \mathrm{Id}_{v_i}+ \boldsymbol{\omega}_{ij}$; but given the central role of parallel transport in our work, we find this second way to discretize connection and parallel transport less appealing.
			\end{remark}

			\begin{remark}
				\label{rem:disc-has-to-stay-local}
				In the context of this discretization approach, we presume the continuity of the underlying local frame field. However, this assumption encounters limitations due to the globally assigned frame per vertex. Notably, the applicability of a continuous local frame encounters constraints, as illustrated by the hairy ball theorem, which prohibits the existence of a global frame for the tangent bundle on a sphere. Extending this notion to bundles with potentially nontrivial characteristic classes, even for a seemingly straightforward base manifold like a tetrahedron, it may become infeasible to employ a single bundle chart to parametrize the bundle on the cell. Our work specifically focuses on scenarios where \emph{a local setup} suffices, and a single chart is adequate to parametrize the bundle. In this case, we point out that the parametrization of the bundle is independent of the chart. However, in cases involving multiple charts, a need arises for compatibility conditions between them to meet global topology constraints. Further exploration of these conditions remains a direction for future research.
			\end{remark}

			\subsection{Discrete Curvature Two-Form}
			\label{sec:discreteCurv2Form}
			
			Given a discrete connection $\boldsymbol{\nabla}$, we now design a notion of \emph{discrete connection curvature} akin to the continuous definition of $\Omega^\nabla$. 
			Since the Ambrose-Singer theorem links the curvature of the connection to the holonomy of an infinitesimal loop, it is tempting to define the discrete curvature $2-$form $\boldsymbol{\Omega}^{\boldsymbol{\nabla}}$ on a triangle $\sigma\!=\![v_0, v_1, v_2]$ on, say, fiber $\bE_{v_0}$ through the difference between the integrated connection $1-$form over the triangle edges and the identity at the evaluation fiber via \vspace*{-2mm}
			\begin{equation}
				\boldsymbol{\Omega^\nabla}(\sigma,v_0) \coloneqq \mathcal{R}_{01}\cdot\mathcal{R}_{12}\cdot\mathcal{R}_{20} - \mathrm{Id}_0\in \mathrm{End}(\bE_{v_0}). \vspace*{-2mm} \label{eq:simpleCurvature}
			\end{equation}
			However, Eq.~\eqref{eq:simpleCurvature} is not a constructive definition: one cannot meaningfully \emph{sum} this definition over two triangles, even if they share the same evaluation vertex $v_i$. In other words, since this curvature is associated with a loop instead of a simple chain, we cannot use the summability of chains to induce a notion of curvature integral over a larger region. We need to leverage the underlying algebraic structure of chains instead.
			
			Inspired by synthetic geometry for infinitesimal cells~\cite{Kock1997} and the definition proposed in \cite{schubel2018discretization,Hirani_Bianchi},
			we propose not to use parallel transport along triangle boundary loops, but instead, to \emph{compare parallel transport along two $1-$chains whose difference forms a triangle boundary loop}. This change, which amounts to comparing the integration of the connection form over two $1-$chains --- and each time parallel-transporting the result back to fiber $\bE_{v_0}$ --- may seem tautological, but it will allow us to define a notion of addition of the curvature $2-$form over two $2-$chains sharing a 1-chain. Note that it now casts the curvature two-form on $M$ as a discrete \emph{homomorphism-valued $2-$form}. 
			
			\begin{definition}
				\label{def:discrete-curvature}
				Let $M$ be a discrete orientable manifold equipped with a discrete vector bundle $\bE(M)$ and connection $\boldsymbol{\nabla} \!=\! \{\mathcal{R}_{ij}\}_{e_{ij}\in\mathcal{E}}.$ Let $\sigma = [v_0,v_1,v_2] \in \mathcal{M} ^2$ be a triangle of $M$. 
				For the evaluation fiber $\bE_{v_0}$ at $v_0$, the three expressions for the discrete curvature $2-$form $\smash{\boldsymbol{\Omega^\nabla}}$ depending on whether we use $v_1$, $v_2$, or $v_0$ as the cut vertex (comparing respectively parallel-transports  over the oriented $1-$chains $e_{01}$ vs. $e_{02}e_{21}$, $e_{01}e_{12}$ vs. $e_{02}$, and $e_{01}e_{12}e_{20}$ vs. $\varnothing$, see Fig.~\ref{fig:curvature-different-cuts}) are: \vspace*{-2mm}

				\begin{equation}\label{eq:curv01}
					\boldsymbol{\Omega^\nabla}(\sigma,v_0,v_1) = 
					\mathcal{R}_{01}- \mathcal{R}_{02} \mathcal{R}_{21} \, \in \operatorname{Hom} (\bE_{v_1},\bE_{v_0}), \vspace*{-1mm}
				\end{equation}
				\begin{equation}\label{eq:curv02}
					\boldsymbol{\Omega^\nabla}(\sigma,v_0,v_2) = 
					\mathcal{R}_{01}\mathcal{R}_{12} - \mathcal{R}_{02} \, \in \operatorname{Hom} (\bE_{v_2},\bE_{v_0}),
				\end{equation}
				\begin{equation} \label{eq:curv00}
					\boldsymbol{\Omega^\nabla}(\sigma,v_0,v_0) = 
					\mathcal{R}_{01}\mathcal{R}_{12} \mathcal{R}_{20}- \mathrm{Id}_0 \, \in \operatorname{Hom} (\bE_{v_0},\bE_{v_0}).
				\end{equation}
				\begin{figure}[h]\vspace*{-4mm}
					\centering
					\includegraphics[width =0.8\linewidth]{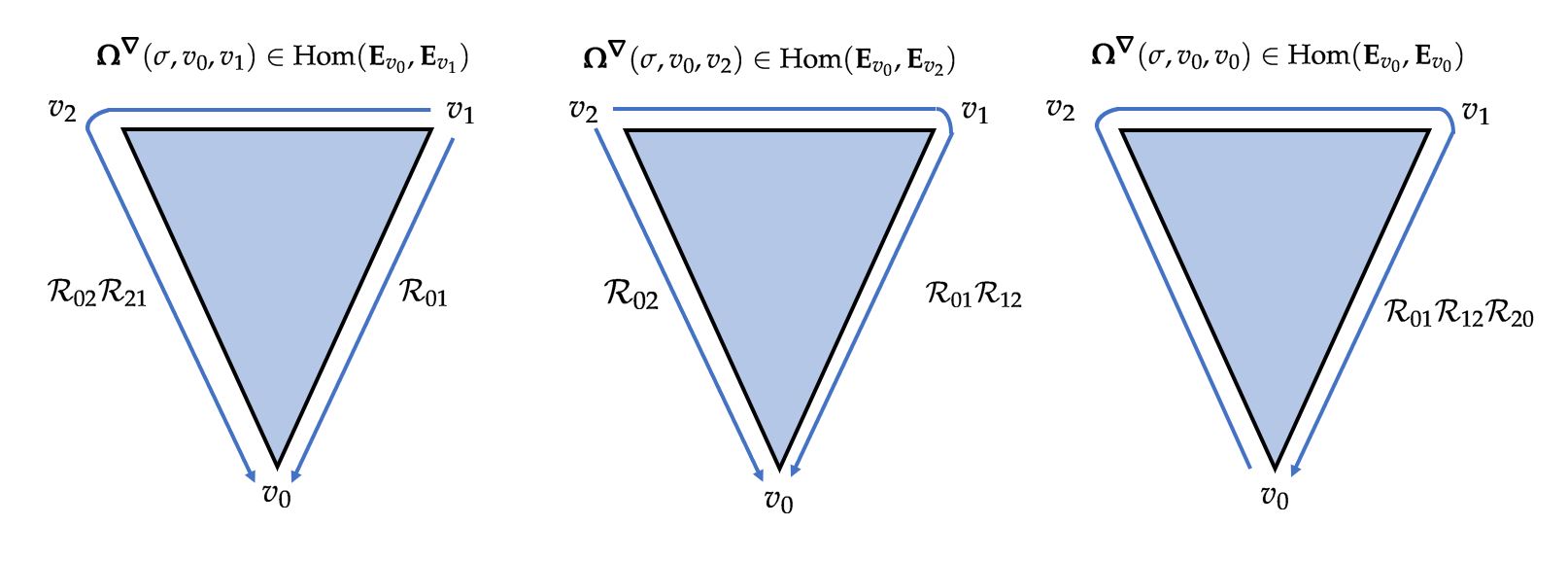}\vspace*{-1mm}
					\caption{\textbf{Curvatures.} Illustration of the discrete curvatures for different cut vertices.}
					\label{fig:curvature-different-cuts}
				\end{figure}
			\end{definition}
			
			From this definition, the meaning of ``\emph{evaluation fiber}'' and ``\emph{cut fiber}'' should become clear: instead of measuring the traditional holonomy of the loop starting at $v_0$ through $\smash{\boldsymbol{\Omega^\nabla}(\sigma,v_0)}$, we instead compare transport along two paths, both from the evaluation fiber to the cut fiber but on opposite sides of the triangle, with the path difference being the original triangle boundary.
			
			\begin{remark}
				\label{rem:discrete-curvature-non-simplicial}
				Given a \emph{non-simplicial} 2-cell with a discrete vector bundle and connection, we can easily extend the definition of the discrete curvature similarly through a difference of parallel transports along two paths following the boundary of the cell from the evaluation vertex to the cut vertex. 
			\end{remark}
			
			Eq.~\eqref{eq:curv00} matches the discrete holonomy expression from Eq.~\eqref{eq:simpleCurvature}. In fact, all three expressions listed in Def.~\ref{def:discrete-curvature} are trivial extensions of the definition in Eq.~\eqref{eq:simpleCurvature}: since $\mathcal{R}_{ij} \circ \mathcal{R}_{ji} \!=\! \operatorname{Id}_{v_i}$, they are all equivalent up to post-multiplication. Similarly, changing the evaluation fiber to another vertex than $v_0$ would simply require a pre-multiplication by the accumulated parallel transport from $v_0$ to the new evaluation fiber. Despite its apparent redundancy, the value of this encoding is in its \emph{summability}: the curvature $2-$form for two adjacent simplices can be summed if their evaluation and cut fibers are the same, resulting in an expression of the curvature on the (now non-simplicial) union still representing the difference in integration of the connection form over two $1-$chains evaluated at the evaluation fiber. 
			Moreover, we will see in the next section that the expressions we propose actually correspond to a particular choice of integration of the curvature $2-$form over $2-$simplices using an ordering induced by a canonical retraction, \emph{putting on a formal footing these seemingly arbitrary expressions}.
	
	\section{Discrete Bundle-valued  Exterior Calculus}
\label{sec:DBEC}
Given the discretization of bundle-valued forms and connections that we reviewed, we are now equipped with all the tools needed to describe our bundle-valued variant of discrete exterior calculus. Along the way, we will point out major differences with the typical discrete calculus of scalar-valued forms due to our need to deal with fiber-based evaluations; similarly, we will point out where a previous attempt at a calculus of bundle-valued forms~\cite{Hirani_Bianchi} fails to provide useful numerical evaluations, despite reproducing discrete Bianchi identities.  
The discretization of vector bundle-valued forms is highly dependent on the specific frame field chosen for the process. Therefore our first objective is to define a notion of canonical frame field that will serve as a tool for the discretization of vector bundle valued differential forms.

\subsection{Approach to discretization: parallel-propagated frame fields}
\label{sec:issue-discretization-discrete-d}
The discrete exterior derivative for scalar-valued forms finds a simple and exact discretization as the coboundary operator in numerical versions of exterior calculus by leveraging Stokes' theorem $\int_S d \cdot = \int_{\partial S} \cdot$, from which ensues a whole discrete de Rham complex~\cite{Bossavit_CEM,disc_differential_forms_modelling,FEEC}. In our bundle-valued case, finding a discrete version of the covariant exterior derivative $d^{\nabla}$ is more difficult:  for a bundle-valued form $\alpha$, one has to define the discrete equivalent of --- and eventually, numerically evaluate --- integrals of the type: \vspace*{-1mm}
\begin{equation}
	\int_{S} d^{\nabla}\alpha = \int_S d\alpha + \int_{S} \omega\wedge\alpha = \int_{\partial S} \alpha + \int_S \omega\wedge\alpha. \label{eq:example} \vspace*{-1mm}
\end{equation}
Note that such integrals of bundle-valued forms necessitate a \emph{chosen frame field} to be well defined, in contrast to the scalar-valued scenario. Moreover, while Stokes' theorem can be partially leveraged as indicated in Eq.~\eqref{eq:example}, the last term of this equation is particularly difficult to handle: it involves the integration of a wedge product which, even for scalar-based forms, endures severe limitations in the discrete realm~\cite{Kotiuga:Limits}. 
In order to deal with these two issues, we define a specific frame field that will enable a consistent framework for the discretization of bundle-valued exterior calculus with convergence guarantees.

\paragraph{Designing gauge fields to simplify evaluations.}\vspace*{-2mm}
Our approach relies on a simple, but key property afforded by gauge transformations: \emph{there always exists a local frame field with respect to which the connection $1-$form is zero at some local point} (see Fig.~\ref{fig:two-frames-one-ppf-one-not}). As described for instance in~\cite[Thm 2.1, page 107]{chern1999lectures}, given a local connection $1-$form $\omega$, one can always construct a local frame field such that the connection $1-$form $\Tilde{\omega}$ after gauge transformation (i.e., expressed in this new frame field) 
vanishes at a given point $c\!\in\! M$. 
Therefore, an approach to discretizing expressions such as Eq.~\eqref{eq:example} --- and thus inducing the notion of a discrete (integrated) exterior covariant derivative --- is to design an ``as-parallel-as-possible'' frame field which will render the connection $1-$form $\omega$ zero at a point in the simplex (we will use this point to be one of its vertices, or sometimes, its barycenter): indeed, assuming the connection $1-$form is Lipschitz continuous, the use of such a frame field will \emph{bound the value of the last term of Eq.~\eqref{eq:example}} by the diameter of the simplex --- thus \emph{guaranteeing its convergence to zero in the limit of mesh refinement}. Consequently, the integral of the covariant exterior derivative of a bundle-valued form over a simplex evaluated in a properly chosen frame field will be expressible, just like in the scalar-valued case, as \emph{boundary integral values} only. This approach, which recovers Stokes' theorem for flat connections, will be key to \emph{ensure correct numerical evaluation} in the limit of mesh refinement as it will allow us to recover actual continuous values of exterior covariant derivatives of bundle-valued forms. 

\begin{figure}[htb] \vspace*{-2mm}
	\centering
	\includegraphics[width=0.95\linewidth]{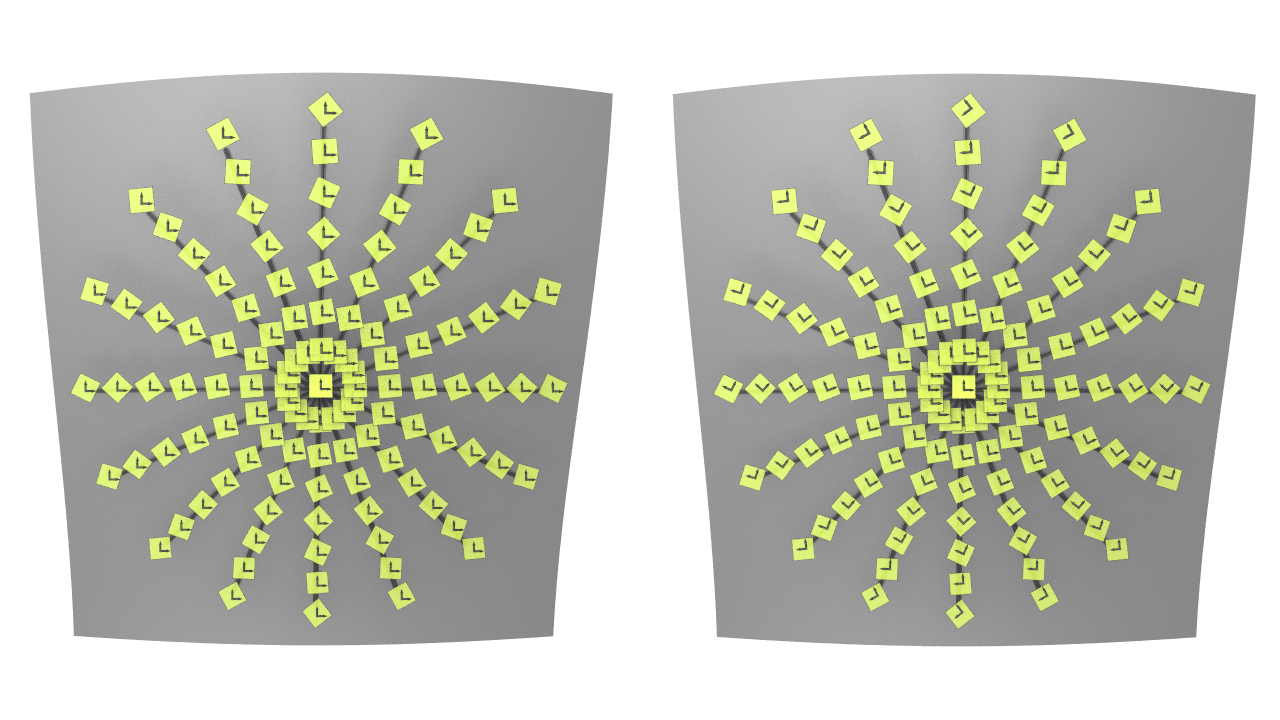} \vspace*{-9mm}
	\caption{We investigate two distinct frame fields on the same vector bundle with an identical connection $(E, \nabla)$. In this context, the boxes represent the individual fibers of the bundle. Both bundles exhibit a similar ``twisted'' structure, whereby the fibers are equally interconnected, implying that the connection (or parallel transport) of the bundle remains consistent.
		With a local frame construction facilitated by a rotation field $R^{\scriptscriptstyle\nabla}$, we can alter both the differential and the connection $1-$form. By employing a parallel-propagated frame field (left), we achieve ${\omega}^{\scriptscriptstyle\nabla}(c) = 0$ at a specific point $c$. In contrast, when using a non-parallel-propagated frame (right), the connection $1-$form does not vanish at $c$.
		Intuitively, the parallel-propagated frame field on the left ``follows'' the bundle along radial lines emanating from $c$. Consequently, $\Tilde{\omega}$ vanishes along these radial lines. This behavior highlights the distinctive properties of the two frame fields in terms of their interaction with the bundle's connection. \vspace*{-2mm}}
	\label{fig:two-frames-one-ppf-one-not}
\end{figure}

\paragraph{Parallel-propagating frame fields.}
\label{sec:ppfIntro}
To explain more concretely how one can design a frame field that bounds the wedge-product term from Eq.~\eqref{eq:example}, let us go back temporarily to the continuous case, where we can formally construct a notion of parallel-propagated frame field and understand its impact on the evaluation of the integrated exterior covariant derivative. 

\begin{definition}[Continuous Parallel-Propagated Frame] 
	\label{def:PPFcontinous}
	Let $ \pi : E \rightarrow M$ be a vector bundle with connection $\nabla$. Let $s=[v_0,\ldots,v_{\ell}]\!\subset\! M$ be a region in $M$ for which there exists a diffeomorphism to an $\ell-$simplex $\sigma\!=\![w_0,\ldots,w_{\ell}]$ where each point $v_i$ are mapped to an associated vertex $w_i$ of $\sigma$. Let $f$ be a local, arbitrary frame field of $E$ over this region $s$. For any given corner $v\!\in\!\{v_0,\ldots,v_{\ell}\}$, we also define a (strong deformation) retraction $\varphi_v:[0,1] \times s \rightarrow s$ derived from a canonical retraction $\varphi_w^\sigma:[0,1] \times \sigma \rightarrow \sigma $ of the simplex $\sigma$ through the aforementioned diffeomorphism, whose retracting paths are radially joining the vertex $w$ associated to point $v$ (see Fig.~\ref{fig:retraction-ppf} for an illustration in 2D); i.e., \vspace*{0mm} 
	\begin{equation}
		\begin{split}
			\varphi_w^\sigma\colon [0,1]\times \sigma\,&\to \sigma\\ (t,p)\,&\mapsto t\, w + (1-t)\, p.\\[-2mm]
		\end{split}
	\end{equation}

			\noindent Moreover, for any point $p\!\in\!\sigma$, we denote by $\mathcal{R}^{\!\scriptscriptstyle\nabla\!,v}(p)\colon E_p\to E_v$ the $\nabla-$induced parallel transport map from $E_v$ to $E_{p}$ along the path induced by the retraction $\varphi_v$ and $R^{\!\scriptscriptstyle\nabla\!,v}\colon \smash{\R^r\to \R^r}$ the matrix field representing $\mathcal{R}^{\!\scriptscriptstyle\nabla\!,v}(.)$ in the coordinate frame $f$.
			
			A frame field $\{f_a^{\scriptscriptstyle\nabla\!,v}\}$ over region $s$  is now called \textbf{parallel-propagated frame field} from $v\!\in\! M$ if 
			\begin{equation}\label{PPF_def} \mathcal{R} ^{\scriptscriptstyle\nabla\!,v}(p) f_a^{\scriptscriptstyle\nabla\!,v}(p)= f_a(v), \quad \text{for all $p \in S$, for all $a=1,\ldots,r$.}
			\end{equation} 
			i.e., if the frame $f_a(v)$ at $v$ has been parallel-transported throughout $s$ via the connection $\nabla$. Furthermore, we call $R^{\scriptscriptstyle\nabla\!,v}$ the \emph{gauge field} of the parallel-propagated frame field from $v\!\in\! M.$
		\end{definition}
		
		With this notion of a parallel-propagated field (denoted PPF subsequently for short), we can now define how a form is expressed in such a PPF field:
		\begin{definition}
			Using the same setting as in Def.~\ref{def:PPFcontinous} we denote the representation $\alpha^{\scriptscriptstyle\nabla\!,v}$ in the PPF from $v$ of an E-valued $\ell-$form $\alpha\!\in\! \Omega^{\ell}(M;E)$ through: \vspace*{-1mm}
			\[
			f_a^{\scriptscriptstyle\nabla\!,v}(\alpha^{\scriptscriptstyle\nabla\!,v})^a= \alpha = f_a\alpha ^a=  \left(f_b^{\scriptscriptstyle\nabla\!,v} (R^{\scriptscriptstyle\nabla\!,v})^b_a\right) \alpha^a,
			\]
			where $\alpha $ denotes at the same time the $\ell-$form and its local expression with respect to the local frame $f$. We have
			\begin{equation}\label{loc_alpha} 
				\alpha ^{\scriptscriptstyle\nabla\!,v} = R^{\scriptscriptstyle\nabla\!,v}\alpha .
			\end{equation} 
			Similarly, we denote the representation of the local connection $1-$form $\omega$ (associated with the connection $\nabla$) in the PPF from $v$ undergoing the gauge transformation by \vspace*{-1mm} 
			\[\omega^{\scriptscriptstyle\nabla\!,v} = R^{\scriptscriptstyle\nabla\!,v}(\omega-(R^{\scriptscriptstyle\nabla\!,v})^{-1}dR^{\scriptscriptstyle\nabla\!,v})(R^{\scriptscriptstyle\nabla\!,v})^{-1},\vspace*{-1mm}\] 
			see Eq.~\eqref{connection_trans}. 
		\end{definition}

		\begin{remark}\label{identification} Let us note that a connection-dependent integral can be written in terms of the frames as
			\[
			\connintidx{\varphi_v}_s \alpha=\int_s \mathcal{R} ^{\scriptscriptstyle\nabla\!,v}\alpha = \int_s \mathcal{R} ^{\scriptscriptstyle\nabla\!,v}(\alpha^{\scriptscriptstyle\nabla\!,v})^a f_a^{\scriptscriptstyle\nabla\!,v}(p)=\int_s (\alpha^{\scriptscriptstyle\nabla\!,v})^a f_a(v) \in E_v,
			\]
			where we used Eq.~\eqref{PPF_def}. 
			Because of the last expression, in our development below, this $E_v-$valued integral will be often identified with the $ \mathbb{R} ^r-$valued integral of the local expression $\alpha^{\scriptscriptstyle\nabla\!,v}$, namely, with
			\[
			\int_s \alpha^{\scriptscriptstyle\nabla\!,v}= \int_s R^{\scriptscriptstyle\nabla\!,v} \alpha,
			\]
			see Eq.~\eqref{loc_alpha}. Here $R^{\scriptscriptstyle\nabla\!,v}$ the matrix of $\mathcal{R} ^{\scriptscriptstyle\nabla\!,v}$ with respect to the frame $\{f_a\}$. One passes from one to the other representation by using the frame $\{f_a\}$ at the point $v$. 
		\end{remark}

		Note that with the change of frame field from $\{f_a\}$ to $\{f^{\scriptscriptstyle\nabla\!,v}_a\}$, forms are not the only entities changing: the differential $d$ also changes its coordinate-based expression in the PPF frame accordingly.
		
		It is now a trivial exercise to check that $\omega^{\scriptscriptstyle\nabla\!,v}(v) = 0$, i.e., $ \omega^{\scriptscriptstyle\nabla\!,v}$ vanishes at the point $v$, by construction. Assuming that the curvature $2-$form is bounded,  we can thus deduce that $\|\omega^{\scriptscriptstyle\nabla\!,v}\| = \mathcal{O}(h)$ on $s$, where $h$ is the diameter of the simplicial region $s$. 
		If we now express the covariant exterior derivative in this novel frame field $\{f^{\scriptscriptstyle\nabla\!,v}_a\}$, as explained in Sec.~\ref{sec:operators-smooth-world}, Eq.~\eqref{eq:ext-cov-derivative-local}, we obtain with this choice of coordinates the following connection-dependent expression:
		\[\left(d^\nabla\alpha\right)^{\scriptscriptstyle\nabla\!,v} = d \alpha^{\scriptscriptstyle\nabla\!,v} + \omega^{\scriptscriptstyle\nabla\!,v}\wedge\alpha^{\scriptscriptstyle\nabla\!,v}.\]
		
		We now recall from Def.~\ref{def:bundle-valued-integral-higher-order} that the connection dependant integral is given by
		\[
		\connintidx{\varphi_v}_s d^\nabla\alpha:=\int_s \mathcal{R} ^{\scriptscriptstyle\nabla\!,v} d^ \nabla \alpha \in E_v.
		\]
		As commented above, this $E_v-$valued integral can be identified with the $ \mathbb{R} ^r -$valued integral $\int_s  (d^\nabla\alpha)^{\scriptscriptstyle\nabla\!,v}$, which yields
		
		\begin{align}
			\int_s  (d^\nabla\alpha)^{\scriptscriptstyle\nabla\!,v}&=  
			\int_s d \alpha^{\scriptscriptstyle\nabla\!,v} + \omega^{\scriptscriptstyle\nabla\!,v}\wedge\alpha^{\scriptscriptstyle\nabla\!,v} = \int_{\partial s}\alpha^{\scriptscriptstyle\nabla\!,v} + \int_{s}\omega^{\scriptscriptstyle\nabla\!,v}\wedge\alpha^{\scriptscriptstyle\nabla\!,v}\notag\\
			 &=\int_{\partial s}R^{\scriptscriptstyle\nabla\!,v}\alpha + \int_{s}\omega^{\scriptscriptstyle\nabla\!,v}\wedge\alpha^{\scriptscriptstyle\nabla\!,v} .\label{eq:integral-in-ppf-of-d-nabla} 
		\end{align}

		\begin{figure}
			\centering
			\includegraphics[width = 0.4\linewidth]{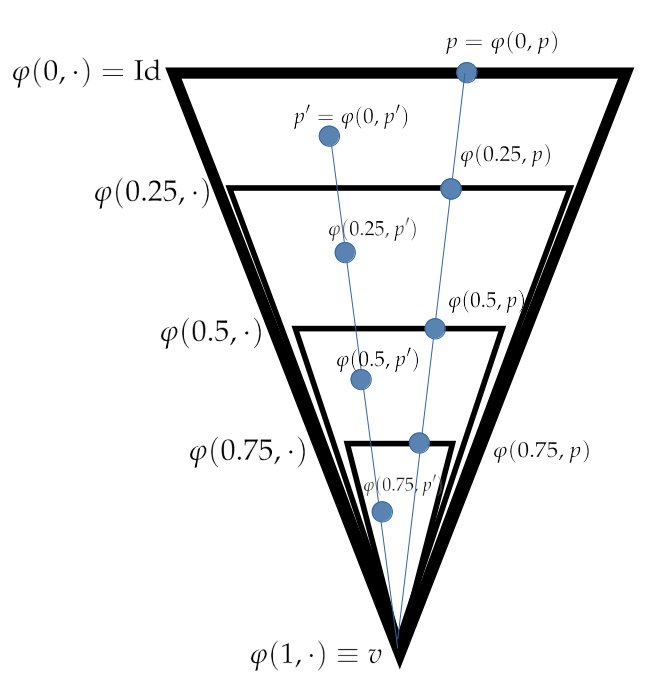}
			\caption{\textbf{Retraction.} In this illustration, we consider a simplicial cell $\sigma$ and demonstrate the retraction of $\sigma$ onto one of its vertices $v \!\in\! \sigma$. To aid visualization, an offset from the initial triangle is used to depict the shrinkage induced by the retraction function $\varphi$. 
				For any point $p \!\in\! \sigma$, the aforementioned retraction induces radial joining paths from $p$ to $v$, which can serve as paths for parallel transport from  source point $p$ to $v$. This parallel transport, defined along these paths, can be used subsequently to define the parallel propagated frame based in $v$.  }
			\label{fig:retraction-ppf}
		\end{figure}
		
		Given that the volume of $s$ is of order $h^{\ell+1}$, the last term $\int_{s} \omega^{\scriptscriptstyle\nabla\!,v}\wedge R^{\scriptscriptstyle\nabla\!,v}\alpha$ is of order {$\mathcal{O}(h^{\ell+2})$} due to our choice of frame field --- thus negligible when analyzing the convergence under refinement of the integrated covariant exterior derivative of the bundle-valued form $\alpha$. Only the boundary integral term is of order $\mathcal{O}(h^{\ell+1})$ and thus dominates under refinement. It holds
		\begin{equation}\label{eq:order-d-nabla} \int_{s}(d^\nabla\alpha)^{\nabla,v} = \int_{\partial s} \alpha^{\nabla,v} + \mathcal{O}(h^{\ell+2}).
		\end{equation}

		\paragraph{Consequences.}
		From Eq.~\eqref{eq:integral-in-ppf-of-d-nabla}, we notice that we can define a discrete bundle-valued exterior calculus built out of discrete forms which are the integrals of their continuous counterparts evaluated using parallel-propagated frames: once the high-order wedge product term is neglected, we get that the PPF-evaluated integral of the covariant exterior derivative of a form  $\alpha$ over a simplex $s$ can be well approximated by PPF-evaluated integrals of the form $\alpha$ over the boundary faces of $s$ --- a simple extension of the regular discrete exterior derivative for scalar-based forms defined via Stokes' theorem. The remainder of this paper describes this approach in detail; we will even show that notions that we proposed previously, like the discrete curvature two-form for example, can also be understood in light of this PPF-based calculus, providing a formal link between discrete and continuous calculus.
		
		\subsection{Revisiting the Discretization of Bundle-valued Forms}
		\label{sec:revisitingDisc}
		While Def.~\ref{def:discreteForm} formalized a discrete bundle-valued differential form as a collection of abstract maps, we can now express what these maps are in relation to the continuous case. Once this discretization (often referred to the de Rham map for scalar-based discrete exterior calculus~\cite{Bossavit:1999}) is defined, we can then define a notion of reconstruction which provides a differential form interpolating the discrete mesh values (referred to as the Whitney map in scalar-based discrete exterior calculus~\cite{Bossavit:1999}). \vspace*{-2mm}
		
		\paragraph{De Rham map for bundle-valued forms.}
		Based on the continuous retraction explained in Sec.~\ref{sec:ppfIntro}, the value of a discrete form evaluated on a simplex $\sigma$ at a vertex $v$ could be defined as the PPF-induced integration of its continuous counterpart form, where a simplicial region $s$ of a continuous manifold is approximated as the simplex $\sigma$, i.e., 
		\begin{equation}\label{eq:vertex-based-discretization} \smash{\boldsymbol{\alpha} (\sigma ,v) \coloneqq 
				\int_{s} \mathcal{R}^{\scriptscriptstyle\nabla\!,v}\alpha}\in E_v.
		\end{equation} 
		However, having to store values for each pair of simplex $\sigma$ and corner vertex $v$ grows rapidly with the dimension of the simplex, and it does not match the number of degrees of freedom of the scalar case of discrete or finite-element exterior calculus. 
		
		\noindent\textbf{Instead, we propose to use a \emph{barycenter-based PPF} per simplex, leading to a single value per $\ell-$simplex for an $\ell-$form} as conventional discretizations of exterior calculus of scalar-based forms~\cite{Bossavit_CEM,disc_differential_forms_modelling,FEEC}. Corner evaluations for this simplex can then be achieved through parallel transport $ \mathcal{R} _{v,c_s}$ from its barycenter $c_s$ to one of its corners $v$ --- which is either evaluated using the continuous parallel transport or approximated using the discrete $1-$connection (through the integral of a least-squares estimate of a linear $\omega$ within the simplex or the use of Whitney basis $1-$forms); in this case, the discretization of a bundle-valued form evaluated on a simplex $\sigma$ at a vertex $v$ becomes: \vspace*{-1mm}
		\begin{equation}
			\label{eq:center-based-discretization}
			\boldsymbol{\alpha}(\sigma, v) = \mathcal{R}_{v,c_s}\int_{s} \mathcal{R}^{\scriptscriptstyle\nabla\!,c_s}\alpha \in E_v.
		\end{equation}
		
		In practice, one could store directly the barycenter-based value, or just one of its derived corner values per simplex $\sigma$ as all other corner values can be recovered through parallel transport.
		While these two definitions exactly match for $1-$forms as parallel transport along edges is known (while it has to be interpolated from edge values for higher-order simplices), they differ by $\mathcal{O}(h^{\ell+3})$ for a general $\ell-$dimensional simplicial region $s$.
		Note that both definitions are frame-dependent: they rely on a canonical frame field, the PPF. This is a specificity of bundle-valued forms on non-flat manifolds as it involves non-commutative compositions of parallel transport: discretization must account for a particular choice of a local frame field to even make sense. Once these initial simplex values are established via the resulting de Rham maps we just described, the actual continuous manifold can be forgotten as all computations of our discrete calculus will solely use these values of discrete forms.

		\paragraph{Whitney map for bundle-valued forms.}
		Given the definition of the de Rham map described above, constructing a Whitney map (mapping all the barycenter-based values of all $\ell-$simplices of a discrete bundle-valued $\ell-$form $\boldsymbol{\alpha}$ to a differential bundle-valued form $\tilde{\alpha}$) is rather simple.
		For a simplex $\tau$ of dimension  $\dim(\tau)\geq \ell$, let $t$ be the simplicial region that is the image of $\tau$ under a diffeomorphism following the conventions in Def.~\ref{def:PPFcontinous}. Let $\sigma^{\ell}(\tau)$ be the set of all simplices of dimension $\ell$ included in $\tau$.
		The reconstruction of a differential $\ell-$form $\tilde{\alpha}$ expressed in the barycenter-based PPF of $\tau$ can be achieved through the weighted sum of the discrete bundle-valued forms for all $\ell-$dimensional simplices $\sigma \!\in\! \sigma^{\ell}(\tau)$. The pointwise evaluation of $\tilde{\alpha}$ for a given point $p$ then depends on the lowest-dimensional simplex $\tau$ that contains $p$: the evaluation of the bundle-valued $\ell-$form $\tilde{\alpha}$ on $\tau$ at point $p \!\in\! \tau$ given by
		\begin{equation}
			\label{eq:center-based-reconstruction}
			\tilde{\alpha}_{|_{\tau}} (p) =  \mathcal{R} ^{\scriptscriptstyle\nabla\!,c_{\tau}}(p)^{-1} \sum_{\sigma\in \sigma^{\ell}(\tau)}  \mathcal{R} _{c_{\tau},c_{\sigma}} \, \phi_{\sigma}(p)  \, \boldsymbol{\alpha}(\sigma,c_{\sigma}) \;\in \Omega  ^\ell( \tau , E),
		\end{equation}
where $\phi_{\sigma}$ is the Whitney $\ell-$form associated with $\sigma$ within $\tau$, and $\mathcal{R} _{c_{\tau},c_{\sigma}}$ is barycenter-to-barycenter parallel transport constructed based on the discrete $1-$form connection. 
		One can check that the de Rham map of a Whitney map of a discrete bundled-valued form is the identity, providing a proper link between continuous and discrete bundle-valued forms: indeed, for a given $\ell-$cell $\eta$, the Whitney map for any point $p$ inside $\eta$ is given by 
\begin{equation*}
\tilde{\alpha}_p =  \mathcal{R} ^{\scriptscriptstyle\nabla\!,c_{\eta}}(p)^{-1} \sum_{\sigma\in \sigma^{\ell}(\eta)}  \mathcal{R} _{c_{\eta},c_{\sigma}} \, \phi_{\sigma}(p)  \, \boldsymbol{\alpha}(\sigma,c_{\sigma}) {=} \mathcal{R} ^{\scriptscriptstyle\nabla\!,c_{\eta}}(p)^{-1}  \mathcal{R} _{c_{\eta},c_{\eta}} \, \phi_{\eta}(p)  \, \boldsymbol{\alpha}(\eta,c_{\eta}), 
\end{equation*}
since the set of all $\ell-$simplices $\sigma^\ell(\eta)$ is reduced to $\eta$ for any $p\!\in\! {\eta}$. Therefore, the re-discretization of the Whitney map through the integral in the parallel-propagated frame field yields back the original discrete form.
		
		\begin{remark}
			Note that we provide, through Eq.~\eqref{eq:center-based-reconstruction}, a piecewise expression of the Whitney map to convert the discrete values of $\ell-$simplices into a differential form. This should not be surprising as transitioning between simplices of different dimensions means that the local frame field changes. Consequently, the coordinate representation of the form's evaluation may exhibit jumps across simplices. While readers may find this behavior perplexing, it mirrors the jumping behavior already observed in scalar-valued Whitney forms when changing cells, as only tangential continuity is enforced. Moreover, in our case, we must now also accommodate the varying frame fields, leaving no choice but to use a piecewise definition. 
		\end{remark}
        
\begin{remark}        
It is possible to construct a PPF for any small neighborhood containing multiple simplices with an arbitrary initial local frame field. Indeed, we may consider a geodesic graph with edges connecting the barycenter of each simplex to the barycenter of each of its incident simplices. Picking a seed center $c_{\sigma_0}$, we can parallel transport its frame to all the barycenters in the neighborhood through a minimum spanning tree. We then parallel-transport the frames at each barycenter to all the points within the interior of its associated simplex. In this PPF covering the neighborhood (we can make it smooth through convolution with a mollifier if the frame field needs to be continuous), the integral of any $\ell-$form can be assembled through $\mathcal{R}_{c_{\sigma_0},c_\sigma}\alpha(\sigma,c_\sigma)$ with an error of $\mathcal{O}(h^{\ell+2})$ under the assumption of bounded curvature. This allows bundle-valued $\ell-$forms to be locally treated as a vector-valued $\ell-$cochain through integration over an oriented local $\ell-$submanifold discretized as an $\ell-$chain. However, one should note that it is path-dependent on the choice of parallel transport (through a minimum spanning tree in our example), and will not result in a discrete version of Stokes' theorem, as expected from the continuous theory.
\end{remark}
		
\subsection{\!Discrete \!Exterior \!Covariant \!Derivative \!for \!Bundle-valued\! Forms\!}
		\label{sec:DECDbundle}
		In this section, we formally define our discrete operator for the exterior covariant derivative, in two steps to highlight the similarities and differences with previous work. 
		
		\paragraph{Sided Exterior Covariant Derivative.}
		Guided by Eq.~\eqref{eq:integral-in-ppf-of-d-nabla}, Eq.~\eqref{eq:order-d-nabla} and our meaning of a discrete form from Eq.~\eqref{eq:vertex-based-discretization}, we first define a discrete operator $\mathfrak{d}^{ \boldsymbol{\nabla}}$ which approximates the exterior covariant derivative.

				\begin{definition}[PPF-induced Covariant Exterior Derivative]
					\label{def:ppf-derivative}
					Let $M$ be a discrete manifold and $\mathbf{E}\!\to\! M$ be a discrete vector bundle with connection $\boldsymbol{\nabla} \!=\! \{\mathcal{R}_{ij} \}$.  Let $\sigma \!=\! [v_0,\ldots,v_{\ell+1}]$ be an $(\ell+1)-$simplex and $\boldsymbol{\alpha}$ a discrete $\bE-$valued $\ell-$form. We define the \emph{sided} discrete exterior derivative of the form $\boldsymbol{\alpha}$ as 
					\begin{align}
						\mathfrak{d}^{\boldsymbol{\nabla}}\boldsymbol{\alpha}([v_0,\ldots,v_{\ell+1}],v_0) = &\mathcal{R}_{0,1}\ \boldsymbol{\alpha}([v_1,\ldots,v_{\ell+1}],v_1)\notag\\ 
						&+ \sum_{i = 1}^{\ell+1} (-1)^i \boldsymbol{\alpha}([v_0,\ldots,\hat{v}_i,\ldots,v_{\ell+1}],v_0).
						\label{eq:defFrakD}
					\end{align}
					
				\end{definition}

				\begin{remark}
					\label{rem:ppf-integral-ppf-derivative}
					The operator $\mathfrak{d}^{\boldsymbol{\nabla}}$ was already introduced in an equivalent form in~\cite{Hirani_Bianchi} as \emph{the} exterior covariant derivative. One of our contributions is to show that if the discrete form $\boldsymbol{\alpha}$ arises through integration in a vertex-based PPF and the discrete connection $\boldsymbol{\nabla}$ arises through integration from a smooth connection $\nabla$, then this operator approximates the integral in Eq.~\eqref{eq:integral-in-ppf-of-d-nabla} with
					
					\[    \mathfrak{d}^{\boldsymbol{\nabla}}\boldsymbol{\alpha}([v_0,\ldots,v_{\ell+1}],v_0) \!=\! \int_{s} \mathcal{R} ^{\scriptscriptstyle\nabla\!,v_0}\  d^\nabla\alpha \!+\! \mathcal{O}(h^{\ell+2})
					\]
					and, thus, converges under refinement.
					To see this, consider the case of a one-form $\alpha$ and a triangle $\sigma\!\coloneqq\![v_a,v_b,v_c]\subset M$. From Eq.~\eqref{eq:integral-in-ppf-of-d-nabla}, we know that the integral over $\sigma$ of the exterior covariant derivative in the PPF based at $v_a$ turns into
					\begin{align*}
						\int_{\sigma} R^{\scriptscriptstyle\nabla\!,v_a} d^\nabla\alpha  
						= \int_{[v_a,v_b]} R^{\scriptscriptstyle\nabla\!,v_a}\alpha + \int_{[v_b,v_c]}  R^{\scriptscriptstyle\nabla\!,v_a}\alpha + \int_{[v_c,v_a]} R^{\scriptscriptstyle\nabla\!,v_a}\alpha + \mathcal{O}(h^{3}),
					\end{align*}
					see Remark~\ref{identification}.
					Def.~\ref{def:ppf-derivative} tells us that $\mathfrak{d}^{\boldsymbol{\nabla}} \boldsymbol{\alpha}([v_a,v_b,v_c],v_a)\coloneqq  \mathcal{R} _{v_av_b} \boldsymbol{\alpha}([v_b,v_c],{v_b}) - \boldsymbol{\alpha}([v_a,v_c],{v_a}) + \boldsymbol{\alpha}([v_a,v_b],{v_a}).$ Since our values from the discrete form $\boldsymbol{\alpha}$ were induced from a $v_a-$centered PPF, we have
					\[
					\int_{[v_a,v_b]} \mathcal{R} ^{\scriptscriptstyle\nabla\!,v_a}\alpha =\boldsymbol{\alpha}([v_a,v_b],{v_a})\quad\text{and}\quad \int_{[v_c,v_a]} \mathcal{R} ^{\scriptscriptstyle\nabla\!,v_a}\alpha=- \boldsymbol{\alpha}([v_a,v_c],{v_a}),
					\] 
					hence the continuous and discrete definitions match if we can prove that:
					\[
					\mathcal{R} _{v_a,v_b} \boldsymbol{\alpha}([v_b,v_c],{v_b}) - \int_{[v_b,v_c]}  \mathcal{R} ^{\scriptscriptstyle\nabla\!,v_a}\alpha  = \mathcal{O}(h^{3}).
					\]
					Indeed, one has:
					\begin{equation}
						\label{eq:error-opp-face-vertex-based-PPF}
						\mathcal{R} _{v_a,v_b}\int_{[v_b,v_c]} \mathcal{R} ^{\scriptscriptstyle\nabla\!,v_b}\alpha - \int_{[v_b,v_c]} \mathcal{R} ^{\scriptscriptstyle\nabla\!,v_a}\alpha = \int_{[v_b,v_c]} \left(\mathcal{R} _{v_a,v_b} \mathcal{R} ^{\scriptscriptstyle\nabla\!,v_b} - \mathcal{R} ^{\scriptscriptstyle\nabla\!,v_a}\right)\alpha,
					\end{equation}
					and since for any point $p$ in $\sigma$ we can build an auxiliary simplex $[v_a,v_b,p]$ such that Eq.~\eqref{eq:error-opp-face-vertex-based-PPF} turns into an integral of evaluations of discrete curvature expression in the sense of Def.~\ref{def:discrete-curvature}, we get 
					\[\left(\mathcal{R} _{v_a,v_b} \mathcal{R} ^{\scriptscriptstyle\nabla\!,v_b} - \mathcal{R} ^{\scriptscriptstyle\nabla\!,v_a}\right)(p) = \boldsymbol{\Omega^\nabla}([v_a,v_b,p],v_a,p)\;\sim\; \mathcal{O} (h^2).\]
					Below, we will delve into further detail regarding how the discrete curvature, defined as the difference of parallel transport, accurately approximates the sampled smooth curvature up to order $\mathcal{O}(h^3)$. This analysis will provide justification for the estimation $\boldsymbol{\Omega^\nabla}([v_a,v_b,p],v_a,p)\;\sim\; \mathcal{O} (h^2)$.  
					
					Hence, the discrete operator $\mathfrak{d}^{ \boldsymbol{\nabla }}$ will converge in the limit to the correct pointwise value of the exterior covariant derivative. Note that this argument holds for a form of arbitrary order $\ell$: each time, it is the opposite face to the evaluation vertex that needs to be examined to prove the equivalence between discrete and continuous definitions up to $\mathcal{O}(h^{\ell+2})$.

				\end{remark}

			\paragraph{Averaging Derivatives.}
			While we proved the convergence of $\mathfrak{d}^{\boldsymbol{\nabla}}$ to its continuous equivalence, one may note that the order of accuracy is not sufficient to ensure that two consecutive applications of this operator ($\mathfrak{d}^{\boldsymbol{\nabla}} \circ \mathfrak{d}^{\boldsymbol{\nabla}}$) will still converge, which is the real obstacle to establishing the Bianchi identities (Eqs.~\eqref{eq:algBianchi} and ~\eqref{eq:diffBianchi}). We thus further alter our discrete exterior covariant derivative operator to gain one additional order of accuracy through local averaging, with an operator we denote as the \emph{alternation operator}.
			
			\begin{definition}[Alternation Operator for Discrete Bundle-Valued Forms] 
				\label{def:alternation_vector_valued_forms}
				Let $\bE$ be a discrete vector bundle with connection $\boldsymbol{\nabla}\!=\!\{\mathcal{R}_{ij}\}_{(ij)\in\mathcal{E}}$ over a discrete manifold $M$. Let $\sigma= [v_0,\ldots,v_{\ell}]\in M$ be an $\ell-$simplex. For a discrete $\bE$ valued $\ell-$form $\boldsymbol{\alpha}$ we define its ``alternated'' form $\smash{\mathrm{Alt}^{\boldsymbol{\nabla}}(\boldsymbol{\alpha})}$ on $\sigma$ evaluated at $v_0$ via
				\[\mathrm{Alt}^{\boldsymbol{\nabla}}(\boldsymbol{\alpha})([v_0,\ldots,v_\ell],{v_0}) = \frac{1}{(\ell+1)!}\sum_{\pi\in S_{\ell+1}}\mathrm{sgn}(\pi) \mathcal{R}_{v_0,v_{\pi(0)}}\boldsymbol{\alpha}([v_{\pi(0)},\ldots,v_{\pi(\ell)}],{v_{\pi(0)}}).
				\]
				where $S_{\ell+1}$ is the set of all permutations $\pi$ of the $(\ell+1)$ vertices of $\sigma$. 				
			\end{definition}

			Note that this operator will allow us to transform the operator $\mathfrak{d}^{ \boldsymbol{\nabla }}$ so that the vertex $v_0$ in Eq.~\eqref{eq:defFrakD} from Def.~\ref{def:ppf-derivative} is no longer singled out arbitrarily: the reshuffling from the alternation operator will average all the similar discrete exterior covariant derivative estimates using all other vertices. Moreover, one can easily check that the connection $1-$form is stable under this alternation operator.
			
			\begin{remark}
				\label{rem:Alternation-not-a-projector}
				Note that the \emph{alternation of a discrete bundle-valued form} requires a connection to allow for the averaging to be performed in a common fiber, in contrast to the direct symmetrization of the smooth case. It is thus easy to show that in general, $\mathrm{Alt}^{\boldsymbol{\nabla}}(\mathrm{Alt}^{\boldsymbol{\nabla}}(\boldsymbol{\alpha}))(\sigma,v_0) \neq \mathrm{Alt}^{\boldsymbol{\nabla}}(\boldsymbol{\alpha})(\sigma,v_{0})$, i.e., the alternation operator is not a projector. However, we preserve the fact that the alternation operator commutes with the pullback of a simplicial map in the sense that $\mathrm{Alt}^{f^\ast\boldsymbol{\nabla}}(f^\ast\boldsymbol{\alpha})(\sigma,{v_0}) \!=\! \mathrm{Alt}^{\boldsymbol{\nabla}}(\boldsymbol{\alpha})(f(\sigma),{f(v_0)})$.
			\end{remark}
			
			\paragraph{Discrete Exterior Covariant Derivative.}
			\noindent We now propose to change the formula of the exterior derivative as follows.
			
			\begin{definition}[Discrete Bundle-valued Exterior Covariant Derivative]\label{def:ext_der_alternation_group}
				Let $\bE$ be a discrete vector bundle with connection $\boldsymbol{\nabla}$ over a discrete manifold $M$. Let $\sigma \!\subset\! M$ with $ \sigma \!=\! [v_0,\ldots,v_{\ell+1}]$ be a $(\ell+1)-$simplex. For a discrete vector-valued differential $\ell-$form $\boldsymbol{\alpha}$ we define the discrete covariant exterior derivative with the alternation from Def.~\ref{def:alternation_vector_valued_forms} as 
				\[
				\bolddnab\boldsymbol{\alpha}([v_0,\ldots,v_{\ell+1}],{v_0}) = \mathrm{Alt}^{\boldsymbol{\nabla}}\mathfrak{d}^{\boldsymbol{\nabla}}(\boldsymbol{\alpha})([v_0,\ldots,v_{\ell+1}],{v_0}).
				\]
			\end{definition}
			
			Applying the alternation operator after $\mathfrak{d}^{\boldsymbol{\nabla}}$ removes the arbitrariness 
			of picking the second vertex in the simplex $\sigma$ in Eq.~\eqref{eq:defFrakD}, thus in effect averaging various \emph{sided} exterior covariant derivative estimates of equal approximation order $\mathcal{O}(h^{\ell+2})$. Through the averaging, the resulting discrete exterior covariant derivative gains one order of accuracy, becoming accurate to $\mathcal{O}(h^{\ell+3})$. This accuracy improvement is akin to the accuracy gain one obtains when using centered vs. sided finite difference approximations of derivatives. More precisely, we will show in the convergence proof in Thm.~\ref{thm:alternation-augments-accuracy} that if $ \alpha \in \Omega ^\ell(M, E)$ and $ \boldsymbol{ \alpha  }$ is the associated discrete form as defined in Eq.~\eqref{eq:vertex-based-discretization}, then the following holds 
			\[\bolddnab\boldsymbol{\alpha}([v_0,\ldots,v_{\ell+1}],v_0) \!=\! \mathcal{R} _{v_0,c_s}\int_{s} \mathcal{R} ^{\scriptscriptstyle\nabla\!,c_s} \, d^\nabla\alpha \!+\! \mathcal{O}(h^{\ell+3}),\]
			
			where $c$ denotes the barycenter of a region. That is, the alternation results in a barycenter-based PPF evaluation, parallel transported back to the evaluation vertex $v_0$. This barycentric PPF, replacing the vertex-based PPF, is what leads to an improvement in accuracy, even though the discretization of $\alpha$ is still achieved through an integral within a vertex-based PPF.
		
			\subsection{Discrete Endomorphism-valued Exterior Covariant Derivative}\label{sec:disrete-end-valued}
			After defining a discrete exterior covariant derivative applicable to bundle-valued forms in Sec.~\ref{sec:DECDbundle}, we need to extend this operator to discrete two-prong endomorphism-valued forms as well in order to be able to apply it, for instance, to discrete curvatures. 
			
			\begin{definition}(PPF-induced discrete exterior covariant derivative for endomorphism-valued forms) 
				\label{def:lie-algebra-ppf-derivative}
				Let $\bE$ be a discrete vector bundle over a discrete manifold $M$ with connection $\boldsymbol{\nabla}\!=\!\{\mathcal{R}_{ij}\}$. Let $\sigma = [v_0,\ldots,v_{\ell+1}]\subset M$ be a $(\ell+1)-$simplex and  $\boldsymbol{\beta}$ be a discrete $(1,1)-$tensor valued $\ell-$form. Finally, for $v\!\in\! \sigma$, let $\sigma_v =\smash{[v_0,\ldots,\hat{v},\ldots,v_{\ell+1}]}$ be the simplicial face of $\sigma$ opposite to vertex $v$.
				We define the PPF-induced discrete exterior covariant derivative on simplex $\sigma$ for $(1,1)-$tensor valued forms given an evaluation fiber $\smash{\bE_{v_0}}$ and a cut fiber $\bE_{v_{\ell+1}}$ through:\vspace*{-3mm}
				\begin{align*}
					\mathfrak{d}^{\!\boldsymbol{\nabla}}\boldsymbol{\beta}(\sigma, v_0,v_{\ell+1}) \!=\! &\mathcal{R}_{01} \boldsymbol{\beta}(\sigma_{v_0},v_1,v_{\ell+1})+ \sum_{i = 1}^\ell (-1)^i \boldsymbol{\beta}(\sigma_{v_i},v_0,v_{\ell+1})\\
					 &+ (-1)^{\ell+1} \boldsymbol{\beta}(\sigma_{v_{\ell+1}},v_0,v_\ell) \mathcal{R}_{\ell,\ell+1}.\vspace*{-3mm}
				\end{align*}
			\end{definition}
			This new definition is simply a variant of Def.~\ref{def:ppf-derivative} to account for the two-prong nature of discrete endomorphism-valued forms. It was already proposed, as is, in~\cite{Hirani_Bianchi}; but for the same convergence issues discussed earlier, we must complement this sided definition with a post-alternation to make it useful numerically, which also requires a new definition for the case of $(1,1)-$tensor-valued forms, as given below.
			
			\begin{definition}(Alternation Operator for Discrete $(1,1)-$tensor-valued Forms)
				\label{def:lie-algebra-alternation}
				Let $\bE$ be a discrete vector bundle with connection $\boldsymbol{\nabla}\!=\!\{\mathcal{R}_{ij}\}_{(ij)\in\mathcal{E}}$ over a discrete manifold $M$. Let $\sigma= [v_0,\ldots,v_{\ell}]\in M$ be an $\ell-$simplex. For a discrete $(1,1)-$tensor valued $\ell-$form $\boldsymbol{\beta}$ we define its ``alternated'' form $\smash{\mathrm{Alt}^{\boldsymbol{\nabla}}(\boldsymbol{\beta})}$ 
				for an $\ell-$form $\boldsymbol{\beta}$ on the simplex $\sigma$ evaluated in $v_0$ with input in $v_\ell$ through 
				\begin{align*}
					\mathrm{Alt}&^{\boldsymbol{\nabla}}\boldsymbol{\beta}([v_0,\ldots,v_\ell],{v_0},v_{\ell}) \\
					&= \frac{1}{(\ell+1)!}\sum_{{\pi\in S_{\ell+1}}} \left(\frac{1+\mathrm{sgn}(\pi)}{2}\mathcal{R}_{v_0,v_{\pi(0)}}\boldsymbol{\beta}([v_{\pi(0)},\ldots,v_{\pi(\ell)}],{v_{\pi(0)}},{v_{\pi(\ell)}}) \mathcal{R}_{v_{\pi(\ell)},v_\ell}\right.\\
					&\quad \left.+\frac{\mathrm{sgn}(\pi)-1}{2}\mathcal{R}_{v_0,v_{\pi(0)}}\boldsymbol{\beta}([v_{\pi(0)},\ldots,v_{\pi(\ell)}],{v_{\pi(0)}},{v_{\pi(\ell)}}) \mathcal{R}_{v_{\pi(\ell)},v_{\pi(\ell-1)}} \mathcal{R}_{v_{\pi(\ell-1)},v_\ell}\right).
				\end{align*}
			\end{definition}
			
			\begin{remark}[Invariance of the alternation for the curvature form]
				\label{rem:Omega-invariant-alt}
				Note that the curvature form is stable under this definition of the alternation of $(1,1)-$tensor valued forms: for a simplicial cell $\sigma = [abc],$ one can directly check that
				\vspace*{-1mm}
				\begin{align*}
					\mathrm{Alt}^{\!\boldsymbol{\nabla}}\!\boldsymbol{\Omega^{\!\nabla}}\!([abc],a,c)&=\frac{1}{6}\big(\boldsymbol{\Omega^{\!\nabla}}\!([abc],a,c) \!+\! R_{ab}\boldsymbol{\Omega^{\!\nabla}}\!([bca],b,a) R_{ac} \!+\! R_{ac}\boldsymbol{\Omega^{\!\nabla}}\!([cab],c,b) R_{bc}\\
					&\;-\!\boldsymbol{\Omega^{\!\nabla}}\!([acb],\!a,\!b)R_{bc} \!-\! R_{ab}\boldsymbol{\Omega^{\!\nabla}}\!([bac],\!b,\!c) \!-\! R_{ac}\boldsymbol{\Omega^{\!\nabla}}\!([cba],\!c,\!a) R_{ab} R_{bc}\big)\\
					& = \boldsymbol{\Omega^{\!\nabla}}([abc],a,c).\vspace*{-3mm}
				\end{align*}
				This property would not have held if an average with pre- and post- composition along a joining edge had been used.
			\end{remark}
			
			From the two operators above, we can now define our discrete exterior covariant derivative for endomorphism-valued forms.
			
			\begin{definition} [Discrete endomorphism-valued exterior covariant derivative]
				Let $\bE$ be a discrete vector bundle over a discrete manifold $M$ with connection $\boldsymbol{\nabla}\!=\!\{\mathcal{R}_{ij}\}$. Let $\sigma = [v_0,\ldots,v_{\ell+1}]\subset M$ be a $(\ell+1)-$simplex and  $\boldsymbol{\beta}$ be a discrete $(1,1)-$tensor valued $\ell-$form. 
				We define the discrete covariant exterior derivative on simplex $\sigma$ for $(1,1)-$tensor valued forms given an evaluation fiber $\smash{\bE_{v_0}}$ and a cut fiber $\bE_{v_{\ell+1}}$ through:\vspace*{-1mm} 
				\[\bolddnab\boldsymbol{\beta}([v_0,\ldots,v_{\ell+1}],{v_0},{v_{\ell+1}}) = \mathrm{Alt}^{\boldsymbol{\nabla}}\mathfrak{d}^{\boldsymbol{\nabla}}\boldsymbol{\beta}([v_0,\ldots,v_{\ell+1}],{v_0},{v_{\ell+1}}).\] 
			\end{definition}
			
			As we will discuss at length in Sec.~\ref{sec:convergence-analysis} when we provide numerical tests, our proposed operator for endomorphism-valued forms satisfies our initial goals of proper accuracy order and consistency with the vector-valued case (so that we obtain a discrete algebraic Bianchi identity). However, it is \emph{not} the only definition with these properties: while we average over all possible permutations, one could pick only a subset of permutations (which, in practice, improves the efficiency of the computations). In this case, not all combinatorial properties will be fulfilled. However, the improved order of convergence in the limit can still be maintained.

			\subsection{Revisiting Discrete Curvature}
			\label{sec:revisitCurv}
			The earlier definition of the discrete curvature two-form $\boldsymbol{\Omega^\nabla}$ in Sec.~\ref{sec:discreteCurv2Form} was \emph{postulated} directly from a discrete connection $\boldsymbol{\nabla}$ (and its associated connection one-form $\boldsymbol{\omega}$, see Sec.~\ref{sec:discrete1Form}) as an early example of discrete (1,1)-tensor-valued form: it was obtained by extending the notion of holonomy-based curvature by measuring the difference of parallel transport on two different paths between a ``cut'' and an ``evaluation'' fiber. Now that the link between discrete bundle-valued forms and continuous ones through PPF-dependent integration has been established, we revisit this key notion to show that it mirrors the continuous definition of the curvature two-form after a specific contraction-based integration.

			\paragraph{Continuous case.}
			We consider a \emph{contractible} two-dimensional region $C$ (assumed, without loss of generality, to be a half-disk of a smooth manifold $M$ and a vector bundle $ \pi :E \!\rightarrow\! M$ with connection $ \nabla $. Given a local frame field $\{f_a\}$, the connection can be expressed with the local connection $1-$form $ \omega $, and the curvature $2-$form is locally given as $\Omega^\nabla = d\omega + \omega\wedge \omega$,  see Eq.~\eqref{eq:Omega-local}. Let us select two points, $v$ and $w$, on the boundary $\partial C$ of $C$. The point $v$ corresponds to the evaluation vertex discussed above in the discrete case, while $w$ corresponds to the cut vertex. We denote by $\gamma_{v,w}\!\subset\! \partial C$ the path joining $v$ and $w$ along the boundary of $C$ with positive orientation, and by $\gamma_{v,w}'$ the corresponding 
			boundary path with negative orientation,
			\begin{wrapfigure}[6]{r}{0.25\textwidth}
				\vspace*{-4.5mm}\hspace*{-3mm}
				\includegraphics[width=1\linewidth]{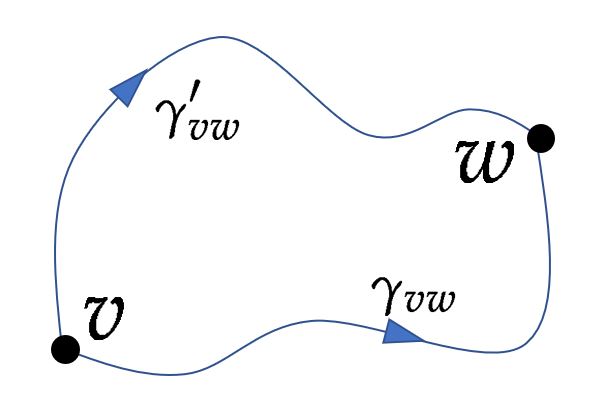}
			\end{wrapfigure}
			such that $\gamma_{v,w} - \gamma_{v,w}'\!=\!\partial C$ (see inset) 
			We proceed by generating a new frame field, denoted as $\{\tilde{f}_a\}$, through parallel transport. The initial frame, $f_a(v)$, situated at vertex $v$, is transported along two distinct paths: $\gamma_{v,w}$ and $\gamma_{v,w}'$. The discrepancy between the frames, which results at vertex $w$ after the parallel transports, is intricately linked to the holonomy. Extending this frame construction to the interior of $C$ involves a retraction corresponding to a 2D polar $(\rho, \theta)$ parametrization of the interior of $C$: in this parametrization, $\mathbf{e}_\rho$ aligns consistently with the retracted boundary, as depicted in Fig.~\ref{fig:retraction}, with vertex $w$ fixed. From this parametrization, we can construct a new frame field by parallel transporting the boundary frames along the $\mathbf{e}_y$ direction towards the interior of $C$. This process forms a new frame field, $\smash{\tilde f_bR^b_a=f_a}$, where $R$ denotes the rotation field responsible for adapting vector components from the local frame field $\{f_a\}$ into the new frame field $\{\tilde f_a\}$.
			
			It now holds for the connection $1-$form $\Tilde{\omega}$ in this new frame field that
			$\Tilde{\omega} \!=\! R\omega R^{-1} \!-\! dR R^{-1}.$
			Consequently, for any point $p\!\in\! C\!\setminus\! w$ in the punctured region, one has 
			$\Tilde{\omega}_p(e_\rho) \!=\! 0.$ 
			The curvature form $\Tilde{\Omega}^\nabla$ in the new frame field becomes $\Tilde{\Omega}^\nabla \!=\! R\Omega R^{-1}$ (see Sec.~\ref{sec:continuousWorld}). 
            With the components expressed in $\Tilde{F}$, one has
			\[\widetilde{\Omega}^\nabla = d \Tilde{\omega} + \Tilde{\omega}\wedge\Tilde{\omega}.\] 
			However in the punctured region $C\setminus w$, $\Tilde{\omega}\wedge\Tilde{\omega}\equiv 0,$ since 
			$\Tilde{\omega}\wedge\Tilde{\omega}(e_\rho,e_\theta) = \Tilde{\omega}(e_\rho)\Tilde{\omega}(e_\theta) - \Tilde{\omega}(e_\theta)\Tilde{\omega}(e_\rho) = 0$ except at the singularity $w$. Thus, the integral of the curvature becomes:
			\begin{equation}
				\int_{C}R {\Omega}^{\!\nabla} R^{-1} =\int_{C}\widetilde{\Omega}^{\!\nabla} = \int_{C} {d}\Tilde{\omega} + \Tilde{\omega}\wedge\Tilde{\omega} = \int_{\partial C}\Tilde{\omega}.\label{eq:curvature-onto-boundary}
			\end{equation}
			In other words, the fact that we parallel-propagated the frame $f_a(v)$ at $v$ through retraction proves that the mismatch at $w$ when we parallel transport the frame along the two sides of the boundary from $v$ (which was, in essence, the meaning of our discrete curvature evaluation) is exactly equal to this retraction-induced connection-dependent \emph{integral of the curvature $2-$form}.
			\begin{figure}[h] \vspace*{-2mm}
				\centering
				\includegraphics[width = 0.7\linewidth]{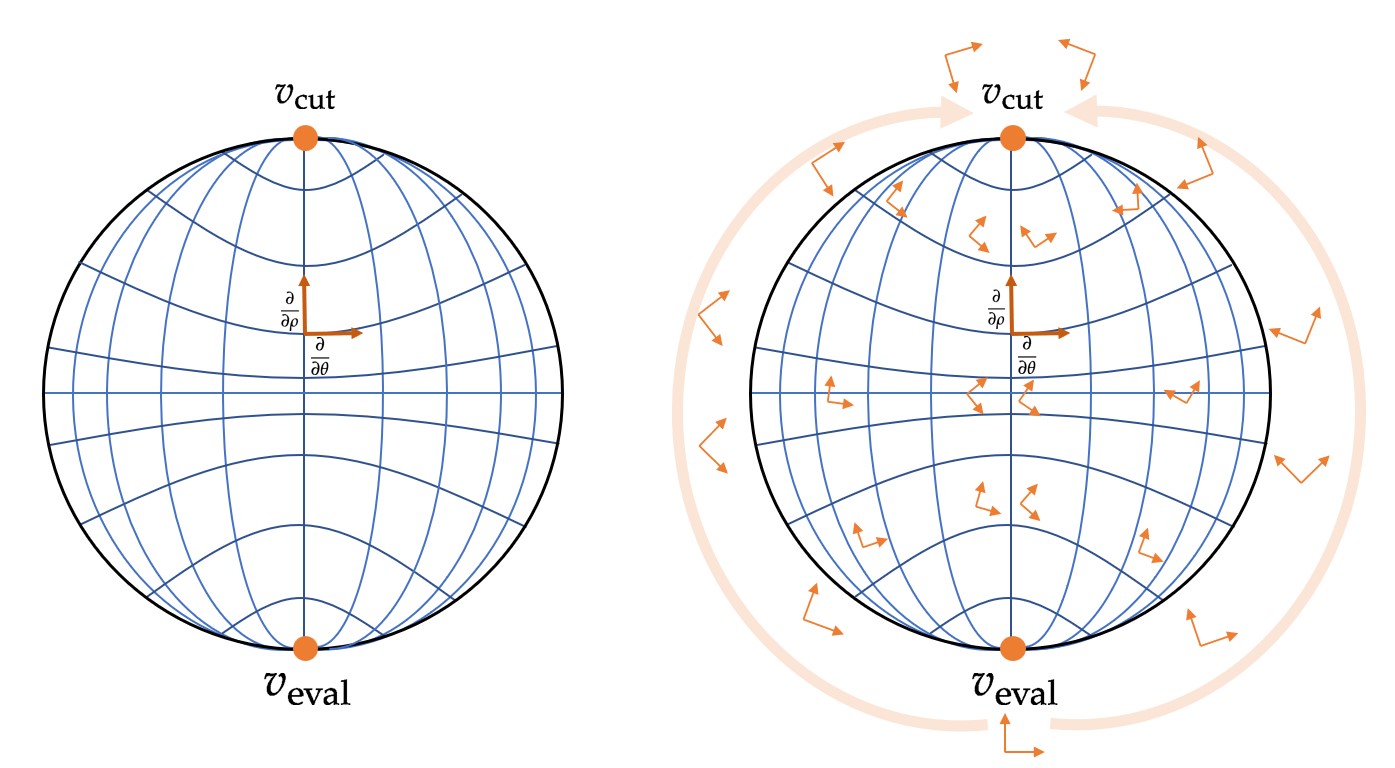}
                \vspace*{-2mm}
				\caption{\textbf{Retraction-based parameterization.} (Left): a retraction of the disk defines a $(\rho,\theta)$ (polar) parameterization of the (punctured) disk, where the $\rho$ direction is along the retraction direction, while the $\theta$ direction corresponds to a transversal direction. (Right): Starting from a parallel-transported frame field $F_\text{eval}$ from $v_\text{eval}$ along the boundary, we can parallel-propagate this boundary frame field towards the interior of the punctured disk along the curves corresponding to the $y$ direction of this retraction-based parameterization. The curvature integral simplifies to the discrepancy between frames at \( v_{\text{cut}} \), transported from \( v_{\text{eval}} \) via the retraction, and evaluated in \( \operatorname{Hom} (\bE_{v_\text{cut}},\bE_{v_\text{eval}})\).
					\vspace*{-2mm}}
				\label{fig:retraction}
			\end{figure}
			
			\begin{figure}[!htb]
				\centering
				\includegraphics[width = 0.7\linewidth]{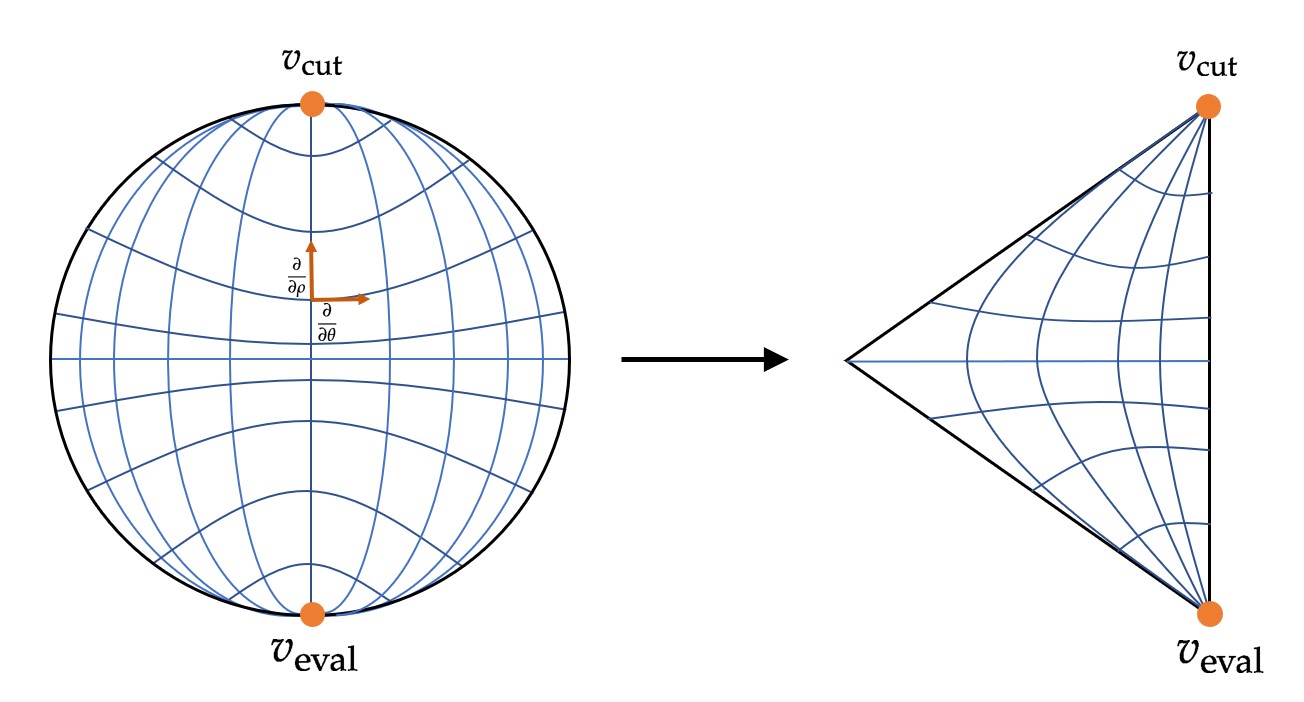}
    \caption{\textbf{Disk to Triangle.} A parallel-propagated frame over a disc induces a parameterization within the triangle once a homeomorphism to transform the disc into a triangle is defined. We illustrate the resulting parallel-transport paths as images under the homeomorphism map. The integrated curvature in this frame field precisely corresponds to the disparity in parallel transport along the two distinct boundary paths connecting the cut and evaluation fibers.}
				\label{fig:disc-to-triangle}
			\end{figure}
			
			\paragraph{Consequences for the discrete curvature two form.}
			
			The interpretation of the frame mismatch described above brings a series of consequences. First and foremost, it puts our definition of the two-prong discrete curvature $2-$form sketched in Sec.~\ref{sec:discreteCurv2Form} on solid footing as the continuous argument above remains valid on a $2-$simplex $\sigma\!=\![v,u,w]$. Additionally, it also justifies the summability of this definition, in stark contrast to a holonomy-based notion of curvature: if one considers a second simplex $\sigma'\!=\![w,z,v]$ sharing the edge ${v,w}$, then the discrete curvature $\boldsymbol{\Omega^{\!\nabla}}(\sigma, v,w)$ and the discrete curvature $\boldsymbol{\Omega^{\!\nabla}}(\sigma', v,w)$ can be summed to become the discrete non-simplicial curvature $\boldsymbol{\Omega^{\!\nabla}}(\sigma\union\sigma',v,w)$ since the two integrals sum trivially: the two canonical simplicial retractions form a retraction over their union, which means that the induced parallel-propagated frame field on each simplex
			corresponds to the frame obtained by applying the aforementioned construction for the non-simplicial union cell (see  Fig.~\ref{fig:sum-curvature-two-chains}). Consequently, the discrete notion of two-prong curvature applies as is for non-simplicial two-dimensional contractible cells, and the sum of two curvatures over a region joined by a discrete edge path delimited by the evaluation vertex and the cut vertex happens naturally as the parallel transports on each side of the common edge path cancel out. Finally, the interpretation of the two-prong discrete curvature as the discrete equivalent of a connection-based integral of the continuous form implies that $\boldsymbol{\Omega^{\!\nabla}}$ at an evaluation vertex $v$ of a simplex $\sigma$ of diameter $h$ is equivalent to the pointwise evaluation of the continuous curvature $2-$form at $v$ (on the two unit tangent vectors along the two outgoing edges at $v$), up to $\mathcal{O}(h^3)$.
			More precisely, for a simplex $\sigma = [v_{\mathrm{eval}},w,v_{\mathrm{cut}}]$ with evaluation in $E_{v_{\mathrm{eval}}}$ and cut in $E_{v_{\mathrm{cut}}}$ we obtain that the curvature integral in the novel frame field reduces to the mismatch of the frame observed at $v_\text{eval}$, expressed in $F_\text{cut}$, i.e 
			\[\int_{\sigma} R \Omega^\nabla R^{-1} = \int_{\partial \sigma}\Tilde{\omega} = R_{v_{\mathrm{eval}}{w}}R_{w v_\mathrm{cut}} - R_{v_{\mathrm{eval}},v_{\mathrm{cut}} }.\]
			By construction it holds for the gauge field at the evaluation vertex that  $R_{v_{\mathrm{eval}}} \!=\! \mathrm{Id}$. Thus, if we expand the gauge field $R$ at $v_{\mathrm{eval}}$ (or the respective Lie algebra element) we obtain that 
			\[\int_{\sigma} R \Omega^\nabla R^{-1} = \int_{\sigma} \Omega^\nabla + \mathcal{O}(h^3).\] Additionally, sampling the curvature at $v_{\mathrm{eval}}\!\in\! \sigma$ yields
			\[\int_{\sigma} \Omega^\nabla = \Omega^\nabla_{v_{\mathrm{eval}}}(v_{\mathrm{cut}}-v_{\mathrm{eval}},w-v_{\mathrm{eval}}) + \mathcal{O}(h^3)\] and therefore
			\begin{align}
				\label{eq:discrete-and-smooth-curvature}
				\int_{\sigma} R\Omega^\nabla R^{-1} &= R_{v_{\mathrm{eval}}w}R_{wv_{\mathrm{cut}}} - R_{v_{\mathrm{eval}}v_{\mathrm{cut}}}\notag \\ &= \boldsymbol{\Omega^\nabla}(\sigma,v_{\mathrm{eval}},v_{\mathrm{cut}})= \Omega^\nabla_{v_{\mathrm{eval}}}(v_{\mathrm{eval}}-v_{\mathrm{cut}},v_{\mathrm{eval}}-w)  + \mathcal{O}(h^3).
			\end{align}
			We will confirm in our section describing our numerical tests that the discrete curvature, with subsequent pre- and post-composition, converges indeed to the integral of the curvature $2-$form in a center-based parallel-propagated frame field with an error decay of $\mathcal{O}(h^4)$.
			
			\begin{figure}[t]
				\centering
				
                \includegraphics[width = 0.7\linewidth]{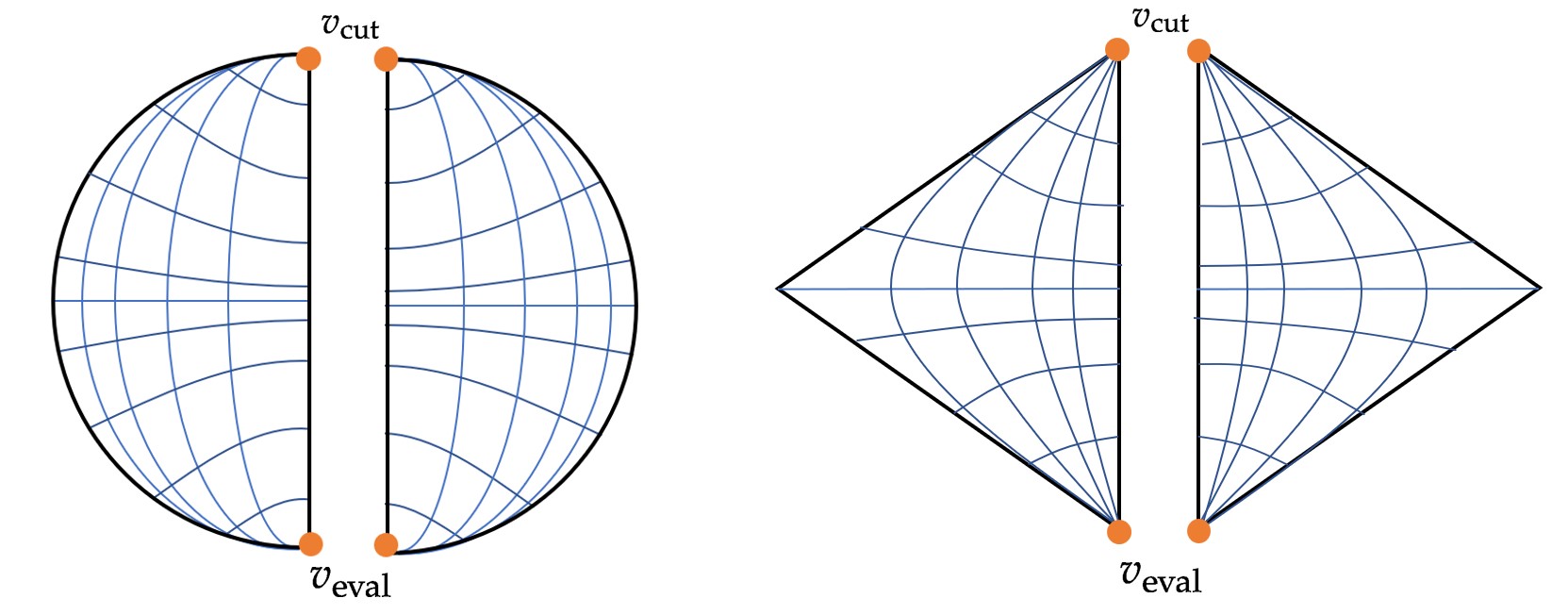}\vspace*{-1mm}
				\caption{\textbf{Summing Curvatures.} For two triangles using the same two prongs ($v_{\mathrm{cut}},v_{\mathrm{eval}}$), the curvature evaluated over their union will be the sum of the curvatures with the same two prongs because the parallel-propagated frames derived from these two prongs properly align. \vspace*{-3mm}}
				\label{fig:sum-curvature-two-chains}
			\end{figure}			
			
			\paragraph{PPF-induced Discrete Covariant Exterior Derivative for the Connection One Form.}
			When computing the discrete covariant exterior derivative of a vector-valued form, only an evaluation fiber suffices for the definition of a parallel-propagated frame field. However, when dealing with the exterior covariant derivative of the connection $1-$form, it becomes necessary to specify both a cut and evaluation fibers since the result is a curvature, i.e., a $(1,1)-$tensor in this particular case.
			Upon specifying these parameters, we obtain an induced parallel propagated frame field, as illustrated in Fig.~\ref{fig:retraction}. This reduction leads to the boundary integral:
			\[
			\int_{C}R {\Omega}^{\!\nabla} R^{-1} =\int_{C}\widetilde{\Omega}^{\!\nabla} = \int_{C} {d}\Tilde{\omega} + \Tilde{\omega}\wedge\Tilde{\omega} = \int_{\partial C}\Tilde{\omega}.
			\]
			
			To evaluate the Lie algebra integral, we can treat the columns of the connection $1-$form as three independent vector valued $1-$forms. This motivates the definition of the PPF-induced exterior covariant derivative, evaluated at $v_{\mathrm{eval}}$ and cut at $v_{\mathrm{cut}}$:
			\[
			\mathfrak{d}^{\boldsymbol{\nabla}}\boldsymbol{\omega}([v_{\mathrm{eval}},w,v_{\mathrm{cut}}],v_{\mathrm{eval}},v_{\mathrm{cut}}) := \boldsymbol{\omega}_{v_{\mathrm{eval}},w} + \mathcal{R}_{w,v_{\mathrm{eval}}}\boldsymbol{\omega}_{w,v_{\mathrm{cut}}} - \boldsymbol{\omega}_{v_{\mathrm{eval}},v_{\mathrm{cut}}},
			\]
			for a given discrete connection $ \boldsymbol{\nabla}$ with discrete connection form $ \boldsymbol{\omega }$.

			Note that this PPF-induced discrete covariant exterior derivative shares a resemblance with the vector-valued one, akin to the smooth case: by regarding the matrix of the discrete connection $1-$form as a set of $r$ vector-valued differential forms, the PPF-induced discrete covariant exterior derivative manifests as a vector-valued discrete exterior covariant derivative on these vector components. The case of the connection $1-$form is, however, special: the covariant exterior derivative of the connection $1-$form is not a Lie-algebra valued quantity anymore, but a $(1,1)-$tensor valued $2-$form instead, necessitating in the discrete case the addition of a choice of ``cut fiber''.
			Using the definition of the discrete connection $1-$form, see Eq.~\eqref{eq:discOneForm}, we obtain that indeed:
			\begin{align}
				\!\!\!\!&\mathfrak{d}^{\boldsymbol{\nabla}}\boldsymbol{\omega}([v_{\mathrm{eval}},w,v_{\mathrm{cut}}],v_{\mathrm{eval}},v_{\mathrm{cut}})\\
    &= (\mathcal{R}_{v_{\mathrm{eval}},w} - \mathrm{Id}) +\mathcal{R}_{v_{\mathrm{eval}},w} (\mathcal{R}_{w,v_{\mathrm{cut}}} - \mathrm{Id}) - (\mathcal{R}_{v_{\mathrm{eval}},v_{\mathrm{cut}}} - \mathrm{Id}) \nonumber \\
				&=\mathcal{R}_{v_\mathrm{eval},w}\mathcal{R}_{w,\mathrm{cut}} - \mathcal{R}_{{v_{\mathrm{eval}}},v_{\mathrm{cut}} } = \boldsymbol{\Omega^\nabla}([v_{\mathrm{eval}},w,v_{\mathrm{cut}}],v_{\mathrm{eval}},v_{\mathrm{cut}}),
				\label{eq:frakomega=Omega}
			\end{align}
			justifying, a posteriori, our definition of discrete curvature.

	\section{Properties and Convergence of our Discrete Operators}\label{sec:PropAndConv}

In the previous sections, we introduced a discrete calculus for discrete bundle-valued forms in a principled manner, based on parallel-propagated frames. We now shift our focus towards discussing the resulting properties of our discrete operators, before analyzing their convergence under refinement. We will then provide quantitative evidence through numerical tests to verify our claims, both in terms of properties and of convergence order. 

\subsection{Properties of the Discrete Exterior Derivative}

The discrete covariant exterior derivative $\bolddnab$ introduced in Def.~\ref{def:ext_der_alternation_group} has several intrinsic properties, and as it differs from the operator introduced by~\cite{Hirani_Bianchi} mostly through a post alternation, it also inherits combinatorial properties proven in~\cite{Hirani_Bianchi}. We review them next.

\paragraph{Naturality of $\bolddnab$.} 
As mentioned in Eq.~\eqref{eq:pullback-commutes-with-d-nabla-smooth}, it is an important result in the theory of smooth exterior calculus that the pullback commutes with the exterior derivative. For the discrete operator $\bolddnab$, a similar result holds true due to a lemma proven in~\cite{Hirani_Bianchi}.

\begin{lemma}[Pullback commutes with $\mathfrak{d}^{\boldsymbol{\nabla}}$, from~\cite{Hirani_Bianchi}]\label{pull_back_lemma}
	Given a simplicial map $f\colon M\!\to\!N$ between two discrete manifolds $M$ and $N$, and a discrete vector bundle with connection $(\bE,\boldsymbol{\nabla})$ over $N$, let $\boldsymbol{\alpha}$ be an $\bE-$valued discrete $\ell-$form. For a $(\ell\!+\!1)-$simplex $\sigma = [v_0,\ldots,v_{\ell+1}]$ of $M$, one has:
	\[ f^\ast(\mathfrak{d}^{\boldsymbol{\nabla}}\boldsymbol{\alpha})(\sigma,{v_0}) =\mathfrak{d}^{f^\ast\boldsymbol{\nabla}} f^\ast\boldsymbol{\alpha}(f(\sigma),{f(v_0)})\]  	
\end{lemma}

\begin{corollary}[Naturality of $\bolddnab$]
	Given a simplicial map $f\colon M\!\to\!N$ between two discrete manifolds $M$ and $N$, and a discrete vector bundle with connection $(\bE,\boldsymbol{\nabla})$ over $N$, let $\alpha$ be an $\bE-$valued discrete $\ell-$form.
	For a $(\ell\!+\!1)-$simplex $\sigma = [v_0,\ldots,v_{\ell+1}]$ of $M$, one has: 
	\[ f^\ast(\bolddnab\alpha)(\sigma,{v_0}) =  \boldsymbol{d}^{f^\ast\boldsymbol{\nabla}} f^\ast\alpha(f(\sigma),{f(v_0)})\]
\end{corollary}
\begin{proof}
	This is a direct consequence of the previous lemma and the fact that the pullback commutes with the alternation operator as mentioned in Rem.~\ref{rem:Alternation-not-a-projector}. 
\end{proof}

\paragraph{Antisymmetry of $\bolddnab\!$.}

In the scalar-valued case, discrete differential forms are skew-symmetric maps, in the sense that permuting the vertices of an evaluation simplex leads to multiplication with the sign of the permutation. In the bundle-valued case, discrete differential forms are defined as abstract maps mapping into a vertex fiber of a given cell. Therefore, our operator is antisymmetric \emph{as long as we keep the evaluation fiber fixed}, since any switch of two vertices will result in a sign flip in the evaluation of $\bolddnab$ due to the alternation within its definition.

\paragraph{Differential Bianchi identity.} The celebrated second Bianchi identity reviewed in Eq.~\eqref{eq:diffBianchi} holds exactly in the discrete setting. 

\begin{proposition}[Differential Bianchi identity]
	Let $\bE$ be a discrete vector bundle with connection $\boldsymbol{\nabla}$ over a discrete manifold $M$. For any $3-$simplex  $\sigma= [v_0,v_1,v_2,v_3]\subset M$. one has:
	\begin{equation}
		\label{eq:differential-Bianchi-Identity}
		\bolddnab\boldsymbol{\Omega^\nabla}(\sigma,{v_0},{v_3}) = 0
	\end{equation}
\end{proposition}
\begin{proof}
	As proven in~\cite{Hirani_Bianchi}, we know that $\mathfrak{d}^{\boldsymbol{\nabla}}\boldsymbol{\Omega^\nabla}(\sigma,{v_0},{v_3}) = 0$. Therefore,\vspace*{-1mm} 
	\[\bolddnab\boldsymbol{\Omega^\nabla}(\sigma,{v_0},{v_3}) = \mathrm{Alt}^{\boldsymbol{\nabla}}\mathfrak{d}^{\boldsymbol{\nabla}}\boldsymbol{\Omega^\nabla}(\sigma,{v_0},{v_{3}}) = 0. \]
\end{proof}

\paragraph{Algebraic Bianchi identity.}
The first Bianchi identity, reviewed in Eq.~\eqref{eq:algebraic-bianchi-identity}, also holds in our discrete calculus. While~\cite{Hirani_Bianchi} showed that their covariant exterior derivative satisfies
$$\mathfrak{d}^{\boldsymbol{\nabla}}\mathfrak{d}^{\boldsymbol{\nabla}}\boldsymbol{\alpha}([v_0,\ldots,v_{\ell+2}],v_0) \!=\! \boldsymbol{\Omega^\nabla}([v_0,v_1,v_2],v_0,v_2)\ \boldsymbol{\alpha}([v_2,\ldots,v_{\ell+2}],v_2)$$ for any vector-valued form $\boldsymbol{\alpha}$, our discrete covariant exterior derivative satisfies the algebraic Bianchi identity for an implicitly-defined combinatorial wedge product.

\begin{theoremanddefinition}[Algebraic Bianchi Identity]
	\label{thm:algebraic-bianchi-identity}
	Let $M$ be a discrete manifold with a discrete bundle $\bE$ and connection $\boldsymbol{\nabla}$. Further let $\boldsymbol{\alpha}$ be a discrete $\bE-$valued $\ell-$form and $\sigma=[v_0,\ldots,v_{\ell+2}]$ a $(\ell+2)-$simplex. In this case there exists a set $K\subseteq C^2( \sigma )\times C^\ell(\sigma)$ consisting of $2-$ and $\ell-$subcells of $\sigma$ such that for all $(m,\kappa)\in K$ with a shared vertex $w_{m,\kappa}$ we have  \vspace{-1mm}
	\begin{align*}
		\boldsymbol{d^\nabla d^\nabla\alpha}(\sigma,v_0)
		 &= \frac{1}{(\ell+3)! \, (\ell+2)!}\sum_{({m},\kappa)\in K} \boldsymbol{\Omega^\nabla}(f,v_0,w_{{m},\kappa}) \ \boldsymbol{\alpha}(\kappa,w_{{m},\kappa}) \\
		& =: \boldsymbol{\Omega^\nabla\wedge\alpha}(s,v_0),\vspace{-1mm}
	\end{align*}
 
	where the sum defines implicitly a \emph{wedge product} between the curvature $2-$form and $\boldsymbol{\alpha}$, reminiscent of scalar-based wedge products defined through cup products~\cite{Hirani2003}. Similarly, for any $(1,1)-$tensor-valued $\ell-$form $\boldsymbol{\beta}$, two consecutive applications of our discrete covariant exterior derivative $\boldsymbol{d^\nabla}$ yield implicitly a discrete commutator:
	\[
	\boldsymbol{d^\nabla d^\nabla\beta} = \boldsymbol{[\Omega^\nabla\wedge\beta]}.
	\]
\end{theoremanddefinition}
\begin{proof}
	For discrete $(1,0)-$tensor valued forms, see~\ref{app:proof-restricted-skew-symmetry}. The case of discrete $(1,1)-$tensor valued forms is similar --- but for an evaluation fiber $\smash{\bE_{v_0}}$ and a cut fiber $\bE_{v_{\ell+2}}$ one has: \vspace{-1mm}
	\begin{align*}
		\mathfrak{d}^{\boldsymbol{\nabla}}\mathfrak{d}^{\boldsymbol{\nabla}}\boldsymbol{\beta}([v_0,\ldots,v_{\ell+2}],v_0,v_{\ell+2})
		&= \boldsymbol{\Omega^\nabla}([v_0,v_1,v_2],v_0,v_2)\boldsymbol{\beta}([v_2,\ldots,v_{\ell+2}],v_2,v_{\ell+2})\\
  & \quad-\boldsymbol{\beta}([v_0,\ldots,v_{\ell}],v_0,v_\ell)\boldsymbol{\Omega^\nabla}([v_{\ell},v_{\ell+1},v_{\ell+2}],v_\ell,v_{\ell+2}),
	\end{align*}
	hence the resulting commutator.
\end{proof}
\begin{remark}
	In the realm of scalar-valued DEC, Kotiuga~\cite{Kotiuga:Limits} established theoretical limitations regarding the definition of a discrete scalar-valued wedge product. However, associativity of the wedge product is not an issue in our context, as there is no inherent product structure on the bundle; and graded anti-commutativity is irrelevant given the non-commutative nature of the underlying product for $(1,1)-$tensors, and vectors. We will demonstrate, in Sec.~\ref{sec:numVerif} that our discrete wedge product converges under refinement toward the smooth wedge product when applied to the proper discretization of a smooth form --- unlike the definition given in~\cite{Hirani_Bianchi} as they did not properly relate discrete and continuous bundle-valued forms, see Fig.~\ref{fig:different_realizations_torsion}.
\end{remark}

\begin{remark}
	\label{rem:curvature-properties-discrete}
	Given a connection $\boldsymbol{\nabla} \!=\! \{ \mathcal{R}_{ab}\}_{e_{ab}\in \mathcal{E}}$ and an arbitrary choice of frame field, the discrete connection $1-$form for the edge between vertices $a$ and $b$ is given by $\boldsymbol{\omega}_{ab} = {R}_{ab} - \mathrm{Id}$. 
	Furthermore, we saw in Eq.~\eqref{eq:frakomega=Omega} that
	\[\mathfrak{d}^{\boldsymbol{\nabla}}\boldsymbol{\omega}([abc],a,c)  
	=  \boldsymbol{\Omega^\nabla}([abc],a,c).\]
	Hence, with our alternation operator for $(1,1)-$tensor based forms and the stability of $\boldsymbol{\Omega^\nabla}$ under the $\operatorname{Alt}$ operator (see Rem.~\ref{rem:Omega-invariant-alt}), it holds that
	\begin{equation}
		\bolddnab\boldsymbol{\omega}([abc],a,c) = \boldsymbol{\Omega^\nabla}([abc],a,c), \label{eq:d-omega-becomes-Omega}
	\end{equation}
	in an exact sense. However it should be noted that the use of $\bolddnab$ above is slightly abusive: it is neither the discrete covariant exterior derivative of a $(1,0)-$tensor valued form, nor of a $(1,1)-$tensor valued form. Instead, it is made out of the composition of a PPF-induced covariant exterior derivative for $(1,0)-$tensor valued form applied to $\omega$, composed with the alternation operator for $(1,1)-$tensor valued forms. This \emph{hybrid} operator is specific to the connection $1-$form and reflects the special and local nature of the curvature form in the continuous case. 
	
	We should also mention here that our linearized version of the matrix logarithm used in the definition of the discrete connection $1-$form is crucial to ensure that this property holds (see Rem.~\ref{rem:omega-or-R-exact}). Finally, we proved that $\boldsymbol{d^\nabla d^\nabla} = \boldsymbol{\Omega^\nabla\wedge}$ for any vector-valued form, and one can check that indeed, for a discrete vector-valued $0-$form $\mathbf{z}$, we have:
	\begin{align*}
		\bolddnab \bolddnab \mathbf{z} ([abc], a) 
		&= \frac{1}{6}\Bigl((R_{ab}R_{bc} - R_{ac})z_c + R_{ab}(R_{bc}R_{ca}-R_{ba})z_{a} + R_{ac}(R_{ca}R_{ab}-R_{cb})z_b\\
		&\quad- (R_{ac}R_{cb}-R_{ab})z_b - R_{ab}(R_{ba}R_{ac}-R_{bc})z_c - R_{ac}(R_{cb}R_{ba}-R_{ca})z_a \Bigr).
	\end{align*}
	Note that this wedge product between the curvature $2-$ form and a vector-valued $0-$form does \emph{not} correspond to a simple pointwise product as in the continuous case: this property will only be valid in the limit of mesh refinement. 
\end{remark}

\subsection{Convergence analysis}
\label{sec:convergence-analysis}

In Sec.~\ref{sec:ppfIntro}, we introduced the use of parallel-propagated frame fields to simplify the high-order terms of the integral of the discrete covariant exterior derivative. In this section, we discuss how the order of approximation is affected by the location of the \emph{origin} of the PPF being used: indeed, if we integrate the exterior derivative in a center-based parallel-propagated frame field, we achieve an even higher accuracy order than using a corner-based PPF.\@
\begin{figure}[!htb]\vspace*{-6mm}
	\centering 
	\includegraphics[width=0.8\linewidth]{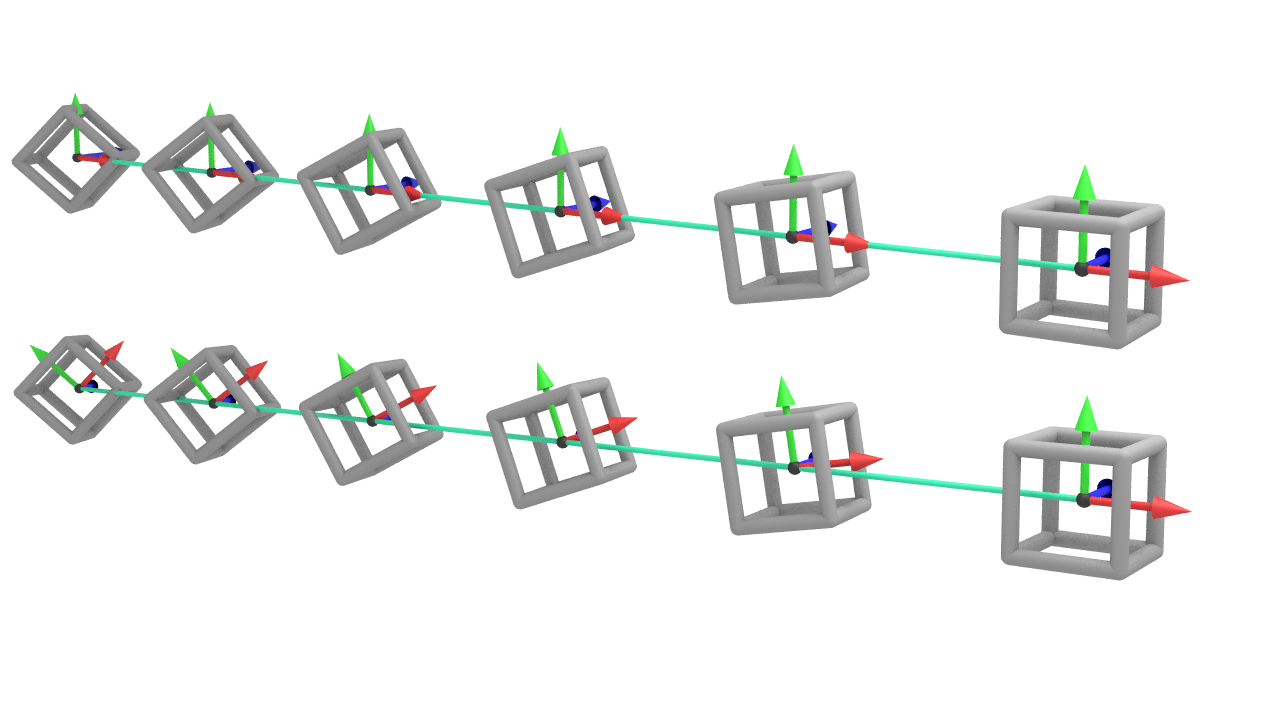}
	\vspace*{-12mm}
	\caption{\textbf{Parallel-propagated frames.} Given a connection, we can use the parallel propagated frame of $T_pM$ (bottom) for integration; integration in some non-parallel-propagated/non-geometric frame (like the one on top) does not lead to convergence.}
	\label{fig:compare_moving_frames}
\end{figure}

\begin{theorem}[Improved accuracy with center-based PPF]
	\label{thm:even-better-reduction-with-center-based-PPF}
	Let $ \pi :E \rightarrow M$ be a vector bundle with connection $\nabla$. Let $s=[v_0,\ldots,v_{\ell}]\!\subset\! M$ be a region in $M$ for which there exists a diffeomorphism to a $\ell-$simplex $\sigma$ where each point $v_i$ is mapped to an associated vertex $w_i$ of $\sigma$, and let us  call the ``center'' $c_s\!\in\!s$ the point of $s$ being mapped to the barycenter of $\sigma$. 
    For $\alpha\in \Omega^\ell(M,E)$ it holds that 
	\[\int_{s}{(d^\nabla\alpha)}^{\nabla,c_s} = \int_{\partial s} \alpha^{\nabla,c_s} + \mathcal{O}(h^{\ell+3})\]
	(compare with Eq.~\eqref{eq:order-d-nabla}).
\end{theorem}
\begin{proof}
	It holds in the center-based PPF that
	\[\int_{s}(d^\nabla\alpha)^{\nabla,c_s} = \int_{\partial s} \alpha^{\nabla,c_s} + \int_{s} \omega^{\nabla,c_s}\wedge\alpha^{\nabla,c_s}.\]
	Additionally, one has \[\int_{s} \omega^{\nabla,c_s}\wedge\alpha^{{\nabla,c_s}} = \mathcal{O}(h^{\ell+3})\] since the integral of the linear part cancels out when the center verifies $\omega^{\nabla,{c_s}}(c_s)=0$. 
	To see this, we consider a Taylor expansion of $\omega^{\nabla,{c_s}}$ and $\alpha$. It holds for  $p\in s$
	\[\left(\omega^{\nabla,c_s}\right)_p =\underbrace{\left(\omega^{\nabla,c_s}\right)_{c_s}}_{ = 0} +\left(\omega^{\nabla,c_s}\right)'_{c_s} (p - c_s) + \mathcal{O}(h^2),\]
	where of $(\omega^{\nabla,c_s})'$ is to be understood as a derivative applied to the components of $\omega^{\nabla,c_s}$. Assuming sufficient smoothness, we get that $(\omega^{\nabla,c_s})'$ is still a smooth $1-$form. Similarly, we obtain
	\[\alpha^{\nabla,c_s}_p = \alpha^{\nabla,c_s}_{c_s} + (\alpha^{\nabla,c_s})'_{c_s} (p - c_s) + \mathcal{O}(h^{\ell+2}).\]
	Hence
	\begin{align*}
    \norm{\int_{s} \omega^{\nabla,c_s}\wedge\alpha^{\nabla,c_s}}&\leq \norm{\int_{p\in s}((\omega^{\nabla,c_s})_{c_s}'\wedge \alpha^{\nabla,c_s}_{c_s})(c_s - p) dp }\\
    &\ + \underbrace{\norm{\int_{s}((\omega^{\nabla,c_s})'_{c_s}\wedge(\alpha^{\nabla,c_s})'_{c_s})}}_{=\mathcal{O}(h^{\ell+1})} \cdot \mathcal{O}(h^{2}) + \mathcal{O}(h^{\ell+3}).
    \end{align*} 
	Since $(\omega^{\nabla,c_s})_{c_s}'\wedge \alpha^{\nabla,c_s}_{c_s}$ is constant over $s$, we get that
    
	\begin{equation}
		\label{eq:better-cancellation-with-ppf-center-of-mass}
		\int_{p\in s}((\omega^{\nabla,c_s})_{c_s}'\wedge (\alpha^{\nabla,c_s})_{c_s})(c_s - p) dp = 0, 
	\end{equation}
	leaving us with $\norm{\int_{s} \omega^{\nabla,c_s}\wedge\alpha^{\nabla,c_s}}= \mathcal{O}(h^{\ell+3})$.
\end{proof}

\begin{theorem}
	\label{thm:ppf-derivative-linear-decay}
	Let $ \pi :E \rightarrow M$ be a vector bundle with connection $\nabla$. Let $s=[v_0,\ldots,v_{\ell+1}]\!\subset\! M$ be a region in $M$ for which there exists a diffeomorphism to a $(\ell+1)-$simplex $\sigma$ where each point $v_i$ are mapped to an associated vertex $w_i$ of $\sigma$. Further let $\alpha\in \Omega^\ell(M,E)$ and $\boldsymbol{\alpha}$ be its discretization defined in Eq.~\eqref{eq:vertex-based-discretization}. In this case we have 
	\begin{equation}
		\label{eq:error-to-estimate-frak-d}
		R_{v_0,c_s}\int_{s}(d^\nabla\alpha)^{\nabla,{c_s}} - \mathfrak{d}^{\boldsymbol{\nabla}}\boldsymbol{\alpha}(s,{v_0}) = \mathcal{O}(h^{\ell+2}),
	\end{equation}
	or expressed intrinsically,
	\begin{equation}
		\mathcal{R} _{v_0,c_s}\;\ \connintidx{\varphi_{c_s}}_s d^\nabla\alpha\;- \mathfrak{d}^{\boldsymbol{\nabla}}\boldsymbol{\alpha}(s,{v_0}) = \mathcal{O}(h^{\ell+2}).
	\end{equation}
\end{theorem}
\begin{proof} We use the identifications stated in Remark~\ref{identification}.
	It holds in the center-based parallel propagated frame that 
	\begin{equation}
		\label{eq:ppf-to-boundary}
		R_{v_0,c_s}\int_{s} (d^\nabla\alpha)^{\nabla,{c_s}} = R_{v_0,c_s}\int_{\partial s} R_{c_s}\alpha + R_{v_0,c_s}\int_{s}{\omega}^{\nabla,c_s}\wedge R_{c_s}\alpha.
	\end{equation}
	
	\noindent As discussed above, we know 
	\[{\int_{s}{\omega}^{\nabla,c_s}\wedge\alpha^{\nabla,c_s}} = \mathcal{O}(h^{\ell+3}).\]
	In the following we denote $s_i = [v_0,\ldots,\hat{v_i},\ldots,v_{\ell+1}]$ the $\ell-$subsimplex of $s$ without $v_i$. By $m_i$ we will denote the $\ell-$chain opposite to $v_i$ with \emph{positive} orientation relative to that of $s,$ i.e., $m_i=(-1)^i s_i$. Extending the discrete $\boldsymbol{\alpha}$ to $m_i$ by $\boldsymbol{\alpha}(m_i,v_0)=(-1)^i\ \boldsymbol{\alpha}(s_i,v_0),$ we have 
	\[\mathfrak{d}^{\boldsymbol{\nabla}}\boldsymbol{\alpha}(s,{v_0}) = R_{v_0,v_1}\ \boldsymbol{\alpha}(s_0,v_1) + \sum_{i =1}^{\ell+1} (-1)^i\ \boldsymbol{\alpha}(s_i,v_0)=R_{v_0,v_1}\ \boldsymbol{\alpha}(m_0,v_1) + \sum_{i =1}^{\ell+1} \ \boldsymbol{\alpha}(m_i,v_0).\]
	For all faces containing $v_0$, we have by definition 
	\[\boldsymbol{\alpha}(m_i,{v_0}) = \int_{m_i} R^{\nabla,v_0}\alpha .\]
	If we compare the contribution of the PPF-derivative to the contribution from the boundary integral in Eq.~\eqref{eq:ppf-to-boundary} coming from $m_0$, we obtain
	\begin{equation}
		\label{eq:estimation-for-face-with-v-0}
		R_{v_0,c_s}\int_{m_0} R^{\nabla,c_s}\alpha - R_{v_0,v_1}\int_{m_0} R^{\nabla,v_1}\alpha = \int_{m_0} (R_{v_0,c_s}R^{\nabla,c_s} - R_{v_0,v_1}R^{\nabla,v_1})\alpha.
	\end{equation}
	If we do a Taylor expansion for
	\[g\colon m_0\to \R^{r\times r}\colon p\mapsto (R_{v_0,c_s}R^{\nabla,c_s} - R_{v_0,v_1}R^{\nabla,v_1})(p) \]
	at $c_{m_0}$, we get
	\[g(c_{m_0})  =(\underbrace{R_{v_0,c_s} R_{c_s,c_{m_0}}}_{R_{v_0,c_{m_0}}} - R_{v_0,v_1}R_{v_1,c_{m_0}}) =  -\boldsymbol{\Omega^\nabla}([v_0,v_1,c_{m_0}],{v_0},{c_{m_0}}) + \mathcal{O}(h^3)\]
	and therefore, it yields \[g(p) = \boldsymbol{\Omega^\nabla}([v_0,v_1,c_{m_0}],{v_0},{c_{m_0}}) + \mathcal{O}(h^3)\] for $p\!\in\! m_0$.
		Now consider any face $m\!\in\! \partial s$ containing $v_0$ with positive orientation. From Sec.~\ref{sec:revisitCurv},
	\begin{equation}
		\label{eq:error-for-face-with-s_0}
		R_{v_0,c_s}\int_{m} R_{c_s}\alpha - R_{v_0,c_m}\int_{m} R_{c_m}\alpha = -\Omega_{v_0}^\nabla(v_1 - v_0,c_m - v_0)\alpha_{c_m}[m] + \mathcal{O}(h^{\ell+3}).
	\end{equation}
	However we have 
	\[\Omega_{v_0}^\nabla(v_1 - v_0,c_m - v_0)\alpha_{c_m}[m] = \mathcal{O}(h^{\ell+2}).\]
We obtain for the contribution of the error coming from $m$ in Eq.~\eqref{eq:error-to-estimate-frak-d},  again with a Taylor expansion at $c_{m}$,
	\begin{equation}
		\label{eq:estimation-error-face-opp-to-v-0}
		\int_{m} R_{v_0}\alpha - R_{v_0,c_s}\int_{m} R_{c_s}\alpha = \Omega_{v_0}^\nabla(c_m - v_0,c_s - v_0)\cdot\alpha_{c_{m}}[m] + \mathcal{O}(h^{\ell+3}).
	\end{equation}
	Since $\Omega_{v_0}^\nabla(c_m - v_0,c_s - v_0)\cdot\alpha_{c_{m}}[m]\! = \! \mathcal{O}(h^{\ell+2})$, this yields the claim.
\end{proof}

Since the volume of the simplicial cell is of order $\mathcal{O}(h^{\ell+1})$, this shows that the operator $\mathfrak{d}^{\boldsymbol{\nabla}}$ converges under refinement \emph{if the input is obtained through integration in an appropriate PPF.}
However, to ensure convergence of a second application of $\bolddnab$, this is not sufficient since the error after division with the volume of a $(\ell+1)-$simple can remain arbitrarily large. This is where the alternation can improve the accuracy order.
\begin{theorem}[Alternation increases accuracy]
	\label{thm:alternation-augments-accuracy}
	Let $\pi :E \rightarrow M$ be a vector bundle with connection $\nabla$. Let $s=[v_0,\ldots,v_{\ell+1}]\!\subset\! M$ be a region in $M$ for which there exists a diffeomorphism to a $(\ell+1)-$simplex $\sigma$ where each point $v_i$ is mapped to an associated vertex $w_i$ of $\sigma$. Further let $\alpha\in \Omega^\ell(M,E)$ and $\boldsymbol{\alpha}$ be its discretization defined in Eq.~\eqref{eq:vertex-based-discretization}. In this case, we have 
	\begin{equation}
		\label{eq:error-estimate-better-with-alternation}
		R_{v_0,c_s}\int_{s}(d^\nabla\alpha)^{(\nabla,c_s)}  =  \mathrm{Alt}^{\nabla}\mathfrak{d}^{\boldsymbol{\nabla}}\boldsymbol{\alpha}(s,{v_0}) + \mathcal{O}(h^{\ell+3}) = \bolddnab\boldsymbol{\alpha}(s,{v_0}) + \mathcal{O}(h^{\ell+3}),
	\end{equation}
	or expressed intrinsically,
	\begin{equation}
		\mathcal{R}_{v_0,c_s}\ \  \connintidx{\varphi_{c_s}}_{s} d^\nabla\alpha  =  \mathrm{Alt}^{\nabla}\mathfrak{d}^{\boldsymbol{\nabla}}\boldsymbol{\alpha}(s,{v_0}) + \mathcal{O}(h^{\ell+3}) = \bolddnab\boldsymbol{\alpha}(s,{v_0}) + \mathcal{O}(h^{\ell+3}).
	\end{equation}
\end{theorem}
\begin{proof}
	It holds by Theorem~\eqref{thm:even-better-reduction-with-center-based-PPF}
	\begin{equation}
		\label{eq:ext-der-onto-boundary-and-higher-order}
		\int_{s} R_{v_0,c_s} (d^\nabla\alpha)^{\nabla,c_s} = R_{v_0,c_s}\int_{\partial s } \alpha^{\nabla,c_s} + \mathcal{O}(h^{\ell+3}).
	\end{equation}
	Let $\{m_i\}\subset \partial s$ be the set of $\ell-$cells with positive orientation forming the boundary of $s$ where $v_i \!\notin\! m_i$. In the following, we will show that the contributions for every $\ell-$face coming from Eq.~\eqref{eq:ext-der-onto-boundary-and-higher-order} and those of $\bolddnab(s,v_0)$ differ by a term of order $\mathcal{O}(h^{\ell+3})$. To see this, we rewrite the exterior covariant derivative as

  \begin{align}\bolddnab\boldsymbol{\alpha}(s,v_0) &= \mathrm{Alt}^{\boldsymbol{\nabla}}\mathfrak{d}^{\boldsymbol{\nabla}}\boldsymbol{\alpha}(s,v_0)= \frac{1}{\ell+2}\sum_{j = 0}^{\ell+1} R_{v_0,v_j} \Bigl(\frac{1}{(\ell+1)!}\sum_{\substack{\pi\in S_{\ell+2}\\\pi(0) = j}} \mathrm{sgn}(\pi)\mathfrak{d}^{\boldsymbol{\nabla}}\boldsymbol{\alpha}(\pi(s),v_j) \Bigr).
		\label{eq:exterior-cov-derivative-splitted}
	\end{align}
  Let 
	\[
	\mathrm{Alt}^{\boldsymbol{\nabla}}_j\mathfrak{d}^{\boldsymbol{\nabla}}\boldsymbol{\alpha}(s,v_j) \coloneqq\frac{1}{(\ell+1)!} \sum_{\pi\in S_{\ell+2},\pi(0) = j} \mathrm{sgn}(\pi)\mathfrak{d}^{\boldsymbol{\nabla}}\boldsymbol{\alpha}(\pi(s),v_j)
	\]
	be the operator that uses an alternation on the face opposite to $v_j$ only.
	For the boundary cell $m_i$ we consider the contributions coming from $\mathrm{Alt}^{\boldsymbol{\nabla}}_j\mathfrak{d}^{\boldsymbol{\nabla}}\boldsymbol{\alpha}$ to Eq.~\eqref{eq:exterior-cov-derivative-splitted} for $i\neq j$ first. In this case the face $m_i$ is not opposite to $v_j$. Hence, its contributions turn into
	\[R_{v_0,v_j}\boldsymbol{\alpha}(m_i,v_j) = R_{v_0,v_j}\int_{m_i} R^{\nabla,v_j}\alpha.\]
	For the case $j \!=\! i$, $m_i$ is the face opposite to the evaluation vertex of $\mathrm{Alt}^{\boldsymbol{\nabla}}_j\mathfrak{d}^{\boldsymbol{\nabla}}\boldsymbol{\alpha}(s,v_j)$. In this case, we can rewrite
	\[\mathrm{Alt}^{\boldsymbol{\nabla}}_i\mathfrak{d}^{\boldsymbol{\nabla}}\boldsymbol{\alpha}(s,v_i) = \frac{1}{\ell+1}\sum_{v_k\in m_i} \frac{1}{\ell !}\sum_{\substack{\pi\in S_{\ell+2}\\ \pi(0)=j\\ \pi(1) = k}}\mathrm{sgn}(\pi) \mathfrak{d}^{\boldsymbol{\nabla}}\alpha(\pi(s),v_j).\] This yields that the contributions coming from $m_i$ are of the form 
	\begin{equation}
		R_{v_0,v_i} \frac{1}{\ell+1}\sum_{k\neq i} R_{v_i,v_k}\int_{m_i} R^{\nabla,v_k}\alpha. \label{eq:permutation-group-less-terms}
	\end{equation}
	If we analyze the difference of this expression to the integral of $\alpha$ in the $v_i-$based parallel propagated frame,  we obtain
	\[\frac{1}{\ell+1}\sum_{k\neq i} R_{v_i,v_k}\int_{m_i} R^{\nabla,v_k}\alpha - \int_{m_i} R^{\nabla,v_i}\alpha =\frac{1}{\ell+1}\sum_{k\neq i} \int_{m_i}(R_{v_i,v_k} R^{\nabla,v_k} - R^{\nabla,v_i})\alpha.\]
	Similar to the discussion in Theorem~\eqref{thm:even-better-reduction-with-center-based-PPF}, we can argue that the integral of the center-based Taylor expansion yields
	\begin{align*}
		&\frac{1}{\ell+1}\sum_{k\neq i} \int_{m_i}(R_{v_i,v_k} R^{\nabla,v_k} - R^{\nabla,v_i})\alpha = \frac{1}{\ell+1}\sum_{k\neq i} (R_{v_i,v_k} R_{v_k,c_{m_i}} - R_{v_i,c_{m_i}})\alpha_{c_{m_i}}[m_i] + \mathcal{O}(h^{\ell+3})
	\end{align*}
	
	As explained in Sec.~\ref{sec:revisitCurv} , it holds 
	\begin{align}
	&(R_{v_i,v_k} R_{v_k,c_{m_i}} - R_{v_i,c_{m_i}})= \boldsymbol{\Omega^\nabla}([v_i,v_k,c_{m_i}],v_i,c_{m_i}) = \Omega^\nabla_{v_i}[v_k - v_i,c_{m_i} - v_i] + \mathcal{O}(h^{3}).
	\end{align}
	
	Therefore,
	\begin{align*}
		\frac{1}{\ell+1}\sum_{k\neq i} (R_{v_i,v_k} R_{v_k,c_{m_i}} - R_{v_i,c_{m_i}})\alpha_{c_{m_i}}[m_i]
		&= \frac{1}{\ell+1}\sum_{k\neq i} \Omega^\nabla_{v_i}[v_k - v_i,c_{m_i} - v_i]\alpha_{c_{m_i}}[m_i] + \mathcal{O}(h^{\ell+3})\\
		&=  \Omega^\nabla_{v_i}\Bigl[\frac{1}{\ell+1}\sum_{k\neq i} v_k - v_i,c_{m_i} - v_i \Bigr]\alpha_{c_{m_i}}[m_i] + \mathcal{O}(h^{\ell+3}).
	\end{align*}
	Since $\frac{1}{\ell+1}\sum_{k\neq i} v_k = c_{m_i},$ we obtain 
	\[\frac{1}{\ell+1}\sum_{k\neq i} R_{v_i,v_k}\int_{m_i} R^{\nabla,v_k}\alpha - \int_{m_i} R^{\nabla,v_i}\alpha \in \mathcal{O}(h^{\ell+3}).\]
	Therefore, we can deduce that the contribution of the face $m_i$ for the difference 
	$\bolddnab\boldsymbol{\alpha}(s,v_0) - R_{v_0,c_s}\int_{s} (d^\nabla\alpha)^{\nabla,c_s}
	$
	is, up to order $\mathcal{O}(h^{\ell+3}),$ of the form
	\begin{align*}
		&\frac{1}{\ell+2}\sum_{j = 0}^{\ell+1} R_{v_0,v_j} \int_{m_i} R^{\nabla,v_j}\alpha - R_{v_0,c_s} \int_{m_i} R^{\nabla,c_s}\alpha= \frac{1}{\ell+2}\sum_{j = 0}^{\ell+1} \int_{m_i}\Bigl(R_{v_0,v_j} R^{\nabla,v_j}- R_{v_0,c_s} R^{\nabla,c_s}\Bigr)\alpha.
	\end{align*}
	
	Integration of the Taylor expansion at $c_{m_i}$ yields
	\begin{align*}
		&\frac{1}{\ell+2}  \sum_{j = 0}^{\ell+1} R_{v_0,v_j} \int_{m_i} R^{\nabla,v_j}\alpha - R_{v_0,c_s} \int_{m_i} R^{\nabla,c_s}\alpha \\&=\frac{1}{\ell+2}\sum_{j= 0}^{\ell+1} (R_{v_0,v_j}R_{v_j,c_{m_i}}  - R_{v_0,c_s}R_{c_s,c_{m_i}})\alpha_{c_{m_i}}[m_i] +\mathcal{O}(h^{\ell+3})\\
		&= \frac{1}{\ell+2}\sum_{j = 0}^{\ell+1} \bigg(R_{v_0,v_j}R_{v_j,c_{m_i}} - R_{v_0,c_{m_i}}  - (R_{v_0,c_s}R_{c_s,c_{m_i}} - R_{v_0,c_{m_i}})\bigg)\alpha_{c_{m_i}}[m_i] + \mathcal{O}(h^{\ell+3})\\
		&= \frac{1}{\ell+2}\sum_{j = 0}^{\ell+1}\bigg( \Omega^\nabla_{v_0}[v_j - v_0,c_{m_i} - v_0]\alpha_{c_{m_i}}[m_i]\\
		 &\,\quad- \Omega^\nabla_{v_0}[c_s - v_0,c_{m_i} - v_0]\ \alpha_{c_{m_i}}[m_i]\bigg) + \mathcal{O}(h^{\ell+3})\\
		&= \Omega^\nabla_{v_0}\biggl[ \underbrace{\frac{1}{\ell+2}\sum_{j = 0}^{\ell+1} v_j}_{ = c_s} - v_0,c_{m_i} - v_0\biggr]\alpha_{c_{m_i}}[m_i]\\
		&\quad\, - \Omega^\nabla_{v_0}[c_s - v_0,c_{m_i} - v_0]\ \alpha_{c_{m_i}}[m_i] + \mathcal{O}(h^{\ell+3})\\
		&= \mathcal{O}(h^{\ell+3}).
	\end{align*} 
	Since $v_i$ was chosen arbitrarily, this result yields the claim.
\end{proof}

To conclude the convergence of $\bolddnab \bolddnab$ under refinement, we need another lemma.
\begin{lemma}
	\label{lem:relationbtwPPFs}
	Let $ \pi : E \rightarrow M$ be a vector bundle with connection $\nabla$. Let $s=[v_0,\ldots,v_{\ell+1}]\!\subset\! M$ be a region in $M$ for which there exists a diffeomorphism to a $\ell-$simplex $\sigma$ where each point $v_i$ are mapped to an associated vertex $w_i$ of $\sigma$. Let $m\!\subset\! s$ be a $\ell-$subcell containing $v_0$. In this case, it holds that
	\begin{equation}
		R_{v_0,c_m}\int_{m} \alpha^{\nabla,c_m} = R_{v_0,c_s}\int_{m} \alpha^{\nabla,c_s} + \mathcal{O}(h^{\ell+2}),
	\end{equation}
	which can be expressed intrinsically as
	\begin{equation}
		\mathcal{R}_{v_0,c_m}\ \ \  \connintidx{\hspace*{-0.4mm}\varphi_{c_m}}_{m} \alpha = \mathcal{R}_{v_0,c_s}\ \ \connintidx{\varphi_{c_s}}_{m} \alpha + \mathcal{O}(h^{\ell+2})
	\end{equation}
	where two PPFs (one from the center $c_m$ of $m$, one from the center $c_s$ of $s$) are used, see Fig.~\ref{fig:two-different-ppf-one-face}.
\end{lemma}

\begin{proof}
	It holds
	\begin{align*}
		R_{v_0,c_s}\int_{m} \alpha^{\nabla,c_s} - R_{v_0,c_m}\int_{m} \alpha^{\nabla,c_m} &= \int_{m} (R_{v_0,c_s} R^{\nabla,c_s} - R_{v_0,c_m} R^{\nabla,c_m})\alpha \\
		&= (R_{v_0,c_s} R_{c_s,c_m} - R_{v_0,c_m})\alpha_{c_m}[m] +\mathcal{O}(h^{\ell+3})\\
		&= \underbrace{\boldsymbol{\Omega^\nabla}([v_0,c_s,c_m],v_0,c_m)}_{= \mathcal{O}(h^2)}\underbrace{\alpha_{c_m}[m]}_{= \mathcal{O}(h^\ell)} = \mathcal{O}(h^{\ell+2}).\\
	\end{align*}
\end{proof}

\begin{figure}
	\centering
	\includegraphics[width = 0.7\linewidth]{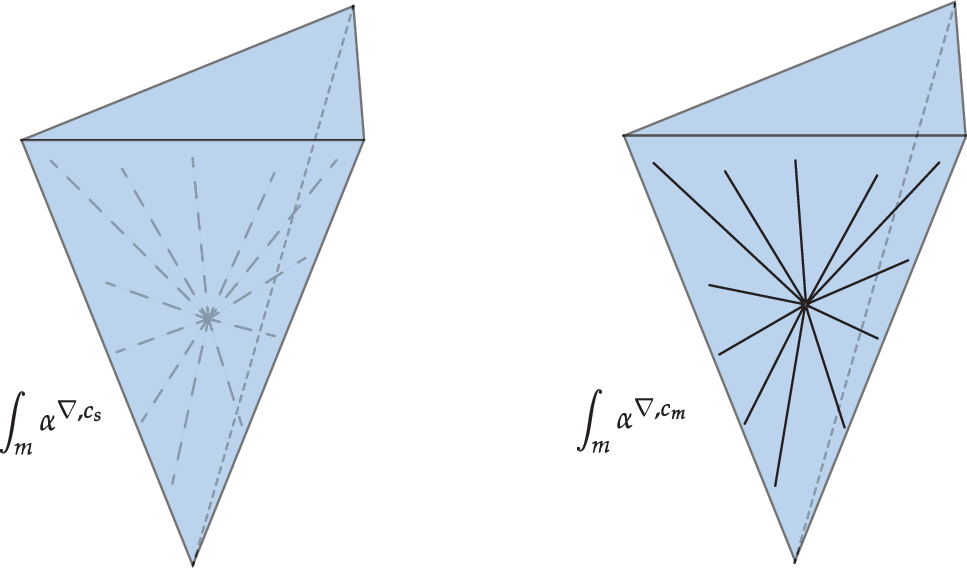}
	\caption{\textbf{Visualization} of the two parallel propagated frame fields used in Lemma~\ref{lem:relationbtwPPFs}. It is noteworthy that the integral computed on the left-hand side involves a PPF whose source is not contained within the integration domain \( m \).}
	\label{fig:two-different-ppf-one-face}
\end{figure}

\noindent This allows us to conclude the convergence of $\bolddnab\bolddnab$ under refinement, as explained next.
\begin{corollary}[Convergence of $\bolddnab\bolddnab$]
	Let $ \pi :E \rightarrow M$ be a vector bundle with connection $\nabla$. Let $s=[v_0,\ldots,v_{\ell+2}]\!\subset\! M$ be a region in $M$ for which there exists a diffeomorphism to a $\ell-$simplex $\sigma$ where each point $v_i$ are mapped to an associated vertex $w_i$ of $\sigma$. Further, let $\alpha\in \Omega^\ell(M,E)$ and $\boldsymbol{\alpha}$ be its discretization defined in Eq.~\eqref{eq:vertex-based-discretization}. We have 
	\begin{equation}
		{\bolddnab\bolddnab\boldsymbol{\alpha}(\sigma,{v_0})} = {R_{v_0,c_s}\int_{s} (d^\nabla d^\nabla \alpha)^{\nabla,c_s}} + \mathcal{O}(h^{\ell+3}).
	\end{equation}
	or equivalently,
	\begin{equation}
		{\bolddnab\bolddnab\boldsymbol{\alpha}(\sigma,{v_0})} = \mathcal{R}_{v_0,c_s}\ \ \connintidx{\varphi_{c_s}}_{s} (d^\nabla d^\nabla \alpha)+ \mathcal{O}(h^{\ell+3}).
	\end{equation}
\end{corollary}
\begin{proof}
	It holds
	\begin{align*}
		&R_{v_0,c_s}\int_{s} (d^\nabla d^\nabla \alpha)^{\nabla,c_s} = R_{v_0,c_s}\int_{\partial s} (d^\nabla\alpha)^{\nabla,c_s} + \mathcal{O}(h^{\ell+4}).
	\end{align*}
	Hence, by the previous lemma and since $d^\nabla\alpha\in \Omega^{\ell+1}(M;E)$, we obtain, 
	\[R_{v_0,c_s}\int_{s} (d^\nabla d^\nabla \alpha)^{\nabla,c_s} = \sum_{f\in \partial s} R_{v_0,c_s}\int_{m}(d^\nabla\alpha)^{\nabla,c_m} + \mathcal{O}(h^{\ell+3}).\]
	As in the previous proof, let $m_i$ be the $(\ell+1)-$subcell in $\partial s$ opposite to $v_i$ with positive orientation. Thus by Theorem~\ref{thm:alternation-augments-accuracy} we obtain 
	\[R_{v_0,c_s}\int_{s} (d^\nabla d^\nabla \alpha)^{\nabla,c_s} = R_{v_0,c_{m_0}} \int_{m_0} (d^\nabla\alpha)^{\nabla,c_{m_0}} + \sum_{i = 1}^{\ell+2} \bolddnab\boldsymbol{\alpha}(m_i,v_0) + \mathcal{O}(h^{\ell+3}).\]
	For the first term, note that 
	\[R_{v_0,c_{m_0}}\int_{m_0} (d^\nabla\alpha)^{\nabla,c_{m_0}} = R_{v_0,v_1}R_{v_1,c_{m_0}}\int_{m_0} (d^\nabla\alpha)^{\nabla,c_{m_0}} - \boldsymbol{\Omega^\nabla}([v_0,v_1,c_{m_0}],v_0,c_{m_0})\int_{m_0} (d^\nabla\alpha)^{\nabla,c_{m_0}}.\]
	Since \[\underbrace{\boldsymbol{\Omega^\nabla}([v_0,v_1,c_{m_0}],v_0,c_{m_0})}_{= \mathcal{O}(h^2)}\underbrace{\int_{m_0} (d^\nabla\alpha)^{\nabla,c_{m_0}}}_{= \mathcal{O}(h^{\ell+1})} = \mathcal{O}(h^{\ell+3})\]
	and 
	\[R_{v_0,v_1}R_{v_1,c_{m_0}}\int_{m_0} (d^\nabla\alpha)^{\nabla,c_{m_0}} = R_{v_0,v_1} \bolddnab\boldsymbol{\alpha}(m_0,v_1),\] we obtain 
    \begin{align*} 
	{R_{0,c_s}\int_{s} (d^\nabla d^\nabla \alpha)^{\nabla,c_s}} &= R_{v_0,v_1}\bolddnab \boldsymbol{\alpha}(m_0,v_1) + \sum_{j=1}^{\ell+2} \bolddnab\boldsymbol{\alpha}(m_j,v_0) + \mathcal{O}(h^{\ell+3}) \\
	&=\mathfrak{d}^{\boldsymbol{\nabla}}\bolddnab\boldsymbol{\alpha}(s,v_0) + \mathcal{O}(h^{\ell+3}).
    \end{align*} 
	Furthermore,
	\begin{align*}
		&\mathrm{Alt}^{\boldsymbol{\nabla}}\left(\mathfrak{d}^{\boldsymbol{\nabla}}\bolddnab\boldsymbol{\alpha}(s,v_0)\right) - R_{v_0,c_s}\int_{s} (d^\nabla d^\nabla \alpha)^{\nabla,c_s}\\
		&= \mathrm{Alt}^{\boldsymbol{\nabla}}\left( {R_{v_0,c_s}\int_{s}  (d^\nabla d^\nabla \alpha)^{\nabla,c_s} }\right) - R_{v_0,c_s}\int_{s}  (d^\nabla d^\nabla \alpha)^{\nabla,c_s} + \mathcal{O}(h^{\ell+3})\\
		&= \frac{1}{\ell+3}\sum_{j=0}^{\ell+2} \left(R_{v_0,v_j}R_{v_j,c_s} - R_{v_0,c_s}\right)\left(  \int_{s}  (d^\nabla d^\nabla \alpha)^{\nabla,c_s}\right)+\mathcal{O}(h^{\ell+3})\\
		&=\Omega^{\nabla}\Bigl(v_0-c_s, \underbrace{\frac{1}{\ell+3}\sum_{j=0}^{\ell+2}v_j - c_s}_{=0}\Bigr)\int_{s}  (d^\nabla d^\nabla \alpha)^{\nabla,c_s} + \mathcal{O}(h^{\ell+3}) \in \mathcal{O}(h^{\ell+3}).
	\end{align*}
	Hence $\bolddnab\bolddnab\boldsymbol{\alpha}(s,{v_0}) - R_{v_0,c_s}\int_{s} (d^\nabla d^\nabla \alpha)^{\nabla,c_s} \in \mathcal{O}(h^{\ell+3})$,
	which yields the claim.
\end{proof}
In the proof, we only show that the operators $\bolddnab\bolddnab$ and $\mathfrak{d}^{\boldsymbol{\nabla}}\bolddnab$ are of the same convergence order, although one might expect that a novel alternation could augment the accuracy of the PPF-derivative. However, for a second application, the alternation cannot compensate for the lack of accuracy from the original input. Numerical tests will confirm that a novel application of the alternation operator to  $\mathfrak{d}^{\boldsymbol{\nabla}}\bolddnab$ cannot augment the accuracy further, indicating that a numerical realization of a third application of the exterior covariant derivative will be, in this setup, impossible to achieve accurately.

\begin{remark}
	This corollary confirms the convergence, under mesh refinement, of the implicitly defined wedge product between the curvature $2-$form and a vector-valued form, as introduced in Theorem~\ref{thm:algebraic-bianchi-identity}. It converges towards the smooth curvature wedge product in Eq.~\eqref{eq:algBianchi}. Thus, the discrete algebraic Bianchi Identity holds both conceptually and numerically in the limit under mesh refinement.
\end{remark}

\begin{remark}
	As discussed in the analysis preceding Def.~\ref{def:lie-algebra-alternation}, a noteworthy reduction in the number of terms for the alternation can be used and will still ensure enhanced convergence of the sided discrete covariant exterior derivative. Eq.~\eqref{eq:permutation-group-less-terms} shows that we can reduce the number of permutations from the full alternation group down to only $(\ell+1)-$terms, resulting in a substantial decrease in the computational cost of the alternation.
	
\end{remark}

\begin{remark}[Convergence analysis for general endomorphism-valued forms]
\label{rem:convergence-1-1-tensor-valued-form}
	In the vector-valued case, we have successfully established that the discrete exterior derivative converges under refinement and satisfies the Bianchi identities. These results hold great significance within the context of our study. We showed that the differential Bianchi Identity holds in an exact sense. In the definition of the exterior covariant derivative for $(1,1)-$tensor valued forms, we used an alternation of the PPF-derivative, although the differential Bianchi Identity is already enforced for the PPF-derivative exactly. However, for the analysis of convergence under refinement of general endomorphism-valued forms, it is necessary to use an alternation of the PPF-derivative as in the vector-valued case. It turns out that the convergence analysis for arbitrary endomorphism-valued forms can be carried out similarly to the vector-valued case. We saw that the alternated PPF-derivative realizes a discrete notion of the identity \[d^{\nabla^{\mathrm{End}}}d^{\nabla^{\mathrm{End}}}\beta = [\Omega^\nabla \wedge\beta],\] for $\beta\in \Omega^\ell(M,\mathrm{End}(E))$.
	For the discretization of a general endomorphism-valued form through integration, we can, in contrast to the curvature form, not expect in general that the integration can be carried out as a boundary integral exactly. As in the case for vector-valued case, one can therefore use a discretization in a parallel-propagated frame for 
	\[d^{\nabla^{\mathrm{End}}}\beta = d\beta + \omega\wedge\beta + (-1)^{\mathrm{deg}(\beta)}\beta\wedge\omega.\]
	However, in contrast to the vector-valued case it is necessary to use a PPF-for the integration for the output \emph{and} input fiber. If we define the discretization through integration for an endomorphism-valued form through 
	\[\boldsymbol{\beta}(s,v,w) = \int_{s} \mathcal{R} ^ {\nabla ,v}\beta (\mathcal{R} ^ {\nabla ,w}) ^{-1},\]
    it is possible to show that for $\beta\in \Omega^\ell(M,\mathrm{End}(E))$ it holds 
	\[\mathfrak{d}^{\boldsymbol{\nabla}}\boldsymbol{\beta}(s,v,w) = \int_{s} R_{v,c_s} (d^{\nabla^{\mathrm{End}}}\beta)^{\nabla,\:{c_s,c_s}} R_{c_s,w} + \mathcal{O}(h^{\ell+2})\] and 
	\[\bolddnab\boldsymbol{\beta}(s,v,w) = \int_{s} R_{v,c_s} (d^{\nabla^{\mathrm{End}}}\beta)^{{\nabla,\:c_s,c_s}} R_{c_s,w} + \mathcal{O}(h^{\ell+3}),\]
	where $(d^{\nabla^{\mathrm{End}}}\beta)^{\nabla,{\: c_s,c_s}}= R^{ \nabla ,\: c_s}\ d^{\nabla^{\mathrm{End}}}\beta\  (R^{ \nabla , \: c_s}) ^{-1}$ is the representation of the form $d^{\nabla^{\mathrm{End}}}\beta$ in the frame that uses a center based PPF for the input and the output of the form.
	This allows to conclude that for a simplicial $(\ell+2)-$cell $s'$ it holds
	\[\bolddnab\bolddnab\boldsymbol{\beta}(s',v,w) = \int_{s'} R_{v,c_{s'}}(d^{\nabla^{\mathrm{End}}}d^{\nabla^{\mathrm{End}}}\beta)^{{\nabla\ c_s,c_s}} R_{c_{s'},w} + \mathcal{O}(h^{\ell+3}),\] ensuring convergence under refinement.
	\end{remark}			
\subsection{Numerical Testing}
\label{sec:numVerif}
 We now report on our tests of the numerical validity of the formulas we propose in this paper. For all examples, we start from a connection $1-$form $\omega$ and a differential form, both given in closed form, which we discretize on a sample mesh. We then compare the smooth (closed-form) exterior derivative and curvature to the results obtained by our discrete operators when applied to the discrete 1-connection and the discrete form using increasingly fine meshes, in order to observe the convergence rates of our operators. Note that our Python code for all these tests is available at \url{https://gitlab.inria.fr/geomerix/public/fulldec}. 
				
\subsubsection{Setup}
\paragraph{Computing discrete connection.} 
For a given analytical connection $1-$form $\omega$ in the standard frame of $\R^3$ and an edge $[a,b]$ (where $a,b\in \R^3$) defined as the straight path $\gamma_{ab}\colon [0,1]\to \R^3\colon t\mapsto (1-t)\cdot a + t\cdot b,$  the discrete parallel transport $R_{ab}$ induced by the connection is
\[R_{ab} = P\ \mathrm{exp}\left(\int_{\gamma}\omega\right).\]
We approximate this discrete path ordering by splitting the interval $[a,b]$ into $N$ sub-intervals. If we define the $i^\text{th}$ path segment via
$\smash{\gamma_i\colon \left[i/N,(i+1)/N\right]\to \R^3\colon t\mapsto (1-t)\cdot a + t\cdot b}$, a discrete parallel transport is computed as 
\begin{equation}
\label{eq:rotation_discretization}
R_{ab} \approx \displaystyle\prod_{i=0}^{N-1} \mathrm{exp}\left(\int_{\gamma_i}\omega \right),
\end{equation}
where the matrix exponential on the right is evaluated \emph{without} path ordering. 

\paragraph{Discretization of vector-valued forms.}
As described in Sec.~\ref{sec:issue-discretization-discrete-d}, we enforce that the discretization of the forms is carried out in a parallel-propagated frame field. That is, given a smooth $T\R^3-$valued $k-$form $\alpha$, a simplex $\sigma\subset \R^3$ with a vertex $v\in\sigma$ and a smooth connection $1-$form $\omega$, the discretization described in Eq.~\eqref{eq:vertex-based-discretization} is$$\smash{\boldsymbol{\alpha} (\sigma ,v) \coloneqq 
\int_{s} R^{\scriptscriptstyle\nabla\!,v}\alpha}.$$
Here, we assume that the underlying chosen local frame field is the standard frame of $T\R^3$, and $R^{\nabla,v}\colon \sigma\to \mathrm{SO}(3)$ is the gauge field turning this standard frame into the PPF centered at $v$.
In our numerical tests, we apply quadrature rules to approximate this integral.
For instance, for the case of a one-form represented in the standard frame of $T\R^3$ via \[\alpha = \begin{pmatrix}
    \alpha_1\\ \alpha_2 \\ \alpha_3
\end{pmatrix}\in \Omega^1(\R^3,T\R^3),\]
where $\alpha_1,\alpha_2,\alpha_3\in \Omega^1(\R^3)$ are standard scalar-valued differential forms. For an edge $\gamma:[0,1]\to \R^3,$ where $\gamma(0)=a$ and $\gamma(1)=b,$ we denote the evaluation of the smooth differential form $\alpha\in \Omega^1(M,\mathbb{R}^3)$ at a point $p\in e=\gamma([0,1])$  for the vector $b-a\in T_p\mathbb{R}^3$ through $\alpha_p(b-a)$, where we use the canonical identification induced by the standard frame of $\mathbb{R}^3$ of $T_a \mathbb{R}^3$ with $T_p \mathbb{R}^3$. Writing the gauge field $R^{\nabla,a}$ as \[R^{\nabla,a} = 
\begin{pmatrix}
	-& R_{a_1}&-\\
	-& R_{a_2}&-\\
	-& R_{a_3}&-
\end{pmatrix} 
= (R_{a_{i}})_j \in SO(3) \text{ for } (i,j)\in\{1,2,3\},\]
it holds by definition that
\begin{equation}
\boldsymbol{\alpha}([a,b],a) = \int_{\gamma} R^{\nabla,a} \alpha =\int_{[0,1]}\gamma^\ast(R^{\nabla,a})\ \gamma^\ast\alpha =  \int_{[0,1]} \begin{pmatrix}{R_{a_1}}({\gamma(t)})\cdot \alpha_{\gamma(t)}(b-a)\\ {R_{a_2}}({\gamma(t)})\cdot \alpha_{\gamma(t)}(b-a)\\ {R_{a_3}}({\gamma(t)})\cdot\alpha_{\gamma(t)}(b-a) \end{pmatrix}\ dt,   \label{eq:integral-in-ppf-discrete}
\end{equation}
where we use the standard matrix product for the multiplication of a row with a column vector.
In this case, we get the integral of the first component, for instance:
	\begin{align*}
		&\int_{\gamma} R_{a_{11}}\ \alpha_1 + R_{a_{12}}\ \alpha_2 + R_{a_{13}}\ \alpha_3\\
        &= \int_{0}^{1} R_{a_{11}}(\gamma(t))\ \alpha_{1_{\gamma(t)}}(\dot{\gamma}(t)) + R_{a_{12}}(\gamma(t))\ \alpha_{2_{\gamma(t)}}(\dot{\gamma(t)}) + R_{a_{13}}(\gamma(t))\ \alpha_{3_{\gamma(t)}}(\dot{\gamma}(t))\ dt,
	\end{align*}
 and we use a $5^\text{th}-$order Gauss-Legendre quadrature to evaluate this 1D integral numerically.
The final expression of the discretization of $\alpha$ thus reads:
$$\boldsymbol{\alpha}(e,a) \approx \sum_{i = 0}^{4} w_i\ R_{a,\gamma(t_i)}\ \alpha_{\gamma(t_i)}(b-a),$$
where $w_i$ are the usual Gauss-Legendre quadrature weights. 
                            
Similarly, if we want to discretize a smooth $(1,1)-$tensor valued $1-$form $\beta$ over $e = [a,b]$ with evaluation at $a$ and cut at $b$, we discretize the integral
    \[
      \boldsymbol{\beta}([a,b],a,b) = \int_{e} R^{\nabla,a}\ \beta\ (R^{\nabla,b})^{-1}.           
    \]
In this case, we approximate the integral through 
$$\boldsymbol{\beta}(e,a,b) \approx \sum_{i = 0}^{4} w_i\ R_{a,\gamma(x_i)}\ \beta_{\gamma(x_i)}(b-a)\ R_{\gamma(x_i),b} .$$

Now, if we are given a vector-valued $2-$form $\alpha$ and we want to calculate its discretization over a triangle  $\sigma = [a,b,c]$ evaluated at $a$, we discretize the 2D integral 
\[
\boldsymbol{\alpha}(\sigma,a) = \int_{\sigma} R^{\nabla,a}\alpha.
\]
We thus use 2D quadrature rules (provided, for instance, in~\cite{Burkardt2010}) to approximate this integral. Similarly, we use for $(1,1)-$tensor valued forms the integral 
\[
\boldsymbol{\beta}(\sigma,a,b) = \int_{\sigma} R^{\nabla,a}\beta (R^{\nabla,b})^{-1}
\]
and apply 2D quadrature to obtain the discretization of the form over the triangle.

\paragraph{Convergence plots.} For each of our examples, we perform our evaluation on a given mesh element, before scaling the element down while keeping the evaluation vertex fixed. For each scale, we calculate the associated relative error, i.e. the norm of the difference between the result of the discrete operators acting on discretized forms to the value of the integral of the smooth ground-truth result, divided by the norm of this ground-truth integral --- providing a plot showing how fast our discrete approximants converge to their continuous counterparts.

\subsubsection{Various numerical tests}

For the connection $1-$form $\omega$ given as
\begin{equation}
    \label{eq:sample-connection-test}
    \omega = \begin{pmatrix}
        0&0&1\\
        0&0&0\\
        -1&0&0
    \end{pmatrix}\ x\ dz + \begin{pmatrix}
        0&1&0\\
        -1&0&0\\
        0&0&0
    \end{pmatrix}\ (y\ dx + dz),  
\end{equation}
one can derive its torsion as
\[\Theta^\nabla = d^\nabla\theta = \begin{pmatrix}
    y\ dx\wedge dy - dy\wedge dz\\
    dx\wedge\ dz \\
    x\ dx\wedge dz
\end{pmatrix},\]
where $\theta=(dx,dy,dz)^T.$
Further, the exterior derivative of the torsion yields
\[d^\nabla\Theta^\nabla = d\Theta^\nabla + \omega\wedge\Theta^\nabla = \Omega^\nabla\wedge \begin{pmatrix}
    dx\\
    dy\\
    dz
\end{pmatrix} = \begin{pmatrix}
    0 \\
    0 \\
    -xy
\end{pmatrix}\  dx\wedge dy\wedge dz ,\] where
\[\Omega^\nabla = \begin{pmatrix}
    0& -dx\wedge dy&  dx\wedge dz\\
    dx\wedge dy&0& -xy\ dx \wedge dz\\
    -dx\wedge dz& xy\ dx\wedge dz&0
\end{pmatrix}.\]

\begin{figure}[htb] \vspace*{-3mm}
    \centering
    \includegraphics[width =0.65\linewidth]{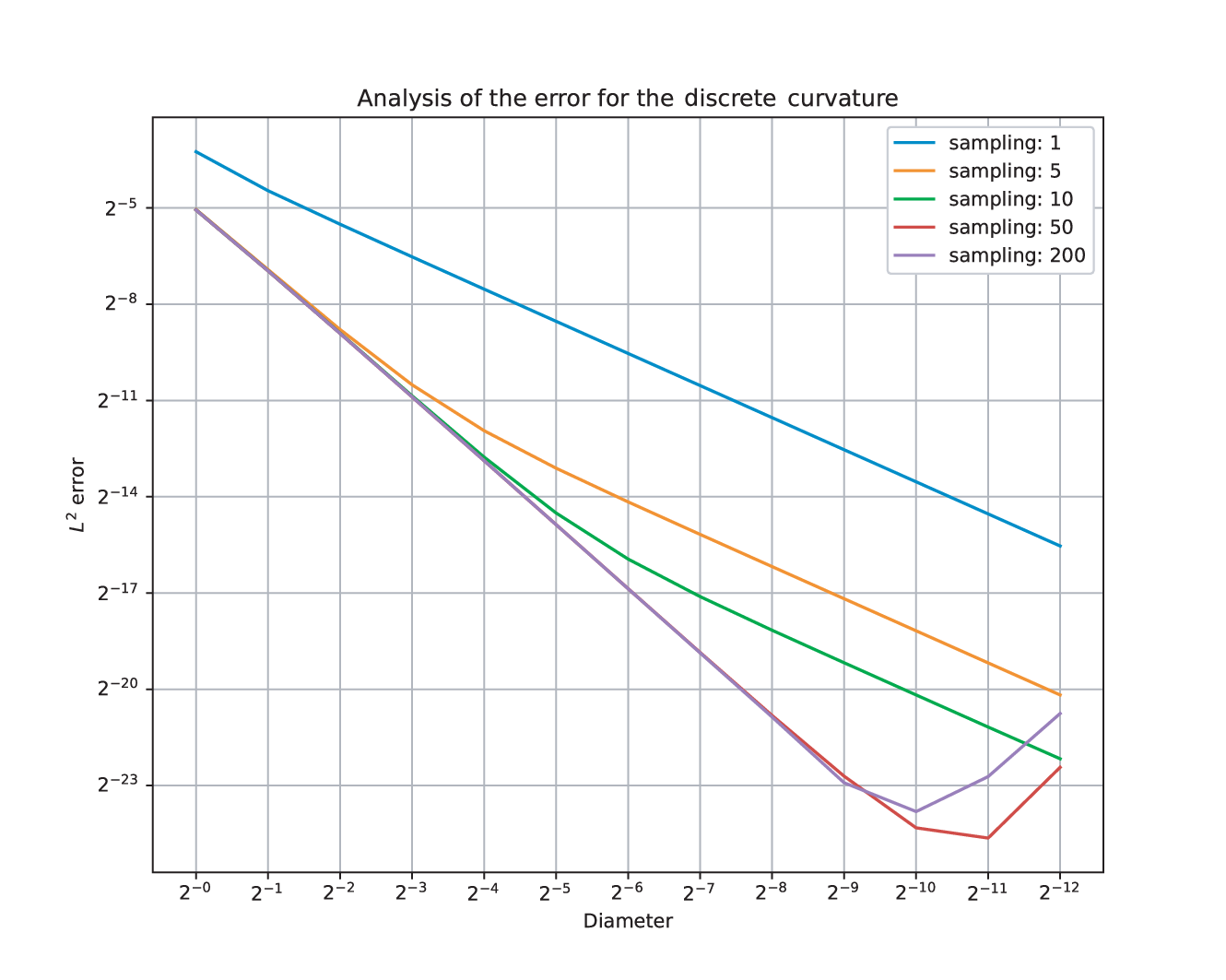} \vspace*{-5mm}
    \caption{\textbf{Convergence of discrete curvature.} The discrete curvature converges to the integrated (continuous) curvature over a center-based parallel propagated frame field; insufficient sampling for the integration of the smooth connection (Eq.~\eqref{eq:rotation_discretization}) affects the convergence rate.\label{fig:decay_curvature} \vspace*{-3mm}}
\end{figure}

\paragraph{Convergence of the Discrete Curvature.}
To verify the convergence of our discrete curvature, we construct a triangle from the three following points: $v_0 \!=\! \begin{pmatrix}
    0,0,0\\
\end{pmatrix}\smash{^T},\quad v_1 \!=\! \begin{pmatrix}
    1,0,0\\
\end{pmatrix}\smash{^T},\quad v_2 \!=\! \begin{pmatrix}
    0.4,0.3,0
\end{pmatrix}\smash{^T},$ which we subsequently shift by the vector $v \!=\! \begin{pmatrix}
    2.4, -1.3,2.9
\end{pmatrix}\smash{^T}$  and rotate using Euler angles $(30^\circ,45^\circ,27^\circ)$ around the $X-$,$Y-$,$Z-$axes to ensure arbitrariness.
We calculate the relative error of the discrete curvature compared to the integral of the smooth curvature form in a center-based PPF with pre- and post-composition to the corresponding evaluation and cut fiber. The resulting convergence plot is in Fig.~\ref{fig:decay_curvature}.

\noindent We note in passing here that the correct evaluation of the path ordering of the matrix exponential is essential to find the correct convergence order ($\mathcal{O}(h^2)$) for the discrete curvature compared to the integral of the smooth curvature in a center-based PPF. 

\begin{figure}[t]
        \centering
        \includegraphics[width =0.6\linewidth]{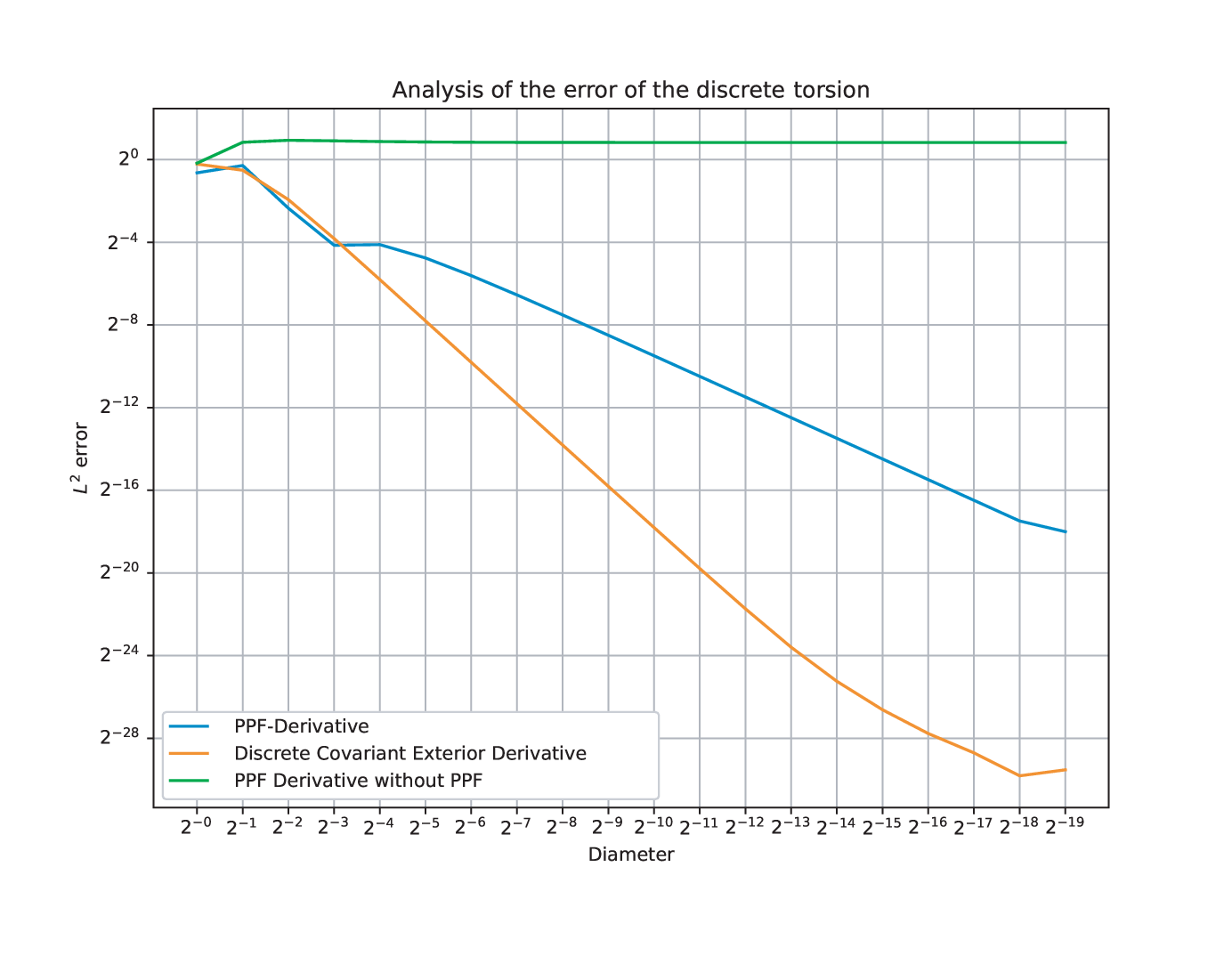}\vspace*{-7mm}
        \caption{\textbf{Convergence of discrete torsion.}
            This graph illustrates that employing a non-parallel propagated frame field for discretizing the solder form fails to yield convergence of the PPF-induced exterior covariant derivative. Utilizing a PPF for the discretization of a form and a subsequent use of the PPF derivative results in a linear decay of errors. Notably, the alternation process further enhances accuracy, leading to a quadratic decay in errors for our full discrete covariant exterior derivative. \vspace*{7mm}\label{fig:different_realizations_torsion}}
\end{figure}

\paragraph{Convergence \!of\! Discrete Covariant Derivative \!of Bundle-valued $\!1-$Forms.}
To numerically test the validity of  Theorem~\ref{thm:alternation-augments-accuracy}, we consider the triangle $[v_0,v_1,v_2]$ that we used above. We shift the triangle by a vector $v = (1.0,4.8,-2.9)\smash{^T}$ and use Euler angles $(30^\circ,25^\circ,10^\circ)$ around the $X-$,$Y-$,$Z-$axes to ensure arbitrariness.
We calculate the integral of the smooth torsion in a barycenter based PPF, transported to $v_0$. Subsequently we compute the sided PPF-induced discrete covariant exterior derivative $\boldsymbol{\mathfrak{d}^\nabla \theta}$ and compute the relative error to the smooth ground truth (blue). Further, we calculate the discrete covariant exterior derivative and calculate the relative error with the smooth ground-truth integral (orange). In addition, we illustrate that the discrete covariant exterior derivative operator presented in \cite{Hirani_Bianchi} does not converge under refinement if the input is obtained through discretization without the usage of a PPF (green).
In Fig.~\ref{fig:different_realizations_torsion}, the size of the triangle is halved in each refinement step. As our theorems stated, we observe that the decay of the relative error for the PPF-derivative is indeed linear, with a quadratic convergence order when alternation is used, as can be seen in the different slopes in the logarithmic scale of the error plots.
Additionally, the tests show that it is essential to use a parallel-propagated frame for the discretization of the differential forms, otherwise convergence cannot be achieved. \\

\noindent
Since the case of the solder form is special, we also illustrate the convergence behavior proven in Theorems~\ref{thm:ppf-derivative-linear-decay} and~\ref{thm:alternation-augments-accuracy} for a more general differential form, which we pick to be:
\begin{equation}
    \label{eq:sample-differential-$1-$form}
    \alpha = \begin{pmatrix} 2x\ dy\\ x\ dx \\ dz - z\ dy\end{pmatrix},
\end{equation}
with the connection introduced in Eq.~\eqref{eq:sample-connection-test}. In this  test,
the same test triangle is shifted by a vector $v = (2.4,-1.3,2.9)$ and is rotated using Euler angles $(30^\circ,45^\circ,27^\circ)$ for $X-$,$Y-$,$Z-$axes. The triangle size is each time reduced by a factor two, leading to the convergence plot in Fig.~\ref{fig:both-notions-converging-covariant-faster}: the alternation operator indeed improves the convergence order of the PPF-derivative. 
\begin{figure}[htb] \vspace*{-3mm}
    \centering
    \includegraphics[width = 0.6\linewidth]{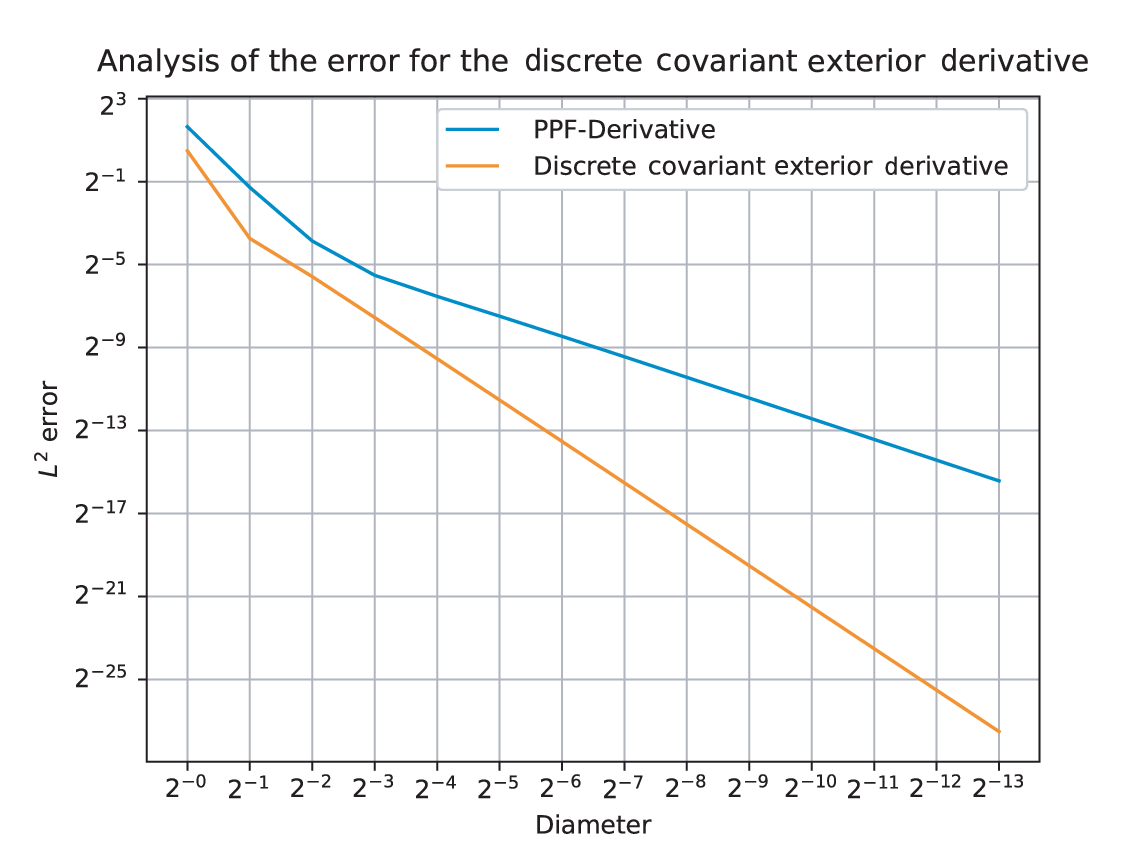}\vspace*{-2mm}
    \caption{\textbf{Discrete covariant exterior derivative for vector-valued $1-$forms.} Comparison of the decay of the relative error for the PPF-derivative and the exterior covariant derivative for the $1-$form $\alpha$ given in Eq.~\eqref{eq:sample-differential-$1-$form}. A different slope in the logarithmic plot indicates a different convergence order (here, linear in blue vs. quadratic in orange).\vspace*{-5mm}}
    \label{fig:both-notions-converging-covariant-faster}
\end{figure}

\paragraph{Convergence \!of Discrete Covariant Derivative \!of Bundle-valued $\!2-$Forms.\!}
In order to check Theorems ~\ref{thm:ppf-derivative-linear-decay}-\ref{thm:alternation-augments-accuracy} for vector-valued $2-$forms as well, we construct a tetrahedron with points 
$p_0 \!=\! \smash{\begin{pmatrix}0,0,0\end{pmatrix}\!\vphantom{a}^T}$, 
$p_1 \!=\! \smash{\begin{pmatrix}\mathrm{sin}(\frac{5\pi}{3}),\mathrm{cos}(\frac{5\pi}{3}),1\end{pmatrix}\!\vphantom{a}^T}$,
\smash{$p_2 \!=\! \begin{pmatrix}\mathrm{sin}(\frac{\pi}{3}),\mathrm{cos}(\frac{\pi}{3}),1\end{pmatrix}\!\vphantom{a}^T$},
\smash{$p_3 \!=\! \begin{pmatrix}\mathrm{sin}(\pi),\mathrm{cos}(\pi),1\end{pmatrix}\!\vphantom{a}^T$}.
We use Euler angles $(50^\circ,15^\circ,70^\circ)$ to rotate the tetrahedron, and shift the tetrahedron by $v \!=\! (3.4,-1.8,3.9)$ for arbitrariness.
For the vector-valued differential $2-$form 
\begin{equation}
    \label{eq:sample-$2-$form}
    \alpha = \begin{pmatrix}
        3y\ dy\wedge dz + x^2 dz\wedge dx\\
        (xyz + 1)\ dx\wedge dy\\
        xy\ dy\wedge dz
    \end{pmatrix}
\end{equation}
and the connection introduced in Eq.~\eqref{eq:sample-connection-test}, we illustrate convergence under mesh refinements for the discrete PPF-induced derivative and the discrete covariant exterior derivative in Fig.~\ref{fig:decay_d_nabla_two_form}. 

\begin{figure}[htb]\vspace*{-3mm}
    \centering
    \includegraphics[width = 0.55\linewidth]{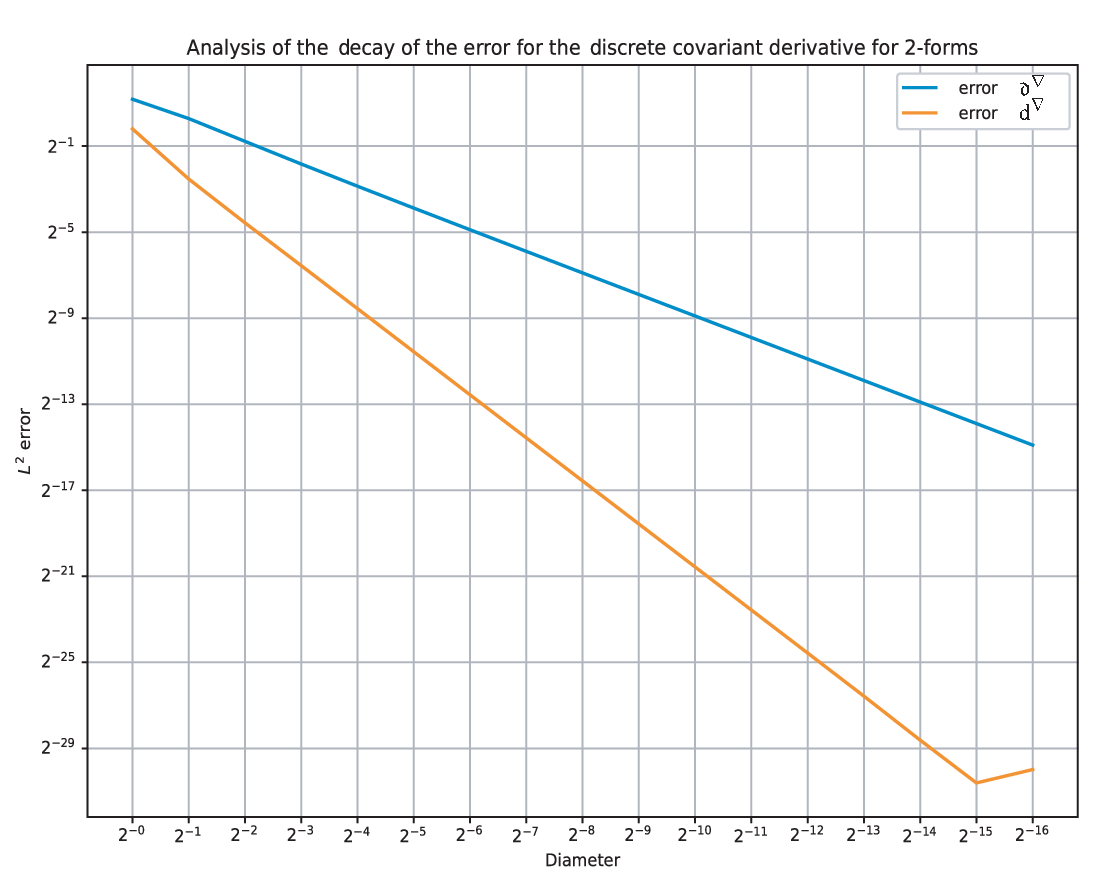}\vspace*{-2mm}
    \caption{\textbf{Discrete covariant exterior derivative for vector-valued $2-$forms.} The alternation operator for discrete vector valued $2-$forms yields, as in the case for $1-$forms, an improved accuracy under refinement compared to the PPF-induced discrete covariant exterior derivative (orange vs. blue curve).\vspace*{-3mm}}
    \label{fig:decay_d_nabla_two_form}
\end{figure}

\paragraph{Algebraic Bianchi for Vector-valued Forms}
To check the convergence of the discrete algebraic Bianchi identity to its continuous counterpart, we consider the connection introduced in Eq.~\eqref{eq:sample-connection-test}. We compare the primal discrete exterior derivative $\mathfrak{d}^{\boldsymbol{\nabla}}\bolddnab$ for the solder form  (orange) against the discrete covariant exterior derivative $\bolddnab\bolddnab$ (blue). We use the tetrahedron introduced above shifted by a vector $v = (3.4,-1.8,3.9)\smash{^T}$ 
and subsequently rotated by three matrices through the Euler angles $(50^\circ,15^\circ,70^\circ)$. 
We observe in Fig.~\ref{fig:decay_algebrac_bianchi_identity} that the convergence is of the same order. Further, we illustrate that two consecutive applications of the operator $\boldsymbol{\mathfrak{d}^\nabla}$ do not converge under refinement, even if the input form is obtained through discretization in a PPF. 

\begin{figure}[!h] \vspace*{-3mm}
    \centering
    \includegraphics[width = 0.6\linewidth]{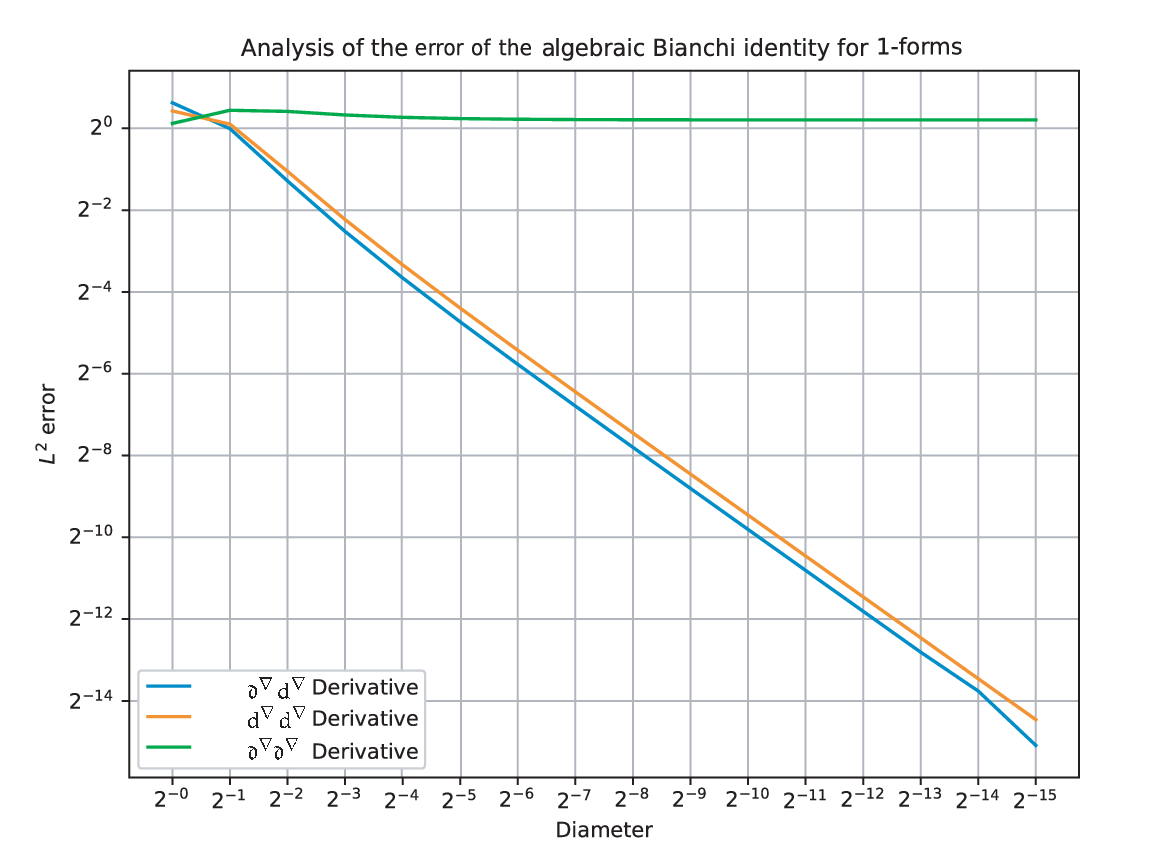}\vspace*{-2mm}
    \caption{\textbf{Algebraic Bianchi identity under refinement for vector-valued $1-$forms} Using the PPF-Derivative of the alternated torsion or the discrete exterior covariant derivative of the alternated discrete torsion both lead under refinement to the same convergence order. However, two consecutive applications of the operator $\mathfrak{d}^{\boldsymbol{\nabla}}$ without alternation in between to improve accuracy do not lead to convergence under refinement, even if the original form is discretized in a parallel-propagated frame field (green curve).}
        \label{fig:decay_algebrac_bianchi_identity}
\end{figure}\medskip

\paragraph{Discrete Covariant Exterior Derivative for $(1,1)-$Tensor-valued One-Forms}
We also illustrate the numerical behavior of the discrete covariant exterior derivative for \((1,1)\)-tensor valued forms.
Given a $(1,1)-$tensor-valued $1-$form
\begin{equation}
    \label{eq:1-1-tensor-valued-beta}
    \beta =\begin{pmatrix}
        0& -x dy& 0\\
        x dy& 0& dz\\
        0&-dz&0
    \end{pmatrix}\in \Omega^1(\R^3,\mathrm{End}(T\R^3)) 
\end{equation}
and the connection $1-$form introduced in Eq.~\eqref{eq:sample-connection-test}, one can calculate that
\[ d^{\nabla^{\mathrm{End}}}\beta = \begin{pmatrix}
    0& -\ dx\wedge dy& y\ dx\wedge dz\\
    dx\wedge dy& 0& x^2 dy\wedge dz\\
    -y\ dx\wedge dz& -x^2\ dy\wedge dz& 0
\end{pmatrix}.\]
We translate the usual triangle we used previously with a vector $v \!=\! (3.5,1.1,-1.2)^T$ and use Euler angles $(30^\circ,45^\circ,27^\circ)$ to rotate the evaluation domain.
We then calculate the discrete endomorphism-valued derivative and calculate the relative error against the smooth ground-truth value, resulting in Fig.~\ref{fig:discrete_endomorphism_valued_forms}.

\begin{figure}[htb] \vspace*{-3mm}
    \centering
    \includegraphics[width =0.55\linewidth]{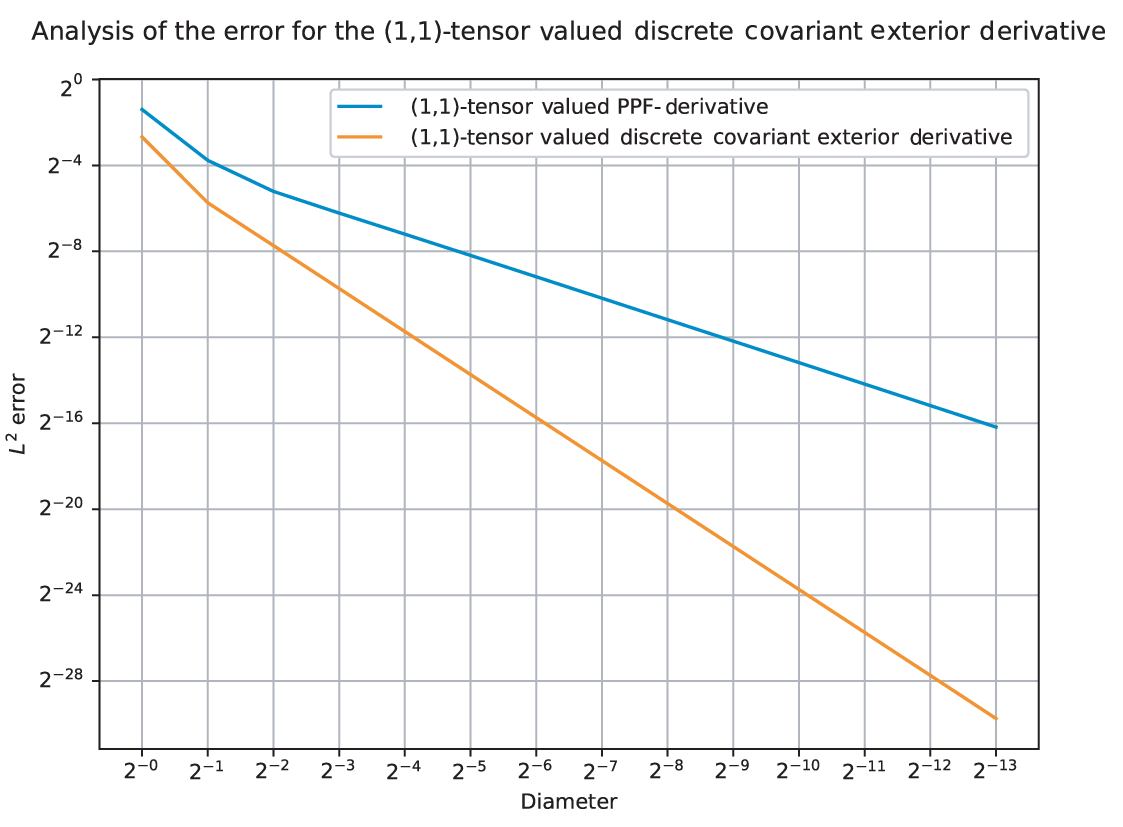}
    \vspace*{-2mm}
    \caption{\textbf{Discrete covariant exterior derivative for $(1,1)-$tensor-valued forms} As in the $(1,0)-$tensor-valued case, we observe that the decay of the relative error for the PPF-derivative for $(1,1)-$tensor-valued forms is linear, and the endomorphism-valued alternation improves the accuracy of the PPF-derivative for a faster decay of the error.\vspace*{-4mm}}
    \label{fig:discrete_endomorphism_valued_forms}
\end{figure}

\paragraph{Discrete Covariant Exterior Derivative for $(1,1)-$Tensor-valued Two-Forms.}
For the endomorphism-valued $2-$form \[\beta = \begin{pmatrix}
    0&x^2\ dy\wedge dz & 2z\ dx\wedge dy\\
    -x^2\ dy\wedge dz &0&0\\
    -2z\ dx\wedge\ dy&0&0
\end{pmatrix},\]
the exterior covariant derivative can be calculated for the connection introduced in Eq.~\eqref{eq:sample-connection-test} as 
\[d^{\nabla^{\mathrm{End}}}\beta = \begin{pmatrix}
    0&2x&2\\
    -2x&0&-2z\\
    -2&2z&0
\end{pmatrix}\ dx\wedge dy\wedge dz.\]
Translating the tetrahedron introduced earlier with $v = \begin{pmatrix} 3,3,2 \end{pmatrix}\!\smash{^T}$
and using Euler angles $(30^\circ,10^\circ,190^\circ)$ to rotate the tetrahedron, we can numerically verify in Fig.~\ref{fig:decay_d_nabla_end_two_form} the convergence behavior explained in Remark~\ref{rem:convergence-1-1-tensor-valued-form}.  

\begin{figure}[htb] \vspace*{-8mm}
    \centering
    \includegraphics[width = 0.55\linewidth]{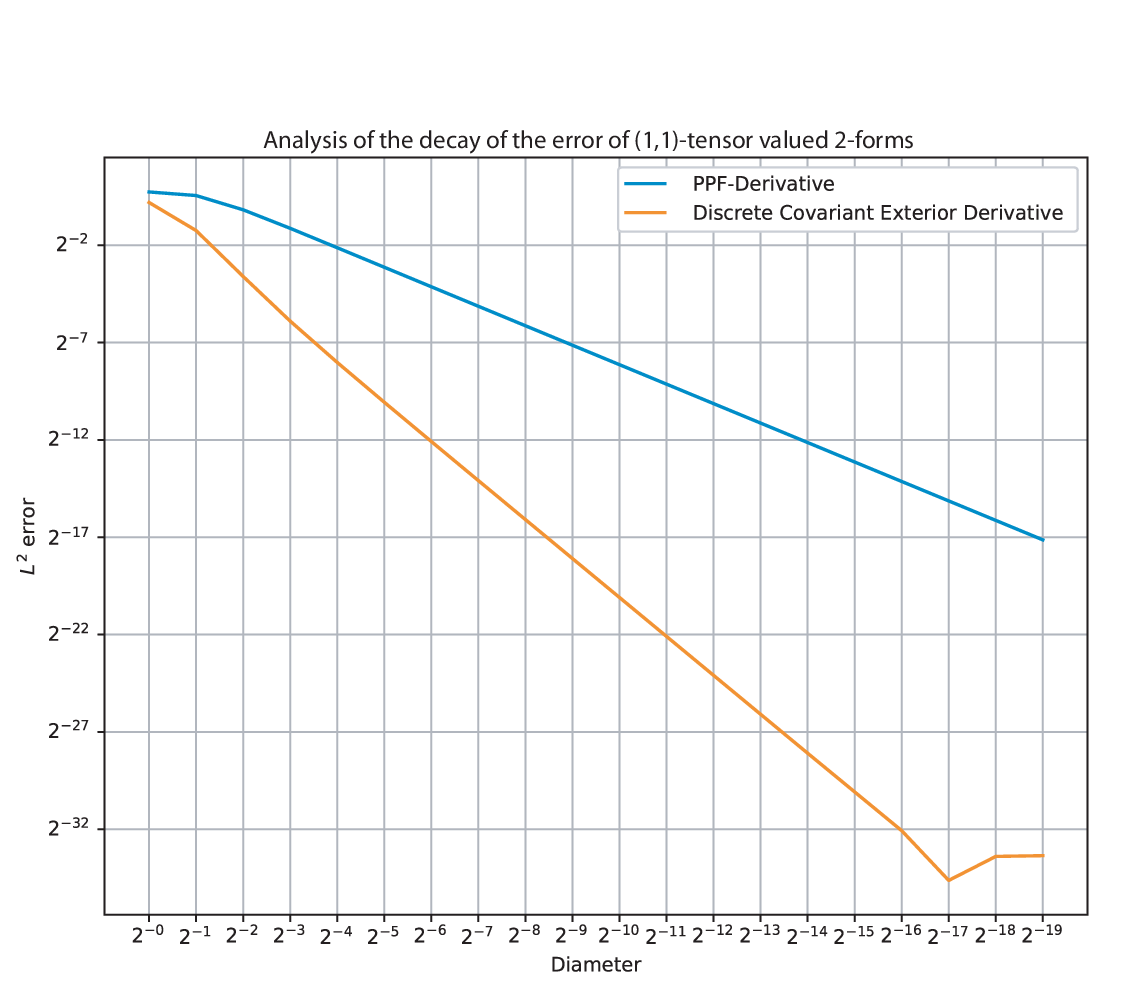}\vspace*{-2mm}
    \caption{ \textbf{Discrete covariant exterior derivative for $(1,1)-$tensor-valued $2-$forms.} As in the $(1,0)-$ tensor-valued case, we observe that the exterior covariant derivative for $(1,1)-$tensor-valued $2-$forms converges faster under refinement compared to the PPF-induced discrete covariant exterior derivative.\vspace*{-4mm}}
    \label{fig:decay_d_nabla_end_two_form}    
\end{figure}

\paragraph{Algebraic Bianchi for $(1,1)-$Tensor-valued One-Forms}
Using the aforementioned tetrahedron rotated with  Euler angles $(60^\circ,50^\circ,120^\circ)$ and shifted by a vector $v = (6.4, -3.8, 1.9)$, we obtain for the $(1,1)-$tensor valued form introduced in Eq.~\eqref{eq:1-1-tensor-valued-beta} convergence of the algebraic Bianchi identity to its continuous counterpart under refinement as seen in Fig.~\ref{fig:decay_algebrac_bianchi_identity_endom}.

\begin{figure}[t] 
    \centering
    \includegraphics[width = 0.55\linewidth]{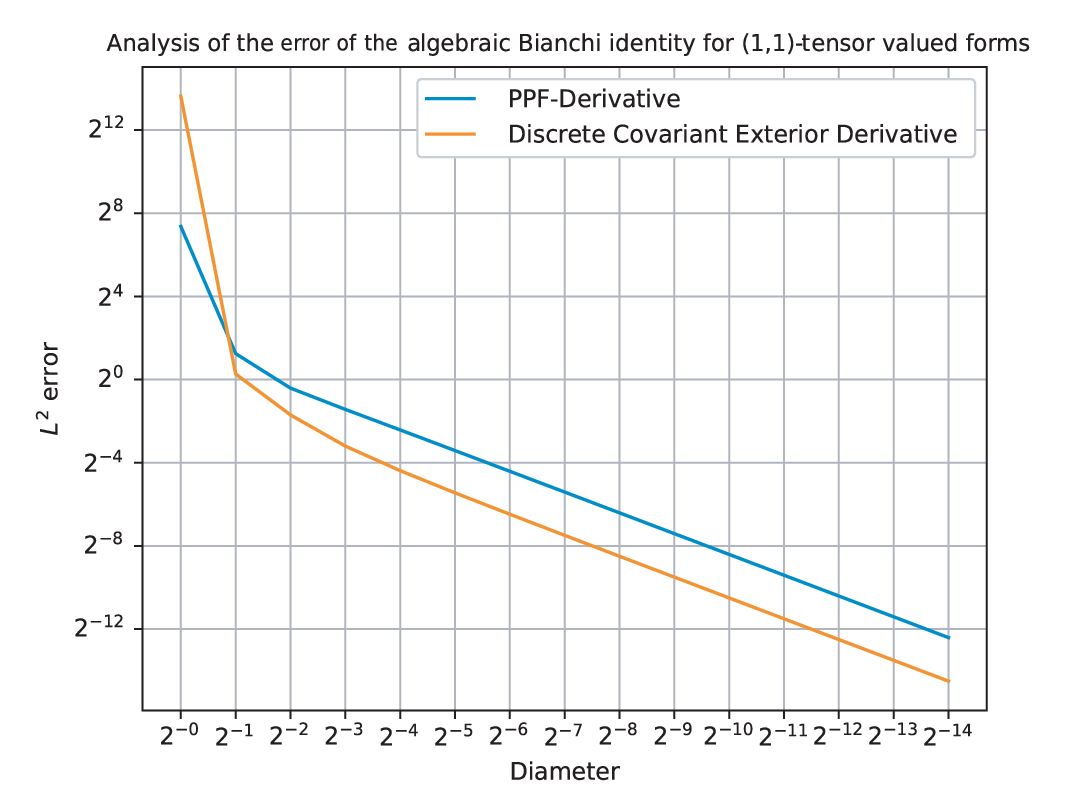}\vspace*{-3mm}
    \caption{\textbf{Algebraic Bianchi identity for $(1,1)-$tensor-valued $1-$forms.} Calculating $\mathfrak{d}^{\boldsymbol{\nabla}}\bolddnab$ and $\bolddnab\bolddnab$ under refinement yields a linear decay under mesh refinement of the relative error compared to the integral of the smooth groundtruth. As in the $(1,0)-$tensor-valued case, another alternation cannot compensate for the lack of accuracy coming from the initial input.\vspace*{-4mm}}
    \label{fig:decay_algebrac_bianchi_identity_endom}
\end{figure}
\clearpage

	\section{Conclusions and Future Work}

This work introduced a discrete bundle-valued exterior calculus based on a principled discretization of continuous forms. Extending the typical chain/cochain approach to discretize scalar-valued exterior calculus, the resulting discrete operators on finite-dimensional bundle-valued forms lead to discrete counterparts of the Bianchi identities. Unlike previous approaches, we provide theoretical proofs and numerical evidence demonstrating the \emph{convergence} of these discrete operators under mesh refinement.

By laying the groundwork for bundle-valued discretization, our work calls for many research directions. To this point, our operators have been limited to simplicial cells; however, extending the discrete covariant exterior derivative operator to more general cells through generalized barycentric coordinates should be feasible, although the notion of the alternation operator will need to be carefully defined to ensure higher-order convergence still.
Furthermore, future investigations could address issues raised in Remark~\ref{rem:disc-has-to-stay-local}, such as the characterization of alignment when multiple charts are needed to parametrize a manifold. This may lead to inquiries into Euler classes or other characteristic classes of vector bundles.

Recent efforts to define a discrete Levi-Civita connection using finite elements such as \cite{finite_element_levi_civita,analysis_curvature_approximation} have not leveraged parallel-propagated frame fields. Yet, our work demonstrates that employing a parallel-propagated frame field for the discretization of bundle-valued forms significantly improves the accuracy of the discrete covariant exterior derivative. Investigating whether our notion of parallel-propagated frame fields can be integrated into these recent approaches to enhance evaluation accuracy is another promising direction for future research. Building upon the finite element approach, several works have also focused on finite element approximations for scalar curvature \cite{gawlik2023FEM_scalar_curvature,ChristiansenDiscreteCurvature,cheeger_curvature} and the description of general relativity \cite{gawlik2023FEM_einstein_tensor,Regge1961,J_W_Barrett_1986}. A promising direction for future research is to explore whether we can define geometric and structure-preserving discretizations that are numerically accurate for these operators using our presented bundle-valued discrete exterior calculus framework as a foundation --- in particular, using the de Rham map and Whitney map for bundle-valued forms define in Section~\ref{sec:revisitingDisc} to project a bundle-valued form onto a finite-dimensional space.
Furthermore, none of these approaches proposed a notion of discrete torsion. In this work, we introduced a new concept of discrete torsion, which raises the question within the context of discrete Riemannian geometry of whether it is possible to find a discrete analog of the fundamental theorem of Riemannian geometry. i.e., the existence of a unique discrete connection that is metric-preserving and has vanishing torsion.

Finally, the recent work of \cite{ChristiansenHu2023} introduced finite element systems for vector bundles to describe a discrete elasticity complex. An intriguing direction for future research is to investigate whether our approach, which increases accuracy through the use of PPFs, can be integrated into the description of continuum mechanics. Incorporating bundle-valued forms may provide insights into improved discrete representations of phenomena such as compressible fluids and nonlinear elasticity, potentially leading to more accurate simulations in these areas.
	
	\bibliographystyle{plain}
	\bibliography{bibliography_project}

\begin{thebibliography}{10}

\bibitem{amrose_singer}
Warren Ambrose and Isadore~Manuel Singer.
\newblock A thoerem of holonomy.
\newblock {\em Transactions of the American Mathematical Society},
  75(3):428--443, 1953.

\bibitem{FEEC}
Douglas~N. Arnold.
\newblock {\em Finite {E}lement {E}xterior {C}alculus}, volume~93 of {\em
  CBMS-NSF Regional Conference Series in Applied Mathematics}.
\newblock SIAM, 2018.

\bibitem{J_W_Barrett_1986}
J~W Barrett.
\newblock The {E}instein tensor in {R}egge's discrete gravity theory.
\newblock {\em Classical and Quantum Gravity}, 3(2):203, mar 1986.

\bibitem{finite_element_levi_civita}
Yakov Berchenko-Kogan and Evan~S. Gawlik.
\newblock Finite element approximation of the levi-civita connection and its
  curvature in two dimensions.
\newblock {\em Foundations of Computational Mathematics}, 24(2):587--637, April
  2024.

\bibitem{Hirani_Bianchi}
Daniel Berwick-Evans, Anil~N. Hirani, and Mark~D. Schubel.
\newblock Discrete vector bundles with connection and the {B}ianchi identity,
  2021.

\bibitem{Bossavit_CEM}
Alain Bossavit.
\newblock {\em Computational Electromagnetism}.
\newblock Academic Press, 1998.

\bibitem{Burkardt2010}
John Burkardt.
\newblock Quadrature\_rules\_tri dataset.
\newblock
  \url{https://people.sc.fsu.edu/~jburkardt/datasets/quadrature_rules_tri/quadrature_rules_tri.html},
  2010.
\newblock Last revised on 28 September 2010.

\bibitem{cheeger_curvature}
Jeff Cheeger, Werner M{\"u}ller, and Robert Schrader.
\newblock {On the curvature of piecewise flat spaces}.
\newblock {\em Communications in Mathematical Physics}, 92(3):405 -- 454, 1984.

\bibitem{chern1999lectures}
S.~S. Chern, W.~Chen, and K.~S. Lam.
\newblock {\em Lectures on Differential Geometry}.
\newblock Series on university mathematics. World Scientific, 1999.

\bibitem{ChristiansenDiscreteCurvature}
Snorre~H Christiansen.
\newblock {On the definition of curvature in Regge calculus}.
\newblock {\em IMA Journal of Numerical Analysis}, page drad095, 01 2024.

\bibitem{ChristiansenHu2023}
Snorre~H. Christiansen and Kaibo Hu.
\newblock Finite element systems for vector bundles: Elasticity and curvature.
\newblock {\em Foundations of Computational Mathematics}, 23(2):545--596, 4
  2023.
\newblock Published on 2023/04/01.

\bibitem{DEC-DGPcourse}
Keenan Crane, Fernando de~Goes, Mathieu Desbrun, and Peter Schr\"{o}der.
\newblock Digital geometry processing with discrete exterior calculus.
\newblock In {\em ACM SIGGRAPH 2013 Courses}, 2013.

\bibitem{Crane:2010:TCD}
Keenan Crane, Mathieu Desbrun, and Peter Schr\"{o}der.
\newblock Trivial connections on discrete surfaces.
\newblock {\em Computer Graphics Forum (SGP)}, 29(5):1525--1533, 2010.

\bibitem{DeGoes:2020}
Fernando De~Goes, Andrew Butts, and Mathieu Desbrun.
\newblock Discrete differential operators on polygonal meshes.
\newblock {\em ACM Trans. Graph.}, 39(4), 7 2020.

\bibitem{disc_differential_forms_modelling}
Mathieu Desbrun, Eva Kanso, and Yiying Tong.
\newblock {\em Discrete Differential Forms for Computational Modeling}, pages
  287--324.
\newblock Birkh{\"a}user Basel, Basel, 2008.

\bibitem{frankel_2011}
Theodore Frankel.
\newblock {\em The Geometry of Physics: An Introduction}.
\newblock Cambridge University Press, 3 edition, 2011.

\bibitem{gawlik2023FEM_scalar_curvature}
Evan~S. Gawlik and Michael Neunteufel.
\newblock Finite element approximation of scalar curvature in arbitrary
  dimension, 2023.

\bibitem{gawlik2023FEM_einstein_tensor}
Evan~S. Gawlik and Michael Neunteufel.
\newblock Finite element approximation of the {E}instein tensor, 2023.

\bibitem{analysis_curvature_approximation}
Jay Gopalakrishnan, Michael Neunteufel, Joachim Sch\"oberl, and Max Wardetzky.
\newblock Analysis of curvature approximations via covariant curl and
  incompatibility for {Regge} metrics.
\newblock {\em The SMAI Journal of computational mathematics}, 9:151--195,
  2023.

\bibitem{NCAlgebra}
J.~William Helton, Mauricío De~Oliveira, and Mark Stankus.
\newblock {NCAlgebra} --- non commutative algebra package for mathematica,
  2010.

\bibitem{Hirani2003}
Anil~N. Hirani.
\newblock {\em Discrete Exterior Calculus}.
\newblock PhD thesis, California Institute of Technology, 2003.

\bibitem{Kock1997}
Anders Kock.
\newblock Combinatorics of curvature, and the {B}ianchi identity.
\newblock {\em Theory and Applications of Categories [electronic only]},
  2:69--89, 1997.

\bibitem{kock_2006}
Anders Kock.
\newblock {\em Synthetic Differential Geometry}.
\newblock London Mathematical Society Lecture Note Series. Cambridge University
  Press, 2 edition, 2006.

\bibitem{Kotiuga:Limits}
P.~Robert Kotiuga.
\newblock Theoretical limitations of discrete exterior calculus in the context
  of computational electromagnetics.
\newblock {\em IEEE Transactions on Magnetics}, 44(6):1162--1165, 2008.

\bibitem{marsh2018mathematics}
A.~Marsh.
\newblock {\em Mathematics for Physics: An Illustrated Handbook}.
\newblock World Scientific Publishing Company Pte. Limited, 2018.

\bibitem{milnor1974characteristic}
J.W. Milnor and J.D. Stasheff.
\newblock {\em Characteristic Classes}.
\newblock Annals of mathematics studies. Princeton University Press, 1974.

\bibitem{Regge1961}
T.~Regge.
\newblock General relativity without coordinates.
\newblock {\em Il Nuovo Cimento (1955-1965)}, 19(3):558--571, 1961.

\bibitem{schubel2018discretization}
Mark~D. Schubel.
\newblock {\em Discretization of Differential Geometry for Computational Gauge
  Theory}.
\newblock PhD thesis, University of Illinois at Urbana-Champaign, 2018.

\bibitem{Bossavit:1999}
T.~Tarhasaari, L.~Kettunen, and A.~Bossavit.
\newblock Some realizations of a discrete hodge operator: a reinterpretation of
  finite element techniques [for em field analysis].
\newblock {\em IEEE Transactions on Magnetics}, 35(3):1494--1497, 1999.

\end{thebibliography}

	\appendix
	
	\section{Proof of Theorem.~\ref{thm:algebraic-bianchi-identity}}
\label{app:proof-restricted-skew-symmetry}
\begin{proof}
	By definition,
   $$\boldsymbol{d^\nabla d^\nabla\alpha}(\sigma,v_0)= \mathrm{Alt}^{\boldsymbol{\nabla}}\ \mathfrak{d}^{\boldsymbol{\nabla}}\ \mathrm{Alt}\ \mathfrak{d}^{\boldsymbol{\nabla}}\boldsymbol{\alpha}([v_0,\ldots,v_{\ell+2}],v_0).$$
	We will show that the discrete map $\mathfrak{d}^{\boldsymbol{\nabla}}\ \mathrm{Alt}\ \mathfrak{d}^{\boldsymbol{\nabla}}\boldsymbol{\alpha}$ can be brought in the desired form. We will argue below that a novel application of the alternation operator does not violate this and will still be of the required form in Theorem.~\ref{thm:algebraic-bianchi-identity}.
	In this case, we have 
	\begin{align*}
		& \mathfrak{d}^{\boldsymbol{\nabla}}\ \mathrm{Alt}\ \mathfrak{d}^{\boldsymbol{\nabla}}\boldsymbol{\alpha}([v_0,\ldots,v_{\ell+2}],v_0)\\
  &= \mathcal{R}_{01}\ \mathrm{Alt}\ \mathfrak{d}^{\boldsymbol{\nabla}}\ \boldsymbol{\alpha}([v_1,\ldots,v_{\ell+2}],v_1) + \sum_{i=1}^{\ell+2} (-1)^i   \mathrm{Alt}\ \mathfrak{d}^{\boldsymbol{\nabla}}\ \boldsymbol{\alpha}([v_0,\ldots,\hat{v}_i,\ldots,v_{\ell+2}],v_0)\\
		&=\mathcal{R}_{01} \frac{1}{\abs{S_{\ell+2}}}\sum_{\pi\in S_{\ell+2}} \mathcal{R}_{1\pi(1)}\ \mathrm{sgn}(\pi)   \mathfrak{d}^{\boldsymbol{\nabla}}\boldsymbol{\alpha}(\pi([v_1,\ldots,v_{\ell+2}]),v_{\pi(1)})\\
		&\quad+ \sum_{i=1}^{\ell+1}(-1)^i\frac{1}{\abs{S_{\ell+2}}}\sum_{\pi\in S_{\ell+2}}\ \mathrm{sgn}(\pi) \mathcal{R}_{0\pi(0)}\ \mathfrak{d}^{\boldsymbol{\nabla}}\boldsymbol{\alpha}(\pi([v_0,\ldots,\hat{v}_i,\ldots,v_{\ell+2}]),v_{\pi(0)}).
	\end{align*}
	In the following, we will use the notation 
    $\sigma^i:=[v_0,\ldots,\widehat{v_i},\ldots,v_{\ell+2}]$, and with it,
	$$(\pi(\sigma^i))^j = [\pi(\sigma^i)_0,\ldots,\widehat{\pi(\sigma^i)_j},\ldots,\pi(\sigma^i)_{\ell+2}],$$
	to mean that we first omit the $i^\text{th}$ vertex, then permute the $(\ell+1)-$simplex with a permutation and then, out of the induced ordering, omit the $j^\text{th}$ vertex to obtain an $\ell-$simplex.
	This yields
	\scriptsize
	\begin{align}
		& \mathfrak{d}^{\boldsymbol{\nabla}}\ \mathrm{Alt}\ \mathfrak{d}^{\boldsymbol{\nabla}}\boldsymbol{\alpha}([v_0,\ldots,v_{\ell+2}],v_0)=\notag\\&\mathcal{R}_{01} \frac{1}{\abs{S_{\ell+2}}}\sum_{\pi\in S_{\ell+2}} \mathrm{sgn}(\pi) \mathcal{R}_{1\pi(\sigma)_1} \left(
		\mathcal{R}_{\pi(s)_1\pi(\sigma^0)_1} \boldsymbol{\alpha}(\pi(\sigma^0)^0,v_{\pi(2)}) + \sum_{j=1}^{\ell+1}(-1)^j \boldsymbol{\alpha}(\pi(\sigma^0)^j, v_{\pi(\sigma)_1}) \right)\label{eq:curvature-form-wedge-product}\\
		&+\sum_{i=1}^{\ell+2}\frac{(-1)^i}{\abs{S_{\ell+2}}}\sum_{\pi\in S_{\ell+2}}\  \mathrm{sgn}(\pi) \mathcal{R}_{0\pi(\sigma)_0}\ \left( \mathcal{R}_{\pi(0),\pi(\sigma^i)_1}\boldsymbol{\alpha}(\pi(\sigma^i)^0,v_{\pi(\sigma^i)_1}) + \sum_{j=1}^{\ell+1}(-1)^j \boldsymbol{\alpha}(\pi(\sigma^i)^j,v_{\pi(\sigma)_0}) \right).\notag
	\end{align}
	\normalsize
	Each of the parts of the sum consists of a sign, the evaluation of the differential form $\boldsymbol{\alpha}$ over a permuted subsimplex with two missing vertices and a parallel transport map to the common evaluation fiber. It thus remains to show that for a given permutation $\tau\!\in\! S_{\ell+1}$, the subsimplex $\tau(\sigma^{ij})$ is used twice for the evaluation of $\boldsymbol{\alpha}$, but with a different sign. In this case, in Eq.~\eqref{eq:curvature-form-wedge-product}, the evaluation of $\boldsymbol{\alpha}$ on $\tau(\sigma^{ij})$ appears twice, with two potentially different parallel transport paths to $v_0$ and with opposite sign, leading to a discrete curvature expression to which we apply the evaluation of $\boldsymbol{\alpha}$ on the permuted sub-simplex.
  
    \noindent Without loss of generality we can assume that $i<j$. In order to create the permuted simplices, we distinguish two cases. For the first application $\mathfrak{d}^{\boldsymbol{\nabla}}$ we omit $v_i$, leading to a sign of $(-1)^i$.
    
    \noindent Starting with $\sigma^i$, since we assumed that $i\!<\!j$, we can now apply $(j-2)$ transpositions to bring $v_j$ in the first slot of the simplex. Next, we can apply the permutation $\tau$ on the last $(\ell+1)$ vertices. This leads to a sign of $(-1)^{j-2}\cdot\mathrm{sgn}(\tau)$. Now, for the second application of $\mathfrak{d}$, the vertex $v_j$ is in the $0^\text{th}$ slot --- thus the sign does not change. The sign of evaluation will then be $(-1)^{i + j-2 + \mathrm{sgn}(\tau)}$.
    
    \noindent On the other hand, consider the case where we omit first the $j^\text{th}$ vertex in the first application of $\mathfrak{d}$. This leads to a sign of $(-1)^j$. On the simplex $\sigma^j$, there are now $(i-1)$ transpositions needed to transport $v_i$ to the $0^\text{th}$ slot. We can now again apply the permutation $\tau$ on the last $(\ell+1)$ vertices. Subsequently, when applying $\mathfrak{d}$ a second time we will omit $v_i$ in the $0^\text{th}$ slot, leading to a total sign of $(-1)^{j + i -1 + \mathrm{sgn}(\tau)}$. 
    
	\noindent This implies that for each evaluation of $\boldsymbol{\alpha}$ over $\tau(\sigma^{ij})$, there exists an evaluation of $\boldsymbol{\alpha}$ over $\tau(\sigma^{ij})$, with a possibly different transport path to the common evaluation fiber over $v_0$, and more importantly a different sign for the parallel transport path, turning the expression of Eq.~\eqref{eq:curvature-form-wedge-product} into a sum of curvatures, to which we apply  evaluations of $\boldsymbol{\alpha}$. 
 
	\noindent Since the alternation of a discrete bundle-valued differential form takes care of consistent signing and transporting to common fibers, we can conclude that there exists a $K\subseteq C^2(\sigma)\times C^\ell(\sigma)$ such that 
	$$ \boldsymbol{d^\nabla d^\nabla\alpha}(\sigma) =\frac{1}{(\ell+3)!\cdot(\ell+2)!} \sum_{({m},\kappa)\in K} \ \boldsymbol{\Omega^\nabla}(f)\ \boldsymbol{\alpha}(\kappa) =: \boldsymbol{\Omega^\nabla}\wedge\boldsymbol{\alpha}.$$
\end{proof}
	
	\section{Simplification of non-commmutative expressions.}

In various parts of this paper, we used the NCAlgebra library \cite{NCAlgebra} to simplify expressions from our operators, where connection matrices and discrete 1-forms are non-commutative variables. A notebook is available at \url{https://gitlab.inria.fr/geomerix/public/fulldec}, structured in the following way: first, as illustrated in Fig.~\ref{fig:symbolic_connection}, we initialize the symbolic parallel transport matrices; then we show in Fig.~\ref{fig:symbolic_differential_form} how to initialize a symbolic vector-valued differential 1-form; finally, we see in Fig.~\ref{fig:symbolic_d_nabla} how the PPF-induced sided discrete covariant exterior derivative and discrete covariant exterior derivative can be implemented with simplified expressions.

\begin{figure}[!htb]
	\centering
	\includegraphics[width = 0.9\linewidth]{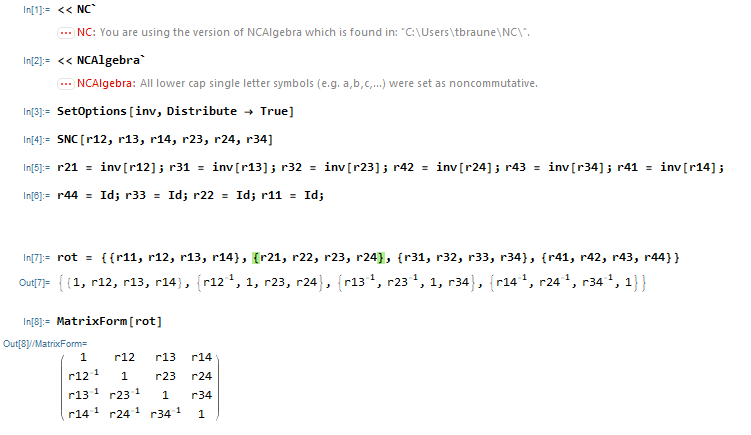}
	\caption{Definition of the symbolic variables for the connection}
	\label{fig:symbolic_connection}
\end{figure}

\begin{figure}[!htb]
	\centering
	\includegraphics[width = 1.1\linewidth]{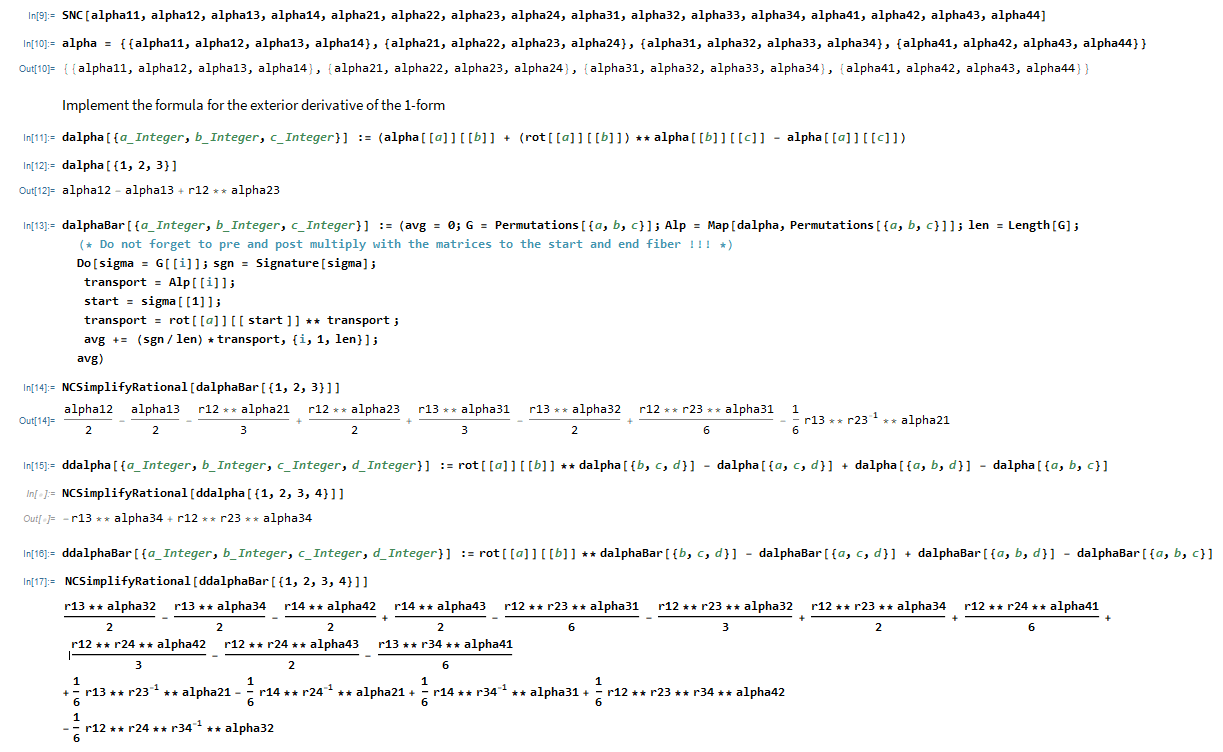 }
	\caption{Setup for the symbolic discrete 1-forms}
	\label{fig:symbolic_differential_form}
\end{figure}

\begin{figure}[!htb]
	\centering
	\includegraphics[width = 1.1\linewidth]{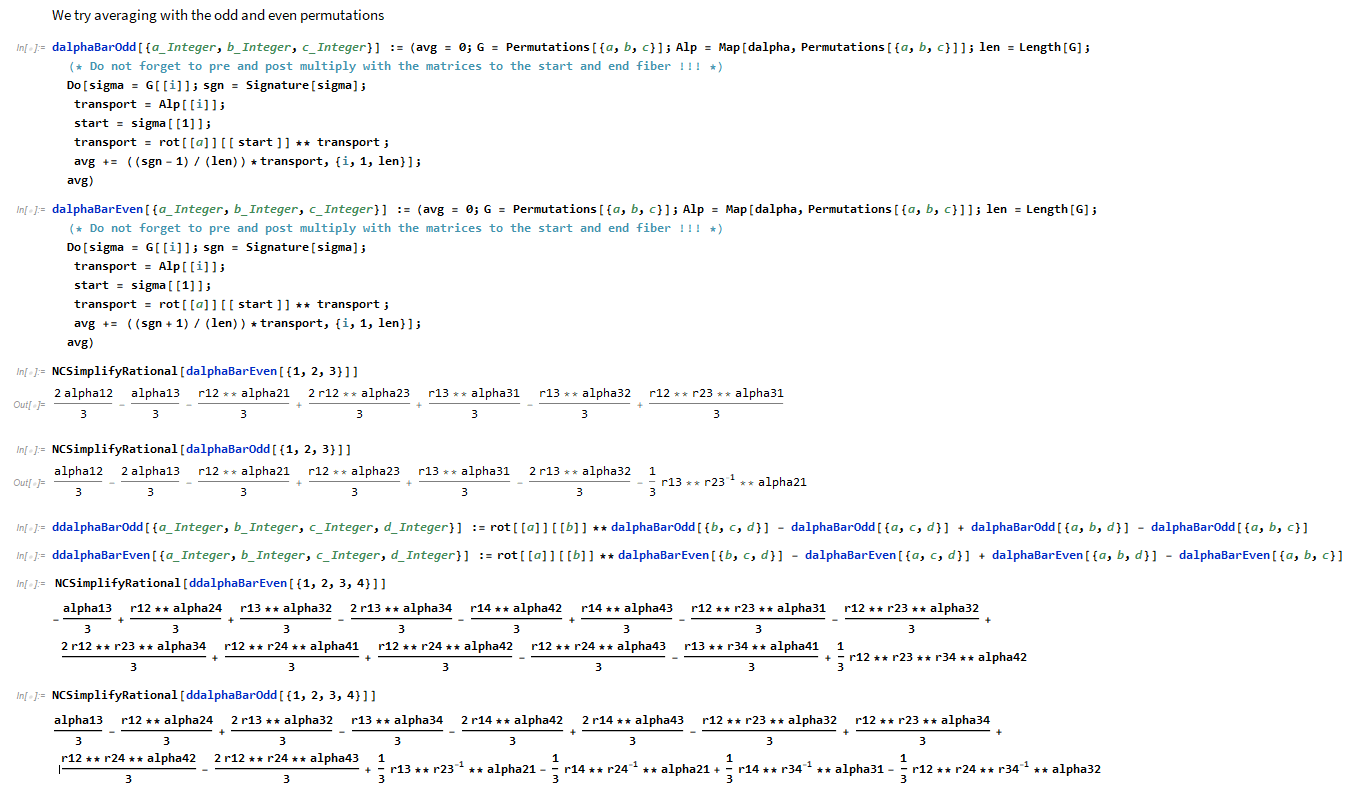}
	\caption{Implementation of the symbolic discrete covariant exterior derivative.}
	\label{fig:symbolic_d_nabla}
\end{figure}
	
	\section{Direct Calculation for the Discrete Covariant Exterior Derivative of Bundle-Valued $1-$Forms}

We want to illustrate here through a direct calculation that the discrete exterior derivative of a discretized $1-$form in the limit converges to the value of the smooth exterior derivative of the continuous $1-$form if the discretization of the differential form was carried out in a parallel-propagated frame field.\\

In the following, let $M$ be a smooth manifold and $\pi\colon E\!\to\! M$ a smooth vector bundle with connection $\nabla$ over $M$. Let $\alpha\in \Omega^1(M,E)$ be an $E-$valued differential $1-$form. Let $\sigma_2 = [abc]$ be a triangle diffeomorphic to a triangular region $s_2\!\subset\! M$. Further let $\sigma_3 = [abcd]$ an arbitrary tetrahedron diffeomorphic to a region $s_3\!\subset\! M$. As explained in Sec.~\ref{sec:convergence-analysis}, we will use the image of the simplicial cells in $M$ to carry out the discretization of $\alpha$. Once this is done, we will just regard this as discrete data available on the simplicial cells. As before, we denote by ${\boldsymbol{\alpha}}_{ab}$ the integral in a parallel-propagated frame of $\alpha$ on $[ab]$ based in $a$. We also denote by $\alpha_{ab}$ the integral in an arbitrary local frame of $\alpha$ on $[ab]$ based at $a$.

\paragraph{Convergence of the PPF-induced Sided Discrete Covariant Exterior Derivative.}
It holds for the second-order approximation of the integral in the parallel-propagated frame that ${\boldsymbol{\alpha}}_{ab}= \left(I +  \frac{\omega_{ab}}{2}\right){\alpha}_{ab} + \mathcal{O}(h^3).$ Therefore,
\begin{align*}
	\mathfrak{d}^\nabla\boldsymbol{\alpha}([abc],a) &= {\boldsymbol{\alpha}}_{ab} + R_{ab}{\boldsymbol{\alpha}}_{bc} - {\boldsymbol{\alpha}}_{ac}\\
 &=  \left(I +  \frac{\omega_{ab}}{2}\right){\alpha}_{ab} + (I+\omega_{ab})\left(I +  \frac{\omega_{bc}}{2}\right){\alpha}_{bc} - \left(I +  \frac{\omega_{ac}}{2}\right){\alpha}_{ac} + \mathcal{O}(h^3) \\
	&= ({\alpha}_{ab}+{\alpha}_{bc}-{\alpha}_{ac}) + \frac{\omega_{ab}{\alpha}_{ab}}{2} + \omega_{ab}{\alpha}_{bc}+\frac{\omega_{bc}{\alpha}_{bc}}{2} - \frac{\omega_{ac}{\alpha}_{ac}}{2}+ \mathcal{O}(h^3) .
\end{align*}
Thus, we obtain
\footnotesize
\begin{align*}
 &\mathfrak{d}^{\boldsymbol{\nabla}}\boldsymbol{\alpha}([abc],a)\\
 &= (\boldd{\alpha})_{abc} + \frac{\omega_{ab}}{2}\alpha_{ab} + \omega_{ab}(\alpha_{ac}-\alpha_{ab} )+ \frac{1}{2}(\omega_{ac}-\omega_{ab})(\alpha_{ac}-\alpha_{ab}) - \frac{\omega_{ac}}{2}\alpha_{ac} + \mathcal{O}(h^3)\\
	&=(\boldd{\alpha})_{abc} + \frac{\omega_{ab}}{2}\alpha_{ab} - \frac{\omega_{ac}}{2}\alpha_{ac} + \omega_{ab}\alpha_{ac} - \omega_{ab}\alpha_{ab} + \frac{\omega_{ac}}{2}\alpha_{ac} -\frac{\omega_{ac}}{2}\alpha_{ab} - \frac{\omega_{ab}}{2}\alpha_{ac} + \frac{\omega_{ab}}{2}\alpha_{ab}\\
	&= (\boldd{\alpha})_{abc}+\frac{\omega_{ab}\alpha_{ac}}{2}  -\frac{\omega_{ac}\alpha_{ab}}{2}+ \mathcal{O}(h^3)\\
	&= (\boldd\alpha)_{abc} + (\omega\wedge\alpha)(e_{ab},e_{ac})+ \mathcal{O}(h^3).
\end{align*}
\normalsize
In the notation \((\omega\wedge\alpha)(e_{ab},e_{ac}) \), we refer to the discretization on edges \([ab]\) and \([ac]\), respectively. This notation highlights the link to the smooth wedge product: it is as if we evaluate the smooth wedge product $\omega \wedge \alpha$ at the point $a$ on the outgoing vectors $e_{ab}$ and $e_{ac}$. This calculation emphasizes the crucial role of discretization in a parallel-propagated frame. Without this frame, only the term $\boldsymbol{d}\alpha$ would remain. However, with the employed parallel-propagated frame, we can demonstrate that this approach precisely yields the terms we would expect from the smooth counterpart.

\paragraph{Convergence of the Discrete Algebraic Bianchi Identity.}

It was shown in \cite{Hirani_Bianchi} that 
$$\mathfrak{d}^{\boldsymbol{\nabla}}\mathfrak{d}^{\boldsymbol{\nabla}}\boldsymbol{\alpha}([abcd],a) = \boldsymbol{\Omega^\nabla}([abc],a,c)\boldsymbol{\alpha}_{cd}.$$
Thus, two consecutive applications of the sided discrete covariant exterior derivative rely on a single stencil for curvature evaluation. As demonstrated in Sec.~\ref{sec:numVerif}, even when the form $\alpha$ is obtained through integration in a PPF, convergence to the smooth algebraic Bianchi identity under mesh refinement is not achieved. In the following, we will illustrate through a direct calculation for $1-$forms that the alternation operator provides a remedy. This operator transforms the resulting discrete expression into a form that closely resembles the smooth wedge product between the curvature form and the vector-valued $1-$form, up to higher-order terms.

\noindent We saw in our numerical tests that we can achieve convergence if we reduce the alternation $\mathrm{Alt}$ of the PPF-induced sided discrete covariant exterior derivative to a subset of the full permutation group that consists only of the even --- or the odd --- permutations $\tau_k$. In this case, the alternation becomes:
$$\mathrm{\widetilde{Alt}}(\alpha)[x_0,\ldots,x_k] = \frac{2}{\abs{S_{k}}}\sum_{\sigma\in \tau_{k}} R_{0,\sigma(0)}\cdot\alpha[x_{\sigma(0)},\ldots,x_{\sigma(k)}].$$
This yields for a $1-$form, after simplification with NCAlgebra \cite{NCAlgebra} on Mathematica,

\begin{align*}
	&\mathfrak{d}^\nabla(\mathrm{\widetilde{Alt}})\ \mathfrak{d}^\nabla \boldsymbol{\alpha}([abcd],a)\\
	&=\frac{1}{3}(\left(R_{ac} -R_{ab}R_{bc} \right)(\boldsymbol{\alpha}_{ca} + \boldsymbol{\alpha}_{cb} -{2}\boldsymbol{\alpha}_{cd})\\
	&\quad+\left(R_{ab}R_{bc}R_{cd} - R_{ad}\right)\boldsymbol{\alpha}_{db} - \left(R_{ab}R_{bd} - R_{ad}\right)\boldsymbol{\alpha}_{dc} + (R_{ab}R_{bd} - R_{ac}R_{cd})\boldsymbol{\alpha}_{da}).
\end{align*}
\normalsize

If we expand this formula, we obtain
\footnotesize
\begin{align*}
	&\mathfrak{d}^\nabla(\mathrm{\widetilde{Alt}})\ \mathfrak{d}^\nabla \boldsymbol{\alpha}([abcd],a)\\
    &= - \frac{1}{3}(\boldsymbol{\Omega}^\nabla([abc],a,c)(\boldsymbol{\alpha}_{ca} - \boldsymbol{\alpha}_{cd} + \boldsymbol{\alpha}_{cb} - \boldsymbol{\alpha}_{cd})+\boldsymbol{\Omega}^\nabla([abd],a,d)\boldsymbol{\alpha}_{dc}\\ &\quad+(R_{ab}R_{bc} - R_{ad}R_{dc})R_{cd}\boldsymbol{\alpha}_{db}
	- \boldsymbol{\Omega}^\nabla([abd],a,d)\boldsymbol{\alpha}_{dc} +(R_{ab}R_{bd} - R_{ac}R_{cd} )\boldsymbol{\alpha}_{da})\\
	&=-\frac{1}{3}\boldsymbol{\Omega}^\nabla([abc],a,c)(\boldsymbol{\alpha}_{ca} + \boldsymbol{\alpha}_{cb} - \boldsymbol{\alpha}_{cd}) + \frac{1}{3}\boldsymbol{\Omega}^\nabla,([abc],a,c) R_{cd}\boldsymbol{\alpha}_{db} - \frac{1}{3} \boldsymbol{\Omega}^\nabla([adc],a,c) R_{cd}\boldsymbol{\alpha}_{db}\\
	&\quad -\frac{1}{3} \boldsymbol{\Omega}^\nabla([abd],a,d) \boldsymbol{\alpha}_{dc} + \frac{1}{3} \boldsymbol{\Omega}^\nabla([abd],a,d) \boldsymbol{\alpha}_{da} - \frac{1}{3} \boldsymbol{\Omega}^\nabla([acd],a,d) \boldsymbol{\alpha}_{da}.
\end{align*}
\normalsize
\noindent We have seen that the discrete exterior derivative for $1-$forms (arising from discretization in a parallel propagated frame of a smooth differential form $\alpha$) generates terms that are of order $\mathcal{O}(h^2),$ therefore, they can be ignored to first order. It holds, for instance, that 
$$\boldsymbol{\alpha}_{ca}  - \boldsymbol{\alpha}_{cd} = - R_{ca}\boldsymbol{\alpha}_{ad}+ \underbrace{\mathfrak{d}^\nabla\boldsymbol{\alpha}([cad],c)}_{\in \mathcal{O}(h^2)}.$$ We obtain
\begin{align*}
	&\mathfrak{d}^\nabla(\mathrm{\widetilde{Alt}})\ \mathfrak{d}^\nabla \boldsymbol{\alpha}([abcd],a)\\
	 &= \frac{1}{3} \boldsymbol{\Omega}^\nabla([abc],a,c) R_{ca}\boldsymbol{\alpha}_{ad} + \frac{1}{3} \boldsymbol{\Omega}^\nabla([abc],a,c) R_{cd}\boldsymbol{\alpha}_{db}\\
	&\quad - \frac{1}{3}\boldsymbol{\Omega}^\nabla([adc],a,c) R_{cd}\boldsymbol{\alpha}_{db} - \frac{1}{3}\boldsymbol{\Omega}^\nabla([abd],a,d) \boldsymbol{\alpha}_{dc}\\
	&\quad +\frac{1}{3}\boldsymbol{\Omega}^\nabla([abd],a,d) \boldsymbol{\alpha}_{da} - \frac{1}{3}\boldsymbol{\Omega}^\nabla([acd],a,d) \boldsymbol{\alpha}_{da} + \mathcal{O}(h^4)\\
	&= \frac{1}{3}\boldsymbol{\Omega}^\nabla([abc],a,c) R_{ca}\boldsymbol{\alpha}_{ad} + \frac{1}{3}\boldsymbol{\Omega}^\nabla([abc],a,c) R_{cb}\boldsymbol{\alpha}_{bd} + \frac{1}{3}\boldsymbol{\Omega}^\nabla([abc],a,c) R_{cd}\boldsymbol{\alpha}_{db}\\
	&\quad- \frac{1}{3}\boldsymbol{\Omega}^\nabla([acd],a,d)(\underbrace{\boldsymbol{\alpha}_{db} - \boldsymbol{\alpha}_{da}}_{ = -R_{da}\boldsymbol{\alpha}_{ab}}) + \frac{1}{3}\boldsymbol{\Omega}^\nabla([abd],a,d) (\underbrace{\boldsymbol{\alpha}_{da} - \boldsymbol{\alpha}_{dc}}_{-R_{da}\boldsymbol{\alpha}_{ac}})+ \mathcal{O}(h^4)\\
	&= \frac{1}{3}\boldsymbol{\Omega}^\nabla([abc],a,a)\boldsymbol{\alpha}_{ad} + \frac{1}{3}\boldsymbol{\Omega}^\nabla([acd],a,a)\boldsymbol{\alpha}_{ab} - \frac{1}{3}\boldsymbol{\Omega}^\nabla([abd],a,a) \boldsymbol{\alpha}_{ac}\\
	&\quad + \frac{1}{3} \boldsymbol{\Omega}^\nabla([abc],a,c) (\underbrace{R_{cb} \boldsymbol{\alpha}_{bd} + R_{cd}\boldsymbol{\alpha}_{db}}_{-\boldsymbol{\Omega}^\nabla([cdb],c,b)\boldsymbol{\alpha}_{bd}})+ \mathcal{O}(h^4).
\end{align*}
Since the curvature form is a $2-$form, we can neglect the last term of the calculation for the convergence analysis since their product is of order $\mathcal{O}(h^4).$ We are then left with:
\footnotesize
$$\mathfrak{d}^{\boldsymbol{\nabla}}(\mathrm{\widetilde{Alt}})\ \mathfrak{d}^\nabla \boldsymbol{\alpha}([abcd],a) =\frac{1}{3}\left(\boldsymbol{\Omega}^\nabla([abc],a,a)\boldsymbol{\alpha}_{ad} + \boldsymbol{\Omega}^\nabla([acd],a,a)\boldsymbol{\alpha}_{ab} - \boldsymbol{\Omega}^\nabla([abd],a,a) \boldsymbol{\alpha}_{ac}\right) + \mathcal{O}(h^4) $$
\normalsize
\noindent If we use instead the odd permutations 
$$\mathcal{o}_k = \{\sigma\in S_n\quad \mathrm{sgn}(\sigma) = -1\}$$ 
for the alternation of the discrete form, we denote the alternation operator as:
$$ \quad \mathrm{{Alt}}^\mathcal{o}(\alpha)[x_0,\ldots,x_k] = \frac{2}{\abs{S_{k}}}\sum_{\sigma\in \mathcal{o}_k} -R_{0,\sigma(0)}\cdot\alpha[x_{\sigma(0)},\ldots,x_{\sigma(k)}].$$

\noindent Using Mathematica, NCAlgebra~\cite{NCAlgebra} simplifies this expression to
\begin{align*}
	&\mathfrak{d}^\nabla(\mathrm{{Alt}}^\mathcal{o})\ \mathfrak{d}^\nabla\boldsymbol{\alpha}([abcd],a)\\
    &= \frac{1}{3}\Bigg((R_{ad}R_{dc} - R_{ac})\boldsymbol{\alpha}_{ca} + (R_{ad}R_{db} - R_{ab})\boldsymbol{\alpha}_{bd} + (R_{ab}R_{bd}-R_{ad})\boldsymbol{\alpha}_{db}\\
	&\qquad-(R_{ab}R_{bc} - R_{ac})\boldsymbol{\alpha}_{bc} - (R_{ab}R_{bd}R_{dc} - R_{ac})\boldsymbol{\alpha}_{cb} + (R_{ab}R_{bc} - R_{ac})\boldsymbol{\alpha}_{cd}\\
	&\qquad-2(R_{ab}R_{bd} - R_{ad})\boldsymbol{\alpha}_{dc} + (R_{ac}R_{cb} - R_{ad}R_{db})\boldsymbol{\alpha}_{ba}\Bigg)\\
\end{align*} 
\begin{align*}	
    &=\frac{1}{3}\Bigg(\boldsymbol{\Omega}^\nabla([adc],{a},{c})\boldsymbol{\alpha}_{ca} + \boldsymbol{\Omega}^\nabla([adb],{a},{b})\boldsymbol{\alpha}_{bd}+\boldsymbol{\Omega}^\nabla([abd],{a},{d})\boldsymbol{\alpha}_{db}\\
	&\qquad -\boldsymbol{\Omega}^\nabla([abc],{a},{c})\boldsymbol{\alpha}_{cb} - \boldsymbol{\Omega}^\nabla([abd],{a},{d}) R_{dc}\boldsymbol{\alpha}_{cb} - \boldsymbol{\Omega}^\nabla([adc],{a},{c})\boldsymbol{\alpha}_{cb}\\
	&\qquad + \boldsymbol{\Omega}^\nabla([abc],{a},{c})\boldsymbol{\alpha}_{cd} - 2 \boldsymbol{\Omega}^\nabla([abd],{a},{d})\boldsymbol{\alpha}_{dc} + \boldsymbol{\Omega}^\nabla([acb],{a},{b})\boldsymbol{\alpha}_{ba} - \boldsymbol{\Omega}^\nabla([adb],{a},{b})\boldsymbol{\alpha}_{ba}\Bigg)\\
	&=\frac{1}{3}\Bigg(  \boldsymbol{\Omega}^\nabla([adc],{a},{c})(\boldsymbol{\alpha}_{ca} - \boldsymbol{\alpha}_{cb}) +  \boldsymbol{\Omega}^\nabla([abc],{a},{b})(\boldsymbol{\alpha}_{cd} - \boldsymbol{\alpha}_{cb} - R_{cb}\boldsymbol{\alpha}_{ba})\\
	&\qquad+ \boldsymbol{\Omega}^\nabla([abd],{a},{d})(-R_{db}\boldsymbol{\alpha}_{bd} + \boldsymbol{\alpha}_{db} - R_{dc}\boldsymbol{\alpha}_{cb} - 2\boldsymbol{\alpha}_{dc} + R_{db}\boldsymbol{\alpha}_{ba})\Bigg).
\end{align*}
We again take advantage of the fact that the product of the curvature $2-$form and $\boldd^\nabla\alpha$ is of order $\mathcal{O}(h^4)$ and therefore can be neglected for the analysis of convergence, leading to: 
\begin{align*}
	&\mathfrak{d}^\nabla(\mathrm{{Alt}}^\mathcal{o})\ \mathfrak{d}^\nabla \boldsymbol{\alpha}([abcd],a)\\
    &= -\frac{1}{3} \boldsymbol{\Omega}^\nabla([adc],{a},{c}) R_{ca}\boldsymbol{\alpha}_{ab} + \frac{1}{3} \boldsymbol{\Omega}^\nabla([abc],{a},{c})(\underbrace{\boldsymbol{\alpha}_{cd} - \boldsymbol{\alpha}_{ca}}_{-R_{cd}\boldsymbol{\alpha}_{da} + \mathcal{O}(h^2)})\\
	&\quad+\frac{1}{3} \boldsymbol{\Omega}^\nabla([abd],{a},{d})(-R_{db}\boldsymbol{\alpha}_{db} + \boldsymbol{\alpha}_{db} - R_{dc}\boldsymbol{\alpha}_{cb} - 2\boldsymbol{\alpha}_{dc} + R_{db}\boldsymbol{\alpha}_{ba}) + \mathcal{O}(h^4)\\
	&=-\frac{1}{3}\boldsymbol{\Omega}^\nabla([adc],{a},{c})R_{ca}(\boldsymbol{\alpha}_{ab}) + \frac{1}{3} \boldsymbol{\Omega}^\nabla([acb],{a},{b}) \underbrace{R_{bc}R_{cd}\boldsymbol{\alpha}_{da}}_{-R_{ba}\boldsymbol{\alpha}_{ad} + \mathcal{O}(h^4)}\\
	&\quad - \frac{1}{3} \boldsymbol{\Omega}^\nabla([abd],{a},{d})R_{da}\boldsymbol{\alpha}_{ac} + \frac{1}{3} \boldsymbol{\Omega}^\nabla([abd],{a},{d})(- \mathfrak{d}^\nabla\boldsymbol{\alpha}([dcb],d) + \mathcal{O}(h^4)\\
	& =  - \frac{1}{3} \boldsymbol{\Omega}^\nabla([adc],{a},{a})\boldsymbol{\alpha}_{ab} - \boldsymbol{\Omega}^\nabla([acb],{a},{a})\boldsymbol{\alpha}_{ad} + \boldsymbol{\Omega}^\nabla([adb],{a},{a})\boldsymbol{\alpha}_{ac} + \mathcal{O}(h^4).
\end{align*}
We can now take the mean over both subsets of the group. We denote $e_{ab}, e_{ac}, e_{ad}$ the edges coming out $a$ and by $(e_{ab},e_{ac},e_{ad})$ the tetrahedron that is formed by this edges. Any selection of two edges denotes the triangle formed by these two edges. With this notation, we have 
\begin{equation}
	\begin{aligned}
		\mathfrak{d}^{\boldsymbol{\nabla}} \bolddnab\boldsymbol{\alpha}((e_{ab},e_{ac},e_{ad}),a) = \frac{1}{6}\sum_{\sigma\in S_3}\  \mathrm{sgn}(\sigma)\boldsymbol{\Omega}^\nabla[\sigma(e_{ab}),\sigma(e_{ac})]\boldsymbol{\alpha}(\sigma(e_{ad})) + \mathcal{O}(h^4).
	\end{aligned}\label{eq:wedge_product_bianchi}
\end{equation}
This is the counterpart to the smooth wedge product between the curvature form and a vector-valued form. It highlights the importance of the alternation operator in achieving convergence under refinement. The calculation demonstrates that, with the use of the alternation operator, the form $\mathfrak{d}^{\boldsymbol{\nabla}} \bolddnab\boldsymbol{\alpha}$ can be brought to a higher order, in the form of a smooth wedge product between the curvature and a vector-valued form. Without the alternation operator, this simply is not the case.

\end{document}